\documentclass{article}

\author{Ofir Gorodetsky, Brad Rodgers} 
\title{Traces of powers of matrices over finite fields}
\date{}

\usepackage[margin=1in]{geometry}
\usepackage[all]{xy}
\usepackage{amssymb}
\usepackage{amsfonts}
\usepackage{amsthm}
\usepackage{amsmath}
\usepackage{hyperref}
\usepackage{mathtools}
\usepackage{url}
\usepackage{xcolor}

\mathtoolsset{showonlyrefs}

\newtheorem{thm}{Theorem}[section]

\newtheorem{lem}[thm]{Lemma}  
\newtheorem{proposition}[thm]{Proposition}
\newtheorem{cor}[thm]{Corollary}
\newtheorem{dfn}[thm]{Definition}

\theoremstyle{definition}

\theoremstyle{remark}
\newtheorem{remark}{Remark}

\newcommand{\ZZ}{\mathbb{Z}}
\newcommand{\FF}{\mathbb{F}}
\newcommand{\CC}{\mathbb{C}}
\newcommand{\PP}{\mathbb{P}}

\newcommand{\Tr}{\mathrm{Tr}}
\newcommand{\Mat}{\mathrm{Mat}}
\newcommand{\chpo}{\mathrm{CharPoly}}
\newcommand{\Aut}{\mathrm{Aut}}

\newcommand{\MM}{\mathcal{M}}

\renewcommand\Re{{\operatorname{Re\,}}}
\renewcommand\Im{{\operatorname{Im\,}}}

\newcommand{\glnq}{\mathrm{GL}(n,q)}
\newcommand{\gln}[1]{\mathrm{GL}(n, #1)}
\newcommand{\glnc}{\mathrm{GL}(n,\CC)}
\newcommand{\slnq}{\mathrm{SL}(n,q)}
\newcommand{\unq}{\mathrm{U}(n,q)}
\newcommand{\unc}{\mathrm{U}(n,\CC)}
\newcommand{\spnq}{\mathrm{Sp}(2n,q)}

\newcommand{\onqp}{\mathrm{O}^+(n,q)}
\newcommand{\onqm}{\mathrm{O}^-(n,q)}
\newcommand{\onqeps}{\mathrm{O}^{\epsilon}(n,q)}

\newcommand{\YY}{\mathbb{Y}}

\newcommand{\YYO}{\mathbb{Y}_{\mathrm{O}}}
\newcommand{\lambdao}{\lambda^{\pm}}

\newcommand{\sumgl}{\mathop{\sideset{}{^{gl}}\sum}}
\newcommand{\sumusr}{\mathop{\sideset{}{^{usr}}\sum}}
\newcommand{\sumsp}{\mathop{\sideset{}{^{sr,1}}\sum}}
\newcommand{\sumo}{\mathop{\sideset{}{^{sr,0}}\sum}}

\newcommand{\gtosym}{\Psi}
\newcommand{\symtog}{\Phi}

\newcommand{\disc}{\mathrm{Disc}}

\newcommand{\Addresses}{{
  \footnotesize
  \bigskip
  \footnotesize

  \textsc{Department of Mathematics, University of Oxford, UK}\par\nopagebreak
  \textit{E-mail address:} \texttt{ofir.goro@gmail.com}

  \medskip

  \textsc{Department of Mathematics and Statistics, Queen's University, Kingston, Ontario, K7L 3N6, Canada}\par\nopagebreak
  \textit{E-mail address:} \texttt{brad.rodgers@queensu.ca}

}}

\numberwithin{equation}{section}

\begin{document}

\maketitle
\abstract{Let $M$ be a random matrix chosen according to Haar measure from the unitary group $\mathrm{U}(n,\mathbb{C})$. Diaconis and Shahshahani proved that the traces of $M,M^2,\ldots,M^k$ converge in distribution to independent normal variables as $n \to \infty$, and Johansson proved that the rate of convergence is superexponential in $n$.
	
	We prove a finite field analogue of these results. Fixing a prime power $q = p^r$, we choose a matrix $M$ uniformly from the finite unitary group $\mathrm{U}(n,q)\subseteq \mathrm{GL}(n,q^2)$ and show that the traces of $\{ M^i \}_{1 \le i \le k,\, p \nmid i}$ converge to independent uniform variables in $\mathbb{F}_{q^2}$ as $n \to \infty$. Moreover we show the rate of convergence is exponential in $n^2$. We also consider the closely related problem of the rate at which characteristic polynomial of $M$ equidistributes in `short intervals' of $\mathbb{F}_{q^2}[T]$. Analogous results are also proved for the general linear, special linear, symplectic and orthogonal groups over a finite field. In the two latter families we restrict to odd characteristic.
	
	The proofs depend upon applying techniques from analytic number theory over function fields to formulas due to Fulman and others for the probability that the characteristic polynomial of a random matrix equals a given polynomial.}

\section{Introduction}

\subsection{A motivation from classical random matrix theory}
The unitary group $\unc \subseteq \glnc$ consists of the matrices
\begin{equation}
\unc = \{ g \in \glnc : g \bar{g^t} = I_n\},
\end{equation}
where $\bar{g}$ is the matrix obtained by complex conjugating the entries of $g$. Let $M$ be a random matrix chosen from Haar measure (of total mass 1) on  $\unc$. Fix a positive integer $k$ and let $\{ Z_j=X_j+i Y_j \}_{j=1}^{k}$ be a sequence of independent complex normal variables with $X_j,Y_j$ real-valued independent, mean $0$ and variance $\frac{1}{2}$ normal variables. Diaconis and Shahshahani \cite[Thm.~1]{diaconis1994} proved that as $n\rightarrow\infty$,
\begin{equation}
\label{eq:diacres}
(\Tr(M),\Tr(M^2),\ldots, \Tr(M^k)) {\longrightarrow} (\sqrt{1}Z_1,\sqrt{2}Z_2,\ldots,\sqrt{k}Z_k),
\end{equation}
where the arrow indicates convergence in distribution. In their proof, Diaconis and Shahshahani used the method of moments and the representation theory of $\unc$. They also proved similar results for the orthogonal and compact symplectic groups. 

This work formed the basis for a number of works in random matrix theory, see e.g. \cite{rains1997high,wieand2002eigenvalue,bump2002toeplitz,bump2006averages}.

Building further upon \eqref{eq:diacres} and confirming a conjecture of Diaconis, Johansson \cite{johansson1997} showed that the rate of convergence is superexponential, in particular
\begin{equation}
\label{eq:johansson}
\PP ( \Re  \Tr M^k \leq x, \Im \Tr M^k \leq y) - \PP (\Re \sqrt{k} Z_k \leq x, \Im \sqrt{k} Z_k \leq y) = O_k(e^{-c_k n \log n}),
\end{equation}
where $c_k$ is a constant that depends on $k$ only. Again, related results are proved for the orthogonal and compact symplectic groups, with a slightly slower rate of convergence. Further work of Duits and Johansson \cite{duits2010powers} and subsequently Johansson and Lambert \cite{johansson2020multivariate} investigates the extent to which the convergence of $\Tr M^k/\sqrt{k}$ to $Z_k$ is uniform in $k$.

\subsection{Equidistribution of traces}

In this note we prove a finite field analogue of these results. We first define some classical groups over finite fields; see the books Artin \cite{artin1988}, Dieudonn\'{e} \cite{dieudonne1971}, Dickson \cite{dickson1958} and Taylor \cite{taylor1992} for a more complete introduction. Fix a prime power $q=p^r$ and let $\FF_q$ denote the finite field with $q$ elements. Let $\glnq:= \{M \in \Mat(n,\FF_q) : \det(M) \neq 0 \}$ be the general linear group over $\FF_q$. The finite unitary group $\unq \subseteq \gln{q^2}$ consists of the matrices
\begin{equation}
\unq = \{ M \in \gln{q^2} : M \bar{M^t} = I_n \},
\end{equation}
where $\bar{M}$ is the matrix obtained by replacing the entries of $g$ by their $q$-th powers. 
In positive characteristic the following phenomena occurs. If $A$ is a square matrix over $\FF_q$, then
\begin{equation}
\Tr(A^p) = \Tr(A)^p.
\end{equation}
Thus, in our setting, if $M$ is a random matrix chosen uniformly from a subgroup of $\glnq$, then $(\Tr(M),\Tr(M^2),\ldots,\Tr(M^k))$ does not converge in distribution to a sequence of \emph{independent} random variables once $k \ge p$. Instead, we have the following results.

\begin{thm}\label{thm:arbexp}
Let $M \in \glnq$ be a random matrix chosen according to Haar measure. Fix a strictly increasing sequence $b_1,\ldots,b_k$ of positive integers coprime to $p$. Then for any sequence $a_1,...,a_k$ of elements of elements from $\FF_{q}$, we have
\begin{equation}\label{eq:probsbounded}
\left|\PP_{M \in \glnq}(\forall 1 \le i \le k: \Tr(M^{b_i}) = a_i) - q^{-k}\right| \le q^{-\frac{n^2}{2 b_k}}  (1+\frac{1}{q-1})^n \binom{n-1+b_k}{n}. 
\end{equation}
\end{thm}
\begin{remark}\label{re:succinct_GL}
The upper bound in \eqref{eq:probsbounded} implies the more succinct upper bound of e.g. $q^{-\tfrac{n^2}{2b_k} + 3n}$, using $\binom{a}{b} \le 2^a$ for $b_k \leq n$, and $q^{-\tfrac{n^2}{2b_k} + 3n} \ge 1$ for $b_k > n$. For small $b_k$ this result is optimal up to constants in the exponent, as $|\glnq|$ grows exponentially in $n^2$.
\end{remark}

\begin{remark}\label{re:large_equi}
Theorem \ref{thm:arbexp} shows that $\Tr M^b$ equidistributes as $n\rightarrow\infty$ as long as $b \leq C_q n$, where $C_q$ is a constant which becomes larger as $q$ increases.
\end{remark}

\begin{thm}\label{thm:arbexpun}
Let $M \in \unq$ be a random matrix chosen according to Haar measure. Fix a strictly increasing sequence $b_1,\ldots,b_k$ of positive integers coprime to $p$. Then for any sequence $a_1,...,a_k$ of elements from $\FF_{q^2}$, we have
\begin{equation}\label{eq:probsboundedun}
\left|\PP_{M \in \unq}(\forall 1 \le i \le k: \Tr(M^{b_i}) = a_i) - (q^2)^{-k}\right| \le q^{-\frac{n^2}{4b_k}} q^{\frac{n}{2}}(1+\frac{1}{q-1})^n \binom{n-1+2b_k}{n} 
\end{equation}
\end{thm}
\begin{remark}\label{re:succinct_U}
More succinctly one can replace the upper bound by e.g. $q^{-\tfrac{n^2}{4b_k} + \tfrac{9}{2}n}$, analogous to Remark \ref{re:succinct_GL}.
\end{remark}

In Theorems~\ref{thm:arbexpsln}, \ref{thm:arbexpspn}, \ref{thm:arbexp_on} we prove analogous results for the traces of matrices drawn from the other finite classical groups $\slnq$, $\spnq$, $\onqp$, $\onqm$. In the symplectic and orthogonal case we restrict to odd characteristic. The more complicated even $q$ can possibly be dealt with by methods developed by Fulman and Guranick \cite{fulman2004} and Fulman, Saxl and Tiep \cite{fulman2012}, but we do not pursue it in the current work.

\subsection{Equidistribution of characteristic polynomials}

The methods we apply to prove Theorems~\ref{thm:arbexp} and \ref{thm:arbexpun} also allow us to prove a closely related result regarding the distribution of the characteristic polynomial of a random matrix within `short intervals' of $\FF_q[T]$.

Let $\MM_q \subset \FF_q[T]$ be the collection of monic polynomials with coefficients from $\FF_q$ and $\MM_{n,q} \subset \MM_q$ be the collection of monic polynomials of degree $n$. For $M \in \glnq$, define
$$
\chpo(M) = \det(T-M),
$$
so $\chpo(M) \in \MM_{n,q}$.

For $f \in \MM_{n,q}$ and $0 \leq h \leq n-1$ define a `short interval' by
\begin{equation}\label{eq:shortintdef}
I(f;h) = \{ g \in \MM_q:\, |f-g| \leq q^h\}=\{ g \in \MM_q:\, \deg(f-g) \leq h\},
\end{equation}
where $\deg(\cdot)$ is the degree of a polynomial, and $|\cdot| = q^{\deg(\cdot)}$ (where we set $\deg(0)=-\infty, |0|=0$).
Such a notion of a short interval is common in function field arithmetic (see e.g. \cite{keating2014, bank2015}).

\begin{thm}\label{thm:apshort}
Let $M \in \glnq$ be a random matrix chosen according to Haar measure. Then for $0 \leq h < n-1$ and for $f \in \MM_{n,q}$,
\begin{equation}\label{eq:glnshort}
\left| \PP_{M \in \glnq} (\chpo(M) \in I(f;h) ) - \frac{q^{h+1}}{q^n} \right| \leq q^{-\frac{n^2}{2(n-h-1)}} (1+\frac{1}{q-1})^n \binom{2n-h-2}{n}.
\end{equation}
\end{thm}

\begin{remark}\label{re:GL_equi}
$|\MM_{n,q}| = q^n$ and for $\deg(f) = n > h$, $|I(f;h)| = q^{h+1}$ so $|I(f;h)|/|\MM_{n,q}| = q^{h+1}/q^n$. That is to say, Theorem~\ref{thm:apshort} is a discrepancy bound for the equidistribution of $\chpo(M)$ in its possible range.
\end{remark}

In Theorems~\ref{thm:apshort sln}, \ref{thm:apshortspn}, \ref{thm:apshorton} we prove analogous theorems for the characteristic polynomials of matrices drawn from other finite classical groups $\slnq, \spnq, \onqp, \onqm$.

Note that Theorem~\ref{thm:apshort} implies a fast rate of equidistribution when $h$ is of order $n$. But it does not necessarily say anything non-trivial for small $h$. In fact, it may be seen using a different method that equidistribution persists for short intervals of a much smaller size, though at a slower rate.

\begin{thm}\label{thm:glveryshort}
Let $M \in \glnq$ be a random matrix chosen according to Haar measure. Then for $0 \leq h < n-1$ and $f \in \MM_{n,q}$,
\begin{equation}\label{eq:glnveryshort}
\left| \PP_{M \in \glnq}( \chpo(M) \in I(f;h)) - \frac{q^{h+1}}{q^n}\right| \leq \frac{n-h}{q^n}.
\end{equation}
\end{thm}
\begin{remark}\label{re:equi_howsmall}
Hence the characteristic polynomial equidistributes as long as $h/\log_q(n) \rightarrow\infty$, since in this case $(n-h)/q^n = o(q^{h+1}/q^n)$. It may be possible to show that equidistribution does not occur at some smaller scales of $h$, but we have not pursued this.
\end{remark}

In Theorems~\ref{thm:slveryshort}, \ref{thm:spqveryshort}, \ref{thm:overyshort} we prove analogous theorems for the characteristic polynomials of matrices drawn from other finite classical groups $\slnq, \spnq, \onqp, \onqm$. In summary we show that in `large' short intervals the characteristic polynomial equidistributes superexponentially and for smaller short intervals the characteristic polynomial equidistributes as well, unless prevented from doing so for reasons related to relatively obvious symmetries. 

From Theorem~\ref{thm:glveryshort} we may deduce immediately, using Lemma~\ref{lem:symm} below, an analogous result for traces.
\begin{thm}\label{thm:manytraces}
Let $M \in \glnq$ be a random matrix chosen according to Haar measure. Then for $1 \le k \le n$ and for $\{a_i\}_{1 \le i \le k, \, p \nmid i} \subseteq \FF_q$, we have
\begin{equation}\label{eq:glnmanytraces}
\left| \PP_{M \in \glnq}( \forall 1 \le i \le k,\, p \nmid i: \Tr(M^i)=a_i)  - \frac{1}{q^{k-\lfloor \frac{k}{p} \rfloor}}\right| \leq \frac{(k+1)q^{\lfloor \frac{k}{p} \rfloor}}{q^n}.
\end{equation}
\end{thm}
Here the non-trivial range is $(n-k)/\log_q(n) \to \infty$.  Theorems~\ref{thm:manytraces sln}, \ref{thm:unmanytraces}, \ref{thm:spmanytraces}, \ref{thm:omanytraces} are analogous results for the other classical compact groups. 

The paper is structured so that each group is treated in a mostly independent section. A reader looking to become quickly acquainted with the main ideas is advised to focus on \S\ref{sec:glnq}; the other groups considered have proofs that are more technical but rely on related ideas.

\subsection{Related works}
\subsubsection{Fixed $n$, large $q$}
In the `large finite field' limit, namely the limit where one takes $n$ fixed and lets $q$ tend to infinity, one can easily obtain equidistribution of traces and characteristic polynomials as follows. By a result of Reiner and of Gerstenhaber, discussed in Theorem~\ref{thm:glnq_charpoly} below, we have
\begin{equation}\label{eq:pm}
\PP_{M \in \glnq} (\chpo(M) = f_0 ) =q^{-n} \left( 1 + O_n\left( \frac{1}{q} \right) \right)
\end{equation}
uniformly for any monic $f_0$ of degree $n$ with $f_0(0) \neq 0$. In particular, by summing over $f_0$ in a short interval, this implies that
\begin{equation}
\PP_{M \in \glnq} (\chpo(M) \in I(f;h) ) =\frac{q^{h+1}}{q^n} \left( 1 + O_n\left( \frac{1}{q} \right) \right)
\end{equation}
for any monic $f$ of degree $n$ and any $0 \le h < n$. Using Lemma~\ref{lem:symm} an analogous result holds for traces. A similar argument works for other classical groups.

\subsubsection{Pointwise bounds}
An explicit form of \eqref{eq:pm} is provided by Chavdarov \cite[\S3]{chavdarov1997}, who by a completely different method proved that for $f_0(0) \neq 0$,
\begin{equation}
\frac{(q-3)^{n^2-n}}{|\glnq|} \le \PP_{M \in \glnq} (\chpo(M) = f_0 ) \le \frac{(q+3)^{n^2-n}}{|\glnq|} 
\end{equation}
for $q>4$. A similar result is proved for the symplectic group. Chavdarov was after a large-$q$ result. However, for fixed $q$ and growing $n$ one can do better. Examining  Theorem~\ref{thm:glnq_charpoly}, one sees that if $f_0$ factorizes as $\prod_{i=1}^{r} P_i^{e_i}$ then
\begin{equation}
q^n \PP_{M \in \glnq} (\chpo(M) = f_0 ) = \prod_{i=1}^r  \prod_{j=1}^{e_i}\left(1 - q^{-m_i j}\right)^{-1}.
\end{equation}
An immediate consequence is a lower bound $q^n \PP_{M \in \glnq} (\chpo(M) = f_0 ) \ge 1$. This is in fact tight, as may be seen by taking $f_0$ to be a prime polynomial of growing degree. As for an upper bound, our proof of Theorem~\ref{thm:glveryshort} in fact works as is for $h=-1$, in which case $I(f_0;h)=\{f_0\}$ is a singleton and one obtains
\begin{equation}
q^n \PP_{M \in \glnq} (\chpo(M) = f_0 )  \le n+2
\end{equation}
whenever $f_0(0) \neq 0$. With more work one should be able to show that the correct upper bound is logarithmic in $n$, though we do not prove this here. This is obtained for instance by taking $f_0 = \prod_{\deg P \le m} P^{\lfloor m/\deg(P) \rfloor} $, where the product is over prime polynomials of degree up to $m$.

\subsubsection{Exponential sums over matrices}
Let $\chi$ be a (possibly trivial) multiplicative character of $\FF_q^{\times}$ and $\psi$ a non-trivial additive character of $\FF_q$. Eichler \cite{eichler1938} and later Lamprecht \cite{lamprecht1957} evaluated the generalized Gauss sum
\begin{equation}\label{eq:gausssum}
\sum_{g\in G} \chi(\det(g))\psi(\Tr(g))
\end{equation}
for $G=\glnq$, in terms of the ordinary Gauss sum. In a series of works, Kim \cite{kim19972,kim19971,kim19973,kim19981,kim19983,kim19982}, Kim and Lee \cite{kim1996} and Kim and Park \cite{kim19974} studied \eqref{eq:gausssum} for other classical groups $G$, and in the case of $G=\glnq$ also evaluated the generalized Kloosterman sum \cite[Thm.~C]{kim19983}
\begin{equation}\label{eq:kloostersum}
\sum_{g\in G} \psi(\Tr(ag+bg^{-1})), \qquad (a,b \in \FF_q^{\times}).
\end{equation}
Their methods relied on a variant of Bruhat decomposition. In \cite{chae2003} and \cite{chae2008} Chae and Kim used $\ell$-adic cohomology and Deligne-Lusztig theory, respectively, to obtain new proofs for these formulas.

In the course of proving our superexponential results, we end up evaluating exponential sums over the various classical groups which are more general than \eqref{eq:gausssum} and \eqref{eq:kloostersum}, see e.g. Theorems~\ref{thm:expo} and \ref{thm:expoun} below. Our evaluation is in terms of zeros of certain $L$-functions associated with each classical group. Additionally we provide bounds on these sums. These results may be of independent interest, in view of recent applications of the works of Kim et al. \cite{castryck2012, cojocaru2017, perret2018, perret2019}. Our proofs differ substantially from those just mentioned, and in particular do not involve the theory of reductive groups.

\subsubsection{Cohen-Lenstra heuristics}
In our approach we will study $P_{GL}(f) = \PP_{M \in \glnq}(\chpo(M) = f)$ as an arithmetic function over $\mathbb{F}_q[T]$ (see Definition~\ref{dfn:P_GL_def} below). As it ends up the Dirichlet series for $P_{GL}(f)$ bear a certain formal resemblance to Dirichlet series which have previously been investigated by Cohen and Lenstra.

Indeed, let $A$ be the ring of integers of a number field, e.g. $A = \ZZ$. The norm of an ideal $\mathfrak{a} \subseteq A$ is defined as $\mathrm{Nm}(\mathfrak{a}) =| A/\mathfrak{a}|$.  If $G$ is a finite $A$-module, let $w_A(G) = |\Aut (G)|^{-1}$, where $\Aut (G)$ is the group of $A$-homomorphisms of $G$. With any such $G$, we associate an ideal of $A$ by the following process: $G$ is necessarily isomorphic to a direct sum $\oplus_i A/\mathfrak{a_i}$ of cyclic submodules, and we set $I(G) := \prod_i \mathfrak{a}_i$; this ideal does not depend on the specific decomposition into cyclic submodules. With this notion in hand, we extend $w_A$ to ideals of $A$ as follows:
\begin{equation}
w_A(\mathfrak{a}) = \sum_{\substack{G \mbox{ up to $A$-ismorphism}\\ I(G)=\mathfrak{a}}} w_A(G).
\end{equation}
Cohen and Lenstra define the Dirichlet series $\zeta_{\infty,A}(s) = \sum_{\mathrm{a}} w_A(\mathrm{a}) \mathrm{Nm}(\mathrm{a})^{-s}$, converging for $\Re s > 0$, where the sum is over non-zero ideals of $A$. They prove that 
$$
\zeta_{\infty,A}(s) = \prod_{j \ge 1} \zeta_{A}(s+j),
$$
where $\zeta_A$ is the Dedekind zeta function of $A$ \cite[Cor.~3.7]{cohen1984} and study the analytic properties of $\zeta_{\infty,A}$ in \cite[\S7]{cohen1984}.
	
As it turns out, the arithmetic function $P_{GL}(f)$  and analogous functions for other classical groups resemble function-field analogues of $w_A(\mathfrak{a})$. In particular, for any Dirichlet character $\chi \colon \FF_q[T] \to \CC$, the Dirichlet series $\sum_{f \in \FF_q[T], \text{ monic}} P_{GL}(f) \chi(f) |f|^{-s}$ behaves similarly to $\zeta_{\infty,A}(s)$  in that we have the factorization 
$$
\sum_{f} \frac{P_{GL}(f) \chi(f)}{|f|^{s}} = \prod_{j \ge 1} L_{GL}(q^{-s-j},\chi)
$$ 
where $L_{GL}(u,\chi)$ is related to the $L$-function of $\chi$, see Theorem~\ref{thm:L_GL_expansion} below. Although Cohen-Lenstra heuristics in the context of matrices over rings were studied extensively in the literature, see e.g. \cite{friedman1989, cheong2018cohen, wood2019}, the formal connection which we note here seems to be new and we do not know if there is a deeper reason for it.

\subsubsection{Trace as a fixed-point count}
There is another perspective which leads to a different finite field analogue for results about the traces of random matrices. Given a permutation $\pi$ on $n$ elements, we may consider its permutation matrix $P_{\pi}$. The trace of $P_{\pi}$ counts the fixed points of $\pi$. Sampling $\pi$ uniformly from the symmetric groups $S_n$ and letting $n$ tend to infinity, $\Tr(P_{\pi})$ is known to tend to a Poisson distribution with parameter $1$, see e.g. the discussion in \cite[\S5]{diaconis1994}. One can similarly study the number of vectors fixed by a random matrix from $\glnq$ as $n$ tends to infinity. This was studied for $\glnq$ and the various finite classical groups by Rudvalis and Shinoda \cite{rudvalis1988}, who obtained formulas for the probability that the fixed space of a random matrix has a given dimension. Recently, Fulman and Stanton \cite{fulman2016} computed the limit of the moments of the number of fixed vectors as $n$ tends to infinity, as well as simplified and unified the results of Rudvalis and Shinoda.

\subsection*{Acknowledgments}

For discussions related to this paper we thank Gilyoung Cheong, Jason Fulman, Valeriya Kovaleva and Will Sawin. We thank the anonymous referee for their suggestions and corrections. The first author was supported by the European Research Council under the European Union’s Seventh Frame-work Programme (FP7/2007-2013) / ERC grant agreements  n$^{\text{o}}$ 320755 and 786758.  The second author was partly supported by the US NSF grant DMS-1701577 and an NSERC grant. Work for this paper was done during visits to the University of Michigan, Tel Aviv University, and the Centre de Recherches Math\'ematiques and we thank these institutions for their hospitality.

\section{Preliminary background}

\subsection{Hayes characters}\label{sec:HayesIntro}
Here we review a generalization of Dirichlet characters which were introduced by Hayes \cite{hayes1965}. Other papers with a review similar to what we give below include \cite{gorodetsky2020mean, gorodetsky2020correlation}.

\subsubsection{Equivalence relation}
Let $\ell$ be a non-negative integer and $H$ be a non-zero polynomial in $\FF_q[T]$. We define an equivalence relation $R_{\ell,H,q}$ on $\MM_q$ by saying that $A \equiv B \bmod R_{\ell,H,q}$ if and only if $A$ and $B$ have the same first $\ell$ next-to-leading coefficients and $A \equiv B \bmod H$. (We do not require that $A$ and $B$ have the same degree. For instance $T^3+T+1 \equiv T^5+T^3+T^2+1 \bmod R_{2,T,q}$ for any $q$.) We adopt the following convention throughout: the $j$-th next-to-leading coefficient of a polynomial $f(T) \in \MM_q$ with $j > \deg (f)$ is considered to be $0$. 
It may be verified that there is a well-defined quotient monoid $\MM_q/R_{\ell,H,q}$, where multiplication is the usual polynomial multiplication. An element of $\MM_q$ is invertible modulo $R_{\ell,H,q}$ if and only if it is coprime to $H$. The units of $\MM_q/R_{\ell,H,q}$ form an abelian group, having as identity element the equivalence class of the polynomial $1$. We denote this unit group by $\left( \MM_q / R_{\ell,H,q}\right)^{\times}$. It may be verified that
\begin{equation}
\left| \left(\MM_q / R_{\ell,H,q}\right)^{\times} \right| = q^{\ell} \phi(H),
\end{equation}
where $\phi(H)$ is Euler's totient function.
\subsubsection{Characters}
For every character $\chi$ of the finite abelian group $\left( \MM_q / R_{\ell,H,q}\right)^{\times}$, we define $\chi^{\dagger}$ with domain $\MM_q$ as follows. If $A$ is invertible modulo $R_{\ell,H,q}$ and if $\mathfrak{c}$ is the equivalence class of $A$, then $\chi^{\dagger}(A)=\chi(\mathfrak{c})$. If $A$ is not invertible, then $	\chi^{\dagger}(A)=0$.

The set of functions $\chi^{\dagger}$ defined in this way are called the characters of the relation $R_{\ell,H,q}$, or sometimes ``characters modulo $R_{\ell,H,q}$''. We abuse language somewhat and write $\chi$ instead of $\chi^{\dagger}$ to indicate a character of the relation $R_{\ell,H,q}$ derived from the character $\chi$ of the group $\left( \MM_q / R_{\ell,H,q}\right)^{\times}$. Thus we write $\chi_0$ for the character of $R_{\ell,H,q}$ which has the value $1$ when $A$ is invertible and the value $0$ otherwise. We denote by $G(R_{\ell,H,q})$ the set $\{ \chi^{\dagger} : \chi \;\textrm{a character of}\; \left( \MM_q / R_{\ell,H,q}\right)^{\times} \}$. 

A set of polynomials in $\MM_q$ is called a representative set modulo $R_{\ell,H,q}$ if the set contains one and only one polynomial from each equivalence class of $R_{\ell,H,q}$. If $\chi_{1},\chi_{2} \in G(R_{\ell,H,q})$, then
\begin{equation}\label{ortho1}
\frac{1}{q^{\ell}\phi(H)} \sum_{F} \chi_1(F) \overline{\chi_2}(F) = \begin{cases} 0 & \text{if }\chi_1 \neq \chi_2 ,\\ 1 & \text{if }\chi_1 = \chi_2 ,\end{cases}
\end{equation}
where in this sum $F$ runs through a representative set modulo $R_{\ell,M,q}$. If $n \ge \ell + \deg (H)$, then $\MM_{n,q}$ is a disjoint union of $q^{n-\ell-\deg (H)}$ representative sets, and a set of polynomials on which $\chi\in G(R_{\ell,H,q})$ vanishes. Thus, applying \eqref{ortho1} with $\chi_2=\chi_0$, we obtain that for all $n \ge \ell+\deg(H)$:
\begin{equation}\label{orthouse}
\frac{1}{q^{n-\deg(H)}\phi(H)} \sum_{F \in \MM_{n,q}} \chi(F) = \begin{cases} 0 & \text{if }\chi \neq \chi_0 ,\\ 1 & \text{if }\chi = \chi_0. \end{cases}
\end{equation}
We also have, if $A,B \in \MM_q$ are coprime to $H$,                                                                                                                                                                                                                                                           
\begin{equation}\label{ortho2}                                
\frac{1}{q^{\ell}\phi(H)}\sum_{\chi \in G(R_{\ell,H,q})} \chi(A)\overline{\chi}(B) = \begin{cases} 1 & \text{if }A\equiv B \bmod R_{\ell,H,q}, \\ 0 & \text{otherwise}. \end{cases}
\end{equation}                                        
If $\chi \in G(R_{\ell,1,q})$ we say that $\chi$ is a \emph{short interval character of $\ell$ coefficients}, and if $\chi \in G(R_{0,H,q})$ we say that $\chi$ is a \emph{Dirichlet character modulo $H$}. Every element of $G(R_{\ell,H,q})$ is a product of an element from $G(R_{\ell,1,q})$ with an element from $G(R_{0,H,q})$.  

\subsubsection{\texorpdfstring{$L$}{L}-Functions}\label{sec:lfunc}
Let $\chi \in G(R_{\ell,H,q})$. The $L$-function of $\chi$ is the following series in $u$:
\begin{equation}
L(u,\chi) = \sum_{f \in \MM_q} \chi(f)u^{\deg (f)},
\end{equation}
which also admits the Euler product
\begin{equation}\label{eulerlchi}
L(u,\chi) = \prod_{P} (1-\chi(P)u^{\deg (P)})^{-1}
\end{equation}
where the product is over irreducibles in $\MM_q$.
If $\chi$ is the trivial character $\chi_0$ of $G(R_{\ell,H,q})$, then
\begin{equation}
L(u,\chi) = \frac{\prod_{P \mid H} (1-u^{\deg (P)})}{1-qu}.
\end{equation}
Otherwise, the orthogonality relation \eqref{orthouse} implies that $L(u,\chi)$ is a polynomial in $u$ of degree at most $\ell + \deg (H)-1$.

The first one to realize that Weil's proof of the Riemann Hypothesis for Function Fields \cite[Thm.~6,~p.~134]{weil1974} implies the Riemann Hypothesis for the $L$-functions of $\chi \in G(R_{\ell,H,q})$ was Rhin \cite[Thm.~3]{rhin1972} in his thesis (cf. \cite[Thm.~5.6]{effinger1991} and the discussion following it). Hence we know that if we factorize $L(u,\chi)$ as
\begin{equation}
L(u,\chi) = \prod_{i=1}^{\deg (L(u,\chi))} (1-\gamma_i(\chi)u),
\end{equation}
then for any $i$, 
\begin{equation}\label{eq:rh}
\left|\gamma_i(\chi) \right| \in \{1, \sqrt{q}\}.
\end{equation}
\begin{remark}
Although we use below the Riemann Hypothesis for Function Fields (in the form \eqref{eq:rh}) whenever we can, the trivial bound $|\gamma_i(\chi)| \le q$ leads to results only slightly weaker than we obtain.
\end{remark}

 \subsection{Traces of powers, and symmetric functions}\label{sec:symmetric}
\begin{lem}\label{lem:symm}
	Fix $n \ge k \ge 1$. Let $\{ a_i \}_{1 \le i \le k, p \nmid i}$ be a sequence of $k'=k-\lfloor \frac{k}{p} \rfloor$ elements from $\FF_q$. There are polynomials $\{ f_i \}_{1 \le i \le q^{\lfloor \frac{k}{p} \rfloor}} \subseteq \MM_{n,q}$ such that for $M \in \Mat(n,\FF_q)$, the following two conditions are equivalent: 
	\begin{enumerate}
		\item For all $1 \le i \le k$ with $p\nmid i$, we have $\Tr(M^i)=a_i$.
		\item For some $1 \le j \le q^{\lfloor \frac{k}{p} \rfloor}$ we have $\chpo(M) \in I(f_j;n-k-1)$. 
	\end{enumerate}
	Moreover, the sets $\{ I(f_j;n-k-1)\}_{1 \le j \le q^{\lfloor \frac{k}{p} \rfloor}}$ are disjoint, so that
	\begin{equation}\label{eq:pviap}
	\PP_{M \in S}(\forall 1 \le i \le k,\, p \nmid i: \Tr(M^i) = a_i) = \sum_{j=1}^{q^{\lfloor \frac{k}{p} \rfloor}} \PP_{M \in S}(\chpo(M) \in I(f_j;n-k-1))
	\end{equation}
	for any non-empty subset $S \subseteq \glnq$, endowed with a uniform probability measure.
\end{lem}
\begin{proof}
For  $M\in \Mat(n,\FF_q)$, write its characteristic polynomial as
	\begin{equation}
	\chpo(M) =T^n+ c_1(M) T^{n-1}+\ldots+c_n(M).
	\end{equation}
	As $\Tr(M^i)$ are the power sum symmetric polynomials in the eigenvalues of $M$ (working in the closure of $\FF_q$), and $(-1)^ic_i(M)$ are the elementary symmetric functions in the eigenvalues of $M$, Newton's identities allow us to express $\Tr(M^i)$ as
	\begin{equation}\label{eq:newton}
	\Tr(M^i) = -ic_i(M) + P_i(c_1(M),\ldots,c_{i-1}(M))
	\end{equation}
	for some explicit $P_i(x_1,\ldots,x_{i-1}) \in \mathbb{Z}[x_1,\ldots,x_{i-1}]$.

	If $c_p(M),c_{2p}(M),\ldots,c_{p\lfloor \frac{k}{p} \rfloor}(M)$ are fixed, then prescribing $\Tr(M^i)=a_i$ for $1 \le i\le k,\, p\nmid i$, amounts to prescribing $ c_i(M)$ for $1 \le i\le k,\, p\nmid i$. This is because for each such $i$ we may iteratively use \eqref{eq:newton} to solve for $c_i(M)$ from $\Tr(M^i)$ and the previous coefficients $c_j(M)$. This works in the other direction as well: we may use \eqref{eq:newton} to solve for $\Tr(M^i)$ from $c_i(M)$ and the previous coefficients $c_j(M)$.
	
	We wish to formulate this observation using short intervals. If $f \in \MM_{n,q}$ and $h<n$, then $\chpo(M) \in I(f;h)$ if and only if $c_1(M),\ldots,c_{n-h-1}(M)$ coincide the first $n-h-1$ next-to-leading coefficients of $f$. Thus, if $\{ c_{pi}(M)\}_{1 \le i \le \lfloor k/p\rfloor}$ are fixed, then the set of matrices $M$ in $\Mat(n,\FF_q)$ with $\Tr(M^i)=a_i$ for $1 \le i\le k,\, p\nmid i$ is of the form $I(f;n-k-1)$, where $f \in \MM_{n,q}$ is determined by $\{ c_{pi}(M)\}_{1 \le i \le \lfloor k/p\rfloor}$ and the values $a_i$. As we vary $\{ c_{pi}(M)\}_{1 \le i \le \lfloor k/p\rfloor}$, this $f$ necessarily changes as well.
	
	All in all, we find that indeed the set of matrices $M$ in $\Mat(n,\FF_q)$ with $\Tr(M^i)=a_i$ for $1 \le i\le k,\, p\nmid i$ is a disjoint union of sets of the form $I(f;n-k-1)$, for $q^{\lfloor \frac{k}{p} \rfloor}$ different polynomials $f$ in $\MM_{n,q}$.
\end{proof}
\begin{remark}
Lemma~\ref{lem:symm} gives Theorem~\ref{thm:manytraces} from Theorem~\ref{thm:glveryshort}. The lemma can also be used to deduce a (weaker) version of Theorem~\ref{thm:arbexp} from Theorem~\ref{thm:apshort}. In fact for the purpose of Theorem~\ref{thm:arbexp}, Lemma \ref{lem:sym} gives a slightly better estimate.
\end{remark}

Given $\lambda_1,\ldots,\lambda_k \in \FF_q$, define a function $\chi_{\vec{\lambda}}\colon  \MM_q \to \mathbb{C}$ by
 \begin{equation}
 \chi_{\vec{\lambda}}(f(T)) = e^{\frac{2\pi i}{p} \Tr_{\FF_q/\FF_p}  \left( \sum_{i=1}^{k} \lambda_i (\alpha_1^{i}+\ldots+\alpha_d^{i})\right) },
 \end{equation}
 where $f \in \MM_q$ and $\{\alpha_i\}_{i=1}^{d}$ are the roots of $f$ in $\overline{\FF_q}$, listed with multiplicities.
 \begin{lem}\label{lem:sym}
 	Let $k$ be a positive integer coprime to $p$. Let $\lambda_1,\ldots,\lambda_k \in \FF_q$ with $\lambda_k \neq 0$. The following hold.
 	\begin{enumerate}
 		\item The function $\chi_{\vec{\lambda}}$ is completely multiplicative, that is, $\chi_{\vec{\lambda}}(fg)=\chi_{\vec{\lambda}}(f)\chi_{\vec{\lambda}}(g)$ for all $f,g \in \MM_q$.
 		\item Given $M \in \glnq$, we have $\chi_{\vec{\lambda}}(\chpo(M)) = e^{\frac{2\pi i}{p} \Tr_{\FF_q/\FF_p} \left(\sum_{i=1}^{k} \lambda_i \Tr(M^{i})\right)}$.
 		\item The function $\chi_{\vec{\lambda}}$ is a short interval character of $k$ coefficients.
 		\item The character $\chi_{\vec{\lambda}}$ is non-trivial.
 	\end{enumerate}
 \end{lem}
 \begin{proof}
 	The first part is trivial. The second part follows from the observation that $\Tr(M^i)$ is the sum of the $i$-th powers of the eigenvalues of $M$, which are the roots of $\chpo(M)$. For the third part, we need to show that $\alpha_1^i+\ldots+\alpha_d^i$ is a function of the first $i$ next-to-leading coefficients of $f$, for each $i$. This is a direct consequence of the fundamental theorem of symmetric polynomials, which in particular says that $\alpha_1^i+\ldots+\alpha_d^i$ is a polynomial (with integer coefficients) in the first $i$ elementary symmetric polynomials in $\alpha_i$, which -- up to sign -- are the first $i$ next-to-leading coefficients of $f(T)$. For the last part, we use Newton's identities which say in particular that
 	\begin{equation}\label{eq:newton2}
 	\sum_{i=1}^{d} \alpha_i^{k} = F(e_1(\alpha_i), e_2(\alpha_i),\ldots,e_{k-1}(\alpha_i)) + (-1)^{k-1} k e_{k}(\alpha_i),
 	\end{equation}
 	where $e_i$ is the $i$-th elementary symmetric polynomial and $F$ is some polynomial with integer coefficients, which assumes the value $0$ at $\vec{0}$. As $k \neq 0$ in $\FF_q$ by assumption, this shows that $\chi_{\vec{\lambda}}(T^{k}+c)$ is not equal to $1$ whenever $\Tr_{\FF_q/\FF_p}(c) \neq 0$. 
 \end{proof}

\subsection{A uniqueness theorem}\label{sec: uniqueness}

We will make use of the following simple result.

\begin{lem}\label{lem:uniqueness}
Two functions $\alpha, \beta\colon  \MM_q \rightarrow \CC$ are identical (that is $\alpha(g) = \beta(g)$ for all $g \in \MM_q$) if and only if for all Hayes characters $\chi\colon \MM_q \rightarrow \CC$,
\begin{equation}\label{eq:mult_equality}
\sum_{f \in \MM_q} \alpha(f) \chi(f) u^{\deg(f)} = \sum_{f \in \MM_q} \beta(f) \chi(f) u^{\deg(f)},
\end{equation}
with equality in the sense of formal power series in $u$.
\end{lem}

\begin{remark}
The formal power series in \eqref{eq:mult_equality} are well-defined; note that the sums $c_n(\alpha,\chi) = \sum_{f \in \MM_{n,q}} \alpha(f) \chi(f)$ are finite for all $n, q$, and $\sum_{f \in \MM_q} \alpha(f) \chi(f) u^{\deg(f)} = \sum_{n \geq 0} c_n(\alpha,\chi) u^n$.
\end{remark}

\begin{proof}
If $\alpha, \beta$ are identical then \eqref{eq:mult_equality} is evident. On the other hand, suppose \eqref{eq:mult_equality} holds for all Hayes characters $\chi$. For any $n \geq 0$ and  $g \in \MM_{n,q}$, take $\ell > n$. Supposing \eqref{eq:mult_equality} holds, one has for all $\chi \in G(R_{\ell,1,q})$ (that is, short interval characters of $\ell$ coefficients),
\begin{equation}\label{eq:chi_ident}
\sum_{f \in \MM_{n,q}} \alpha(f) \chi(f) = \sum_{f \in \MM_{n,q}} \beta(f) \chi(f).
\end{equation}
Hence from \eqref{ortho2}
\begin{equation}\label{eq:nochi_ident}
\alpha(g) = \frac{1}{q^\ell} \sum_{\chi \in G(R_{\ell,1,q})} \chi(g) \sum_{f \in \MM_{n,q}} \alpha(f) \overline{\chi}(f) = \frac{1}{q^\ell} \sum_{\chi \in G(R_{\ell,1,q})} \chi(g) \sum_{f \in \MM_{n,q}} \beta(f) \overline{\chi}(f) = \beta(g).
\end{equation}
Since $g$ was arbitrary, this proves $\alpha$ and $\beta$ are identical.
\end{proof}

\subsection{Some $q$-series identities}\label{sec:q}
We make use of the following results from the theory of $q$-series, with the notation
\begin{equation}\label{eq:q_pochhammer}
(a;\nu)_n = (1-a)(1-a\nu)\cdots(1-a\nu^{n-1}),
\end{equation}
and for $|\nu| < 1$,
\begin{equation}\label{eq:q_pochhammer_infinity}
(a;\nu)_\infty = \lim_{n\rightarrow\infty} (a;\nu)_n.
\end{equation}
\begin{thm}[Euler]\cite[Eq.~(18)]{gasper2004}\label{thm:q_infinity}
For $|x|,|\nu| < 1$,
\begin{equation}\label{eq:q_infinity}
\frac{1}{(x;\nu)_\infty} = \sum_{j=0}^\infty \frac{x^j}{(\nu;\nu)_j}.
\end{equation}
\end{thm}
As a restatement, setting $\nu = 1/V$ and $x=y/V$, we have for $|y|<|V|$ and $|V| > 1$,
\begin{align}
\prod_{i=1}^\infty \frac{1}{1-y/V^i} &=1+ \sum_{j=1}^\infty \frac{V^{j(j-1)}}{(V^j-1)(V^j-V)\cdots(V^j-V^{j-1})}\, y^j \label{eq:swapped_infinity}\\
&= 1+ \sum_{j=1}^\infty \frac{V^{j(j-1)/2}}{(V^j-1)(V^{j-1}-1)\cdots(V-1)} \, y^j. \label{eq:swapped_infinity_numden}
\end{align}
\begin{thm}[Euler]\cite[Eq.~(19)]{gasper2004}\label{thm:q_infinity2}
For $|x|,|\nu| < 1$,
\begin{equation}\label{eq:q_infinity2}
\sum_{j=0}^\infty \frac{(-1)^j \nu^{j(j-1)/2} x^j}{(\nu;\nu)_j} = (x;\nu)_\infty.
\end{equation}
\end{thm}

As a restatement, setting $\nu = 1/V$ and $x = y/V$, we have for $|y|<|V|$ and $|V| > 1$,
\begin{equation}\label{eq:swapped_infinity2}
1+ \sum_{j=1}^\infty \frac{(-1)^j}{(V^j-1)(V^{j-1}-1)\cdots(V-1)}\, y^j = \prod_{i=1}^\infty (1-y/V^i).
\end{equation}

\section{Results for $\glnq$}
\label{sec:glnq}

\subsection{$\glnq$ and the space of characteristic polynomials}\label{subsec:back_GL}
Recall that $\glnq = \{M \in \Mat(n,\FF_q):\, \det(M)\neq 0\}$.

\begin{proposition}\label{prop:GLsize}\cite[Ch.~IV.3]{artin1988}
	We have $|\glnq| = (q^n-1)(q^n-q)\cdots (q^n-q^{n-1}).$
\end{proposition}

We define $\MM_{n,q}^{gl} = \{f \in \MM_{n,q}:\, (f,T)=1\}$. It is plain for $M \in \glnq$ that $\chpo(M) \in \MM_{n,q}^{gl}$, and it follows by considering companion matrices that in fact $\MM_{n,q}^{gl} = \{\chpo(M):\, M \in \glnq \}$. The reader should take a moment to verify
\begin{equation}
|\MM_{n,q}^{gl}| = q^{n}(1-q^{-1}), \qquad n \ge 1.
\end{equation}
Let $\MM_{q}^{gl} = \cup_{n\geq 0} \MM_{n,q}^{gl}$. It is easy to see that $\MM_{q}^{gl}$ is a submonoid of $\MM_q$ with respect to multiplication (that is, if $f,g \in \MM_{q}^{gl}$, then $fg \in \MM_{q}^{gl}$). Furthermore, by unique factorization of polynomials with coefficients in $\FF_q$, it follows that each $f \in \MM_q^{gl}$ factorizes uniquely as $f = P_1^{e_1}\cdots P_r^{e_r}$ with each $P_i$ an irreducible monic polynomial and $P_i \neq T$ for all $i$.

\subsection{Expressions for $P_{GL}(f)$}\label{subsec:P_GL}
Throughout this subsection we will take random $M \in \glnq$ according to Haar measure (that is, uniform measure). 

\begin{dfn}\label{dfn:P_GL_def}
	For $f \in \MM_q$, we define the arithmetic function
	\begin{equation}\label{eq:P_GL_def}
	P_{GL}(f) = \PP_{M \in \glnq}(\chpo(M) = f) = \frac{|\{M \in \glnq:\, \chpo(M) = f\}|}{|\glnq|},
	\end{equation}
	where we take $n = \deg(f)$. For $f=1$ we define $P_{GL}(1)=1$.
\end{dfn}

Previous authors have considered the counts $|\{M \in \glnq:\, \chpo(M) = f\}|$. A closed formula for these is due independently to Reiner \cite[Thm.~2]{reiner1961} and Gerstenhaber \cite[Sec.~2]{gerstenhaber1961number}. Another proof based on cycle types of matrices was later given in \cite[Thm.~17]{fulman1999}. We have found it natural to phrase these formulas in the language of probability as above.

\begin{thm}[Reiner, Gerstenhaber]\label{thm:glnq_charpoly}
	If $f \in \MM_{q}^{gl}$ factorizes as $f = P_1^{e_1} \cdots P_r^{e_r}$ with each $P_i$ an irreducible monic polynomial, set $m_i = \deg(P_i)$. Then
	\begin{equation}
	P_{GL}(f) = \prod_{i=1}^r \frac{q^{m_i e_i(e_i-1)}}{|\mathrm{GL}(e_i, q^{m_i})|}.
	\end{equation}
	If $f$ is a monic polynomial but $f \notin \MM_{q}^{gl}$, then $P_{GL}(f) = 0$.
\end{thm}

We make the new observation that $P_{GL}(f)$ can also be written as an infinite Dirichlet convolution of simple arithmetic functions defined on $\MM_q$.

\begin{thm}\label{thm:GL_conv}
	For $g \in \MM_q$, define the arithmetic functions $\alpha_1^{GL}, \alpha_2^{GL},...$ by
	\begin{equation}\label{eq:alpha_GL}
	\alpha_i^{GL}(g) := \frac{1}{|g|^i} \mathbf{1}_{\MM_q^{gl}}(g)
	= \begin{cases} 1/|g|^i & \textrm{for } g \in \MM_{n,q}^{gl} ,\\ 0 & \textrm{otherwise.} \end{cases}
	\end{equation}
	Then for $f \in \MM_q$,
	\begin{equation}\label{eq:GL_conv}
	P_{GL}(f) = \lim_{k\rightarrow\infty} (\alpha_1^{GL}\ast\alpha_2^{GL}\ast \cdots \ast \alpha_k^{GL}) (f).
	\end{equation}
\end{thm}

The proof of Theorem \ref{thm:GL_conv} is given below. It is a corollary of the following theorem.

\begin{thm}\label{thm:L_GL_expansion}
	For a Hayes character $\chi$ and $|u| < 1/q$, define
	\begin{equation}\label{eq:L_GL_def}
	L_{GL}(u,\chi) = \sum_{f \in \MM_q^{gl}} \chi(f) u^{\deg(f)}.
	\end{equation}
	This sum converges absolutely for $|u| < 1/q$, and for $|u| < 1$ we have
	\begin{equation}\label{eq:P_GL_generating}
	\sum_{f \in \MM_q} P_{GL}(f) \chi(f) u^{\deg(f)} = \prod_{i=1}^\infty L_{GL}\Big(\frac{u}{q^i},\chi\Big),
	\end{equation}
	with both the left hand sum and the right hand product converging absolutely.
\end{thm}

\begin{proof}
	We treat claims about convergence first. For \eqref{eq:L_GL_def}, note that $\sum_{f\in\MM_{n,q}^{gl}}|\chi(f)| \leq q^n$, which furthermore implies $|L(u,\chi)| \leq 1/(1-q|u|)$. For the left hand side of \eqref{eq:P_GL_generating}, note that $\sum_{f \in \MM_{n,q}} |P_{GL}(f)| \leq 1$, and for the right hand side note that $|L_{GL}(u/q^i,\chi)| \leq 1/(1-|u|/q^{i-1})$.
	
	We will make use of the Euler product: 
	\begin{align}\label{eq:L_GL_euler}
	\notag L_{GL}(u,\chi) &= \sum_{f \in \MM_q^{gl}} \chi(f) u^{\deg(f)} \\
	&= \prod_{P\neq T} \frac{1}{1-\chi(P) u^{\deg(P)}},
	\end{align}
	where the product is over all irreducible monic polynomials $P$ in $\MM_q$ with $P \neq T$. As $|\{P \in \MM_{n,q}:\, P \textrm{ irreducible}\}| \leq |\MM_{n,q}| \leq q^n$, the reader may verify that the product in \eqref{eq:L_GL_euler} converges for $|u|< 1/q$.  As $|P| = q^{\deg(P)}$, we have
	\begin{equation}\label{eq:L_at_u/q^i}
	L_{GL}\Big(\frac{u}{q^i},\chi\Big) = \prod_{P\neq T} \frac{1}{1 - \chi(P) u^{\deg(P)}/|P|^i}.
	\end{equation}
	From Theorem \ref{thm:glnq_charpoly}, one sees that $P_{GL}$ is a multiplicative arithmetic function on $\MM_q$, with
	\begin{equation}
	P_{GL}(P^e) = \frac{|P|^{e(e-1)}}{(|P|^e-1)(|P|^e-|P|)\cdots (|P|^e-|P|^{e-1})},
	\end{equation}
	for $P$ an irreducible monic polynomial with $P\neq T$. Hence
	\begin{align}
	\sum_{f \in \MM_q} P_{GL}(f) \chi(f) u^{\deg(f)} &= \prod_{P\neq T} \Big( 1+ \sum_{e=1}^\infty \frac{|P|^{e(e-1)}}{(|P|^e-1)(|P|^e-|P|)\cdots (|P|^e-|P|^{e-1})} (\chi(P) u^{\deg(P)})^e\Big) \\
	&= \prod_{P\neq T} \prod_{i=1}^\infty \frac{1}{1-\chi(P)u^{\deg(P)}/|P|^i} = \prod_{i=1}^\infty L_{GL}\Big(\frac{u}{q^i},\chi\Big),
	\end{align}
	using \eqref{eq:swapped_infinity} to pass to the second line. The order of products in the final line may be swapped owing to absolute convergence (justified using the same facts as for the absolute convergence of \eqref{eq:L_GL_euler}).
\end{proof}

\begin{proof}[Proof of Theorem \ref{thm:GL_conv}]
	Theorem~\ref{thm:L_GL_expansion} implies for $|u|< 1$,
	\begin{align}\label{eq:P_GL_sum_to_conv}
	\sum_{f\in\MM_q} P_{GL}(f) \chi(f) u^{\deg(f)} &= \lim_{k \rightarrow\infty} \prod_{i=1}^k \sum_{g \in \MM_q} \alpha_i^{GL}(g) \chi(g) u^{\deg(g)} \\
	\notag &= \lim_{k\rightarrow\infty} \sum_{f \in \MM_q} (\alpha_1^{GL} \ast \cdots \ast \alpha_k^{GL}) (f) \chi(f) u^{\deg(f)}.
	\end{align}
	We claim the limit and sum can be swapped in the last line. If $|(\alpha_1^{GL}\ast\cdots\ast\alpha_k^{GL})(f)|\leq 1$ for all $f \in \MM_q$, then this follows via dominated convergence. But indeed we do have $|(\alpha_1^{GL}\ast\cdots\ast\alpha_k^{GL})(f)|\leq 1$. For, set $\chi = 1$; the left hand side of \eqref{eq:P_GL_sum_to_conv} is $\sum_{n\geq 0} u^n$, while for each $f \in \MM_q$, the quantity $(\alpha_1^{GL}\ast\cdots\ast\alpha_k^{GL})(f)$ is non-decreasing in $k$, so comparing coefficients we have $\sum_{f \in \MM_{n,q}} (\alpha_1^{GL}\ast\cdots\ast\alpha_k^{GL})(f) \leq 1.$ Since all terms in this sum are non-negative this verifies the stated bound. (Note that this also implies $\lim_{k\rightarrow\infty}(\alpha_1^{GL} \ast \cdots \ast \alpha_k^{GL}) (f)$ exists, for all $f \in \MM_q$.)
	
	Hence for all Hayes characters $\chi$,
	\begin{equation}
	\sum_{f\in\MM_q} P_{GL}(f) \chi(f) u^{\deg(f)} =  \sum_{f \in \MM_q} \lim_{k\rightarrow\infty}(\alpha_1 \ast \cdots \ast \alpha_k) (f) \chi(f) u^{\deg(f)}.
	\end{equation}
	Lemma \ref{lem:uniqueness} then implies the theorem.
\end{proof}

\subsection{Character sums over $\glnq$}\label{subsec:charsums_glnq}

We now treat averages of $\chi(\chpo(M))$ for $M$ ranging over $\glnq$ and $\chi$ a fixed Hayes character. This information will be used to treat the distribution of traces of powers and characteristic polynomials of $M \in \glnq$ in subsequent subsections. Recall the definition of $L_{GL}(u,\chi)$ in \eqref{eq:L_GL_def}.

\begin{thm}\label{thm:expo}
	Let $n$ be a positive integer. Let $\chi$ be a non-trivial Hayes character from $G(R_{\ell,H,q})$. We have
	\begin{equation}\label{eq:L_GL_factor}
	L_{GL}(u,\chi) = L(u,\chi) \cdot (1-\chi(T) u),
	\end{equation}
	which is a polynomial of degree
	\begin{equation}\label{eq:L_GL_degree_bound1}
	d:=\deg (L_{GL}(u,\chi))\leq \ell+\deg(H).
	\end{equation}
	Factorize $L_{GL}$ as
	\begin{equation}\label{eq:lgnfactor}
	L_{GL}(u,\chi) = \prod_{i=1}^{d} (1-\gamma_i u).
	\end{equation}
	(Note that $\gamma_i$ will have a dependence on $\chi$.) We have $|\gamma_i| \in \{1,\sqrt{q}\}$ for $i = 1,..,d$ and furthermore:
	
	\begin{enumerate}
		\item We have the identity
		\begin{equation}\label{eq:glndesirediden}
		\frac{1}{\left|\glnq\right|} \sum_{M \in \glnq} \chi(\chpo(M))= (-1)^n\sum_{\substack{a_1+\ldots+a_d=n \\ a_1,...,a_d \geq 0}}\prod_{j=1}^{d} \frac{\gamma_j^{a_j}}{(q^{a_j}-1)(q^{a_j-1}-1)\cdots (q-1)},
		\end{equation}
		with the convention that $(q^a-1)\cdots(q-1)$ is $1$ when $a=0$.
		\item If $d>0$ then the following estimate holds:
		\begin{equation}\label{eq:glnestchar}
		\left|\frac{1}{\left|\glnq\right|}\sum_{M \in \glnq} \chi(\chpo(M))\right| \le q^{-\frac{n^2}{2d}} (1+\frac{1}{q-1})^n \binom{n+d-1}{n}.
		\end{equation}
	\end{enumerate}
\end{thm}

\begin{proof}
	\eqref{eq:L_GL_factor} follows from \eqref{eq:L_GL_euler} and \eqref{eulerlchi}. The degree bound \eqref{eq:L_GL_degree_bound1} and bound on $|\gamma_i|$ follow directly from the bound on the degree of $L(u,\chi)$ and Riemann Hypothesis for Functions Fields discussed in \S\ref{sec:lfunc}.
	
	For the exact formula \eqref{eq:glndesirediden}, note that
	\begin{equation}
	\frac{1}{\left|\glnq\right|} \sum_{M \in \glnq} \chi(\chpo(M)) = \sum_{f \in \MM_{n,q}} P_{GL}(f) \chi(f) = [u^n] \prod_{i=1}^\infty L_{GL}\Big(\frac{u}{q^i},\chi\Big),
	\end{equation}
	with the second equality following from \eqref{eq:P_GL_generating}. Using \eqref{eq:lgnfactor} we obtain
	\begin{equation}\label{formav3}
	\frac{1}{\left|\glnq\right|}  \sum_{M \in \glnq} \chi( \chpo(M)) = [u^n] \prod_{i=1}^\infty \prod_{j=1}^{d}(1-\gamma_j \frac{u}{q^i}) = [u^n] \prod_{j=1}^{d} \prod_{i=1}^\infty (1-\gamma_j \frac{u}{q^i}).
	\end{equation}
	Using \eqref{eq:swapped_infinity2} with $y = \gamma_j u$ and $V = q$, we have
	\begin{equation}\label{eq:eulerapp}
	\prod_{i=1}^\infty (1-\gamma_j \frac{u}{q^i}) = \sum_{a \ge 0} \frac{(-1)^a \gamma_j^au^a}{(q^a-1)(q^{a-1}-1)\cdots (q-1)}.
	\end{equation}
	From \eqref{formav3} and \eqref{eq:eulerapp} we have
	\begin{equation}\label{formav4}
	\frac{1}{\left|\glnq\right|}  \sum_{M \in \glnq} \chi( \chpo(M)) = (-1)^n \sum_{a_1 + \ldots + a_{d} = n} \prod_{j=1}^{d} \frac{\gamma_j^{a_j}}{(q^{a_j}-1)\cdots (q-1)},
	\end{equation}
	which establishes \eqref{eq:glndesirediden}.
	
	For the bound \eqref{eq:glnestchar}, we can apply the triangle inequality, the inequality $q^i-1 \ge (\frac{q-1}{q})q^i$ and the bound $|\gamma_j| \le \sqrt{q}$ to \eqref{formav4} and obtain 
	\begin{equation}\label{ineqav}
	\begin{split}
	\left|\frac{1}{\left|\glnq\right|}  \sum_{M \in \glnq} \chi( \chpo(M)) \right| &\le q^{\frac{n}{2}} \sum_{a_1 + \ldots + a_{d} = n} \prod_{j=1}^{d} \frac{1}{(q^{a_j}-1)\cdots (q-1)}\\
	&\le (1+\frac{1}{q-1})^n q^{\frac{n}{2}} \sum_{a_1 + \ldots + a_{d} = n} \prod_{j=1}^{d} \frac{1}{q^{a_j+(a_j-1)+\ldots+1}} \\
	&\le (1+\frac{1}{q-1})^n  q^{\frac{n}{2}} \binom{n+d-1}{n} \max_{a_1 + \ldots + a_{d} = n} q^{-\sum_{j=1}^{d} \binom{a_j+1}{2}}.
	\end{split}
	\end{equation}
	Let $S= \{ (a_1,\ldots,a_d) \in \mathbb{N}_{\ge 0} : a_1 + \ldots + a_{d} = n\}$. The function 
	\begin{equation}
	f(a_1,\ldots,a_d)= \sum_{j=1}^{d} \binom{a_j+1}{2}
	\end{equation}
	on $S$ is minimized when $g(a_1,\ldots,a_d)=\max_{1 \le j \le d} a_j - \min_{1 \le j \le d} a_j$ is minimized, namely, equal to $0$ (if $d \mid n$) or to $1$ (otherwise). Indeed, if we assume in contradiction that the minimum is attained at $(a_1,\ldots,a_d) \in S$ with $a_1 = \max_{1 \le j \le d} a_j$, $a_2= \min_{1 \le j \le d} a_j$ and $a_1 - a_2 > 1$, then $(a_1-1,a_2+1,a_3,\ldots,a_d) \in S$ and
	\begin{equation}
	f(a_1-1,a_2+1,a_3,\ldots,a_d)-f(a_1,a_2,\ldots,a_d) = a_2+1-a_1<0,
	\end{equation}
	a contradiction. If we write $n$ as $dm+r$ ($m \in \mathbb{N}_{\ge 0}$, $0 \le r < d$) then $g(a_1,\ldots,a_d)$ is minimized for $a_1=\ldots = a_r = \lceil \frac{n}{d} \rceil, a_{r+1}=\ldots = a_d = \lfloor \frac{n}{d} \rfloor$, in which case a short calculation shows that
	\begin{equation}\label{eq:minf}
	\min_{\vec{x} \in S} f(\vec{x}) = f(a_1,\ldots,a_d) = \frac{n^2}{2d} +\frac{n}{2} +\frac{r(1-\frac{r}{d})}{2} \ge \frac{n^2}{2d} +\frac{n}{2}.
	\end{equation}
	From \eqref{ineqav} and \eqref{eq:minf}, \eqref{eq:glnestchar} is established.
\end{proof}

\subsection{The distribution of the characteristic polynomial: superexponential bounds for $\glnq$}

We are now in a position to prove Theorem~\ref{thm:apshort}.

\begin{proof}[Proof of Theorem~\ref{thm:apshort}]
	By \eqref{ortho2} with $H=1$ and $\ell=n-h-1$, we have,
	\begin{equation}\label{eq:ortho charpoly}
	\mathbf{1}_{\chpo(M) \in I(f;h)} = \frac{1}{q^{n-h-1}} \sum_{\chi \in G(R_{n-h-1,1,q})} \overline{\chi}(f)\chi(\chpo (M)).
	\end{equation}
	Thus,
	\begin{equation}\label{eq:glshortint}
	\begin{split}
	\PP_{M\in\glnq}(\chpo(M) \in I(f;h) ) &= \frac{1}{q^{n-h-1}} \sum_{\chi \in G(R_{n-h-1,1,q})} \overline{\chi}(f) \frac{\sum_{M \in \glnq}\chi(\chpo(M))}{\left|\glnq\right|} \\
	&= q^{-(n-h-1)} \left(1+\sum_{\substack{\chi \in G(R_{n-h-1,1,q})\\\chi\neq \chi_0}} \overline{\chi}(f) \frac{\sum_{M \in \glnq}\chi(\chpo(M))}{\left|\glnq\right|}\right).
	\end{split}
	\end{equation}
	When $n$ is fixed, the right hand side of \eqref{eq:glnestchar} is monotone increasing in $d$. From \eqref{eq:glshortint} and \eqref{eq:glnestchar}, we have
	\begin{equation}\label{eq:glnshortwithd}
	\begin{split}
	\left|\PP_{M\in\glnq}(\chpo(M) \in I(f;h) ) - q^{-(n-h-1)}\right| & \le q^{-(n-h-1)} \sum_{\substack{\chi \in G(R_{n-h-1,1,q})\\\chi\neq\chi_0}} \left| \frac{\sum_{M \in \glnq}\chi(\chpo(M))}{\left|\glnq\right|} \right| \\
	&\le q^{-\frac{n^2}{2d_1}}(1+\frac{1}{q-1})^n \binom{n+d_1-1}{n},
	\end{split}
	\end{equation}
	where 
	\begin{equation}d_1 = \max_{\substack{\chi \in G(R_{n-h-1,1,q})\\\chi\neq \chi_0}} \deg (L_{GL}(u,\chi)).
	\end{equation}
	By \S\ref{sec:lfunc}, $d_1 \le n-h-1$. Hence, we see that \eqref{eq:glnshort} is established from \eqref{eq:glnshortwithd}. 
\end{proof}

\subsection{The distribution of traces: superexponential bounds for $\glnq$}

\begin{proof}[Proof of Theorem~\ref{thm:arbexp}]
	Consider the additive character $\psi\colon \FF_q\to\CC$ defined as $\psi(a) = e^{\frac{2\pi i}{p} \Tr_{\FF_q/\FF_p}(a)}$. For every $\lambda_1,\ldots,\lambda_k \in \FF_q$, set
	\begin{equation}
	S(n;\lambda_1,\ldots,\lambda_k):=\frac{1}{\left| \glnq \right|} \sum_{M \in \glnq} \psi\left(\sum_{i=1}^{k} \lambda_i \Tr(M^{b_i})\right).
	\end{equation}
	By orthogonality of the characters of the additive group $(\FF_q)^k$, we have
	\begin{equation}
	\PP_{M \in \glnq}(\forall 1 \le i \le k: \Tr(M^{b_i}) = a_i) = q^{-k} \sum_{(\lambda_1,\ldots,\lambda_k) \in (\FF_q)^k} \psi(-\sum_{i=1}^{k}  \lambda_i a_i) S(n;\lambda_1,\ldots,\lambda_k).
	\end{equation}
	The choice $\lambda_1=\ldots=\lambda_k=0$ contributes $q^{-k}$. For the other choices, Lemma~\ref{lem:sym} says that
	\begin{equation}
	S(n;\lambda_1,\ldots,\lambda_k) = \frac{\sum_{M \in \glnq} \chi_{\vec{\lambda}}(\chpo(M))}{\left|\glnq\right|} ,
	\end{equation}
	where $\chi_0 \neq \chi_{\vec{\lambda}}\in G(R_{b_k,1,q})$ was defined in the lemma. By Theorem~\ref{thm:expo},
	\begin{equation}\label{eq:snbound}
	\left| S(n;\lambda_1,\ldots,\lambda_k) \right| \le q^{-\frac{n^2}{2d}}(1+\frac{1}{q-1})^n  \binom{n+d-1}{n}  
	\end{equation}
	where $d=\deg(L_{GL}(u,\chi_{\vec{\lambda}})) \le b_k$. As in the proof of Theorem \ref{thm:apshort}, we use the observation that the right hand side of \eqref{eq:snbound} is monotone increasing in $d$, together with the bound $d \le b_k$, to establish the theorem.
\end{proof}

\subsection{The characteristic polynomial in very short intervals for $\glnq$}\label{subsec:glveryshort}

Finally we treat the case of much smaller short intervals, in Theorem~\ref{thm:glveryshort}. Our method is a variant of the hyperbola method in analytic number theory. We need a few preliminary lemmas. For notational reasons, we let $\sumgl$ denote a sum with all summands restricted to $\MM_q^{gl}$, so for instance
\begin{equation}
\sumgl_{d,\delta: \, d\delta = f} = \sum_{ d,\delta \in \MM_q^{gl}, d\delta  = f}.
\end{equation}

\begin{lem}\label{lem:glnq_P_GL_iterative}
	For $f \in \MM_q$ we have 
	\begin{equation}
	P_{GL}(f) = \frac{1}{|f|} (\mathbf{1}_{\MM_q^{gl}}\ast P_{GL})(f) = \sumgl_{d,\delta:\, d\delta = f} \frac{P_{GL}(\delta)}{|d\delta|}.
	\end{equation}
\end{lem}

\begin{proof}
	This follows directly from an examination of Theorem \ref{thm:GL_conv}. The second equality is just a restatement of the first.
\end{proof}

\begin{lem}\label{lem:glnq_divisorcount1}
	For $n > h \geq 0$ and $\delta \in \MM_q$, if $\deg(\delta) \leq h$, then
	\begin{equation}\label{eq:glnq_divisorcount1}
	\sumgl_{d:\, d\delta \in I(f;h)} 1 = \frac{q^{h+1}}{q^n} \sumgl_{d:\, d\delta \in \MM_{n,q}} 1
	\end{equation}
	for any $f \in \MM_{n,q}$. In particular the above expression is constant as $f$ ranges over $\MM_{n,q}$, for fixed $n$ and $q$. Furthermore the right hand side of $\eqref{eq:glnq_divisorcount1}$ simplifies. In fact for $\delta \in \MM_q^{gl}$ and $\deg(\delta) \leq n$, we have
	\begin{equation}\label{eq:glnq_divisorcount1prime}
	\frac{q^{h+1}}{q^n} \sumgl_{d:\, d\delta \in \MM_{n,q}} 1 = \begin{cases}
	\frac{q^{h+1}}{|\delta|}(1-q^{-1}) & \textrm{if}\; \deg(\delta) < n, \\
	\frac{q^{h+1}}{|\delta|} & \textrm{if}\; \deg(\delta) = n.
	\end{cases}
	\end{equation}
\end{lem}

\begin{proof}
	Let $\Delta = \deg(\delta) \leq h$. By the orthogonality relation \eqref{ortho2},
	\begin{equation}
	\sumgl_{d:\, d\delta \in I(f;h)} 1= \frac{1}{q^{n-h-1}} \sum_{\chi \in G(R_{n-h-1,1,q})} \overline{\chi}(f) \sum_{d \in \MM^{gl}_{n-\Delta,q}} \chi(d)\chi(\delta).
	\end{equation}
	But from Theorem~\ref{thm:expo}, $L_{GL}(u,\chi)$ is a polynomial of degree no more than $n-h-1$ for all non-trivial $\chi \in G(R_{n-h-1,1,q})$. For $\Delta \leq h$, we have $n-\Delta > n-h-1$, so
	\begin{equation}
	\sum_{d \in \MM_{n-\Delta,q}^{gl}} \chi(d) = 0, \quad \textrm{for } \chi \in G(R_{n-h-1,1,q}),\, \chi \neq \chi_0.
	\end{equation}
	Hence
	\begin{equation}
	\sumgl_{d:\, d\delta \in I(f;h)} 1 = \frac{1}{q^{n-h-1}} \chi_0(f) \sum_{d \in \MM_{n-\Delta,q}^{gl}} \chi_0(d) \chi_0(\delta),
	\end{equation}
	and this quantity does not depend on $f$ as $\chi_0(f)=1$ for all $f \in \MM_{n,q}$.
	Hence, to see \eqref{eq:glnq_divisorcount1}, note that because this expression is constant for all $f \in \MM_{n,q}$,
	\begin{equation}
	\sumgl_{d:\, d\delta \in I(f;h)} 1 = \frac{1}{q^n} \sum_{f \in \MM_{n,q}} \sumgl_{d:\, d\delta \in I(f;h)} 1 = \frac{q^{h+1}}{q^n} \sumgl_{d:\, d\delta \in \MM_{n,q}} 1,
	\end{equation}
	since in passing to the last equality, each $f \in \MM_{n,q}$ will have been counted $q^{h+1}$ times.
	
	Finally, for \eqref{eq:glnq_divisorcount1prime}, it is plain that if $\Delta \le n$,
	\begin{equation}
	\sumgl_{d:\, d\delta \in \MM_{n.q}} 1 = |\MM^{gl}_{n-\Delta,q}| = \begin{cases}
	q^{n-\Delta}(1-q^{-1}) & \textrm{if}\; \Delta < n, \\
	1 & \textrm{if}\; \Delta = n,
	\end{cases}
	\end{equation} 
	and \eqref{eq:glnq_divisorcount1prime} follows from this. 
\end{proof}

\begin{remark} 
	Lemma~\ref{lem:glnq_divisorcount1} can be proved via a more elementary counting argument, but the argument given above will generalize to the other finite classical groups to be considered later.
\end{remark}

\begin{lem}\label{lem_glnq_divisorcount2}
	For $n > h \geq 0$ and $\delta \in \MM_q$, if $\deg(\delta) \geq h+1$, then
	\begin{equation}\label{eq:glnq_divisorcount2}
	\Big|\sumgl_{d:\, d\delta \in I(f;h)} 1 \Big|\leq 1
	\end{equation}
	for all $f \in \MM_{n,q}$.
\end{lem}

\begin{proof}
	Plainly, if $\deg(\delta) \geq h+1$, there will be either exactly one multiple of $\delta$ in $I(f;h)$, or no multiples of $\delta$ in $I(f;h)$.
\end{proof}

We now use these lemmas to estimate the likelihood that a random $M\in \glnq$ lands in a short interval.

\begin{proof}[Proof of Theorem~\ref{thm:glveryshort}]
	Note that
	\begin{align}
	\PP_{M\in \glnq}(\chpo(M) \in I(f;h)) &= \sum_{g \in I(f;h)} P_{GL}(g) \\
	&= \sumgl_{d,\delta:\, d\delta \in I(f;h)} \frac{P_{GL}(\delta)}{|d\delta|},
	\end{align}
	by Lemma \ref{lem:glnq_P_GL_iterative}. But this is
	\begin{equation}\label{eq:gl_short_decomp}
	= \sum_{\deg(\delta) \leq h} \frac{P_{GL}(\delta)}{q^n}\sumgl_{d:\, d\delta \in I(f;h)} 1 + \sum_{h+1 \leq \deg(\delta)  \leq n} \frac{P_{GL}(\delta)}{q^n}\sumgl_{d:\, d\delta \in I(f;h)} 1.
	\end{equation}
	Utilizing \eqref{eq:glnq_divisorcount1}, we have
	\begin{equation}
	\sum_{\deg(\delta) \leq h} \frac{P_{GL}(\delta)}{q^n}\sumgl_{d:\, d\delta \in I(f;h)} 1 = \frac{q^{h+1}}{q^n} \sum_{\deg(\delta) \leq h} \frac{P_{GL}(\delta)}{q^n}\sumgl_{d:\, d\delta \in \MM_{n,q}} 1,
	\end{equation}
	with the advantage that the inner sum may be evaluated for any $\delta$, not only for $\deg(\delta)\le h$. Continuing, we have
	\begin{align}
	\sum_{\deg(\delta) \leq h} \frac{P_{GL}(\delta)}{q^n}\sumgl_{d:\, d\delta \in I(f;h)} 1 &=  \frac{q^{h+1}}{q^n} \sum_{\delta \in \MM_q} \frac{P_{GL}(\delta)}{q^n}\sumgl_{d:\, d\delta \in \MM_{n,q}} 1 - \frac{q^{h+1}}{q^n} \sum_{\deg(\delta) > h} \frac{P_{GL}(\delta)}{q^n}\sumgl_{d:\, d\delta \in \MM_{n,q}} 1 \\
	&= \frac{q^{h+1}}{q^n} \sum_{f \in \MM_{n,q}} P_{GL}(f) -   \frac{1}{q^n} \sum_{n \ge \deg(\delta) > h} P_{GL}(\delta)\frac{q^{h+1}}{|\delta|}(1-q^{-1} \cdot \mathbf{1}_{\deg(\delta)<n}),
	\end{align}
	with the first term in the last line following from Lemma~\ref{lem:glnq_P_GL_iterative} and the second term in the last line following from \eqref{eq:glnq_divisorcount1prime}. But, using the fact that $P_{GL}$ is a probability measure on $\MM_{n,q}$ to simplify both the first and second terms, the above simplifies to 
	\begin{equation}\label{eq:glnq_hyperbola_term1}
	= \frac{q^{h+1}}{q^n} - \frac{q^{h+1}}{q^n}\left( (q^{-h-1} + q^{-h-2} + \ldots + q^{-n+1} )(1-q^{-1}) + q^{-n}\right)	 = \frac{q^{h+1}}{q^n} - \frac{1}{q^n}.
	\end{equation}
	On the other hand, turning to the second term in \eqref{eq:gl_short_decomp}, we have
	\begin{equation}\label{eq:glnq_hyperbola_term2}
	\sum_{h+1 \leq \deg(\delta)  \leq n} \frac{P_{GL}(\delta)}{q^n}\sum_{\substack{d\delta \in I(f;h) \\ d,\delta \in \MM_q^{gl}}} 1 \leq \sum_{h+1 \leq \deg(\delta)  \leq n} \frac{P_{GL}(\delta)}{q^n} \leq \frac{n-h}{q^n},
	\end{equation}
	using \eqref{eq:glnq_divisorcount2} to bound the inner sum and then the fact that $P_{GL}$ is a probability measure on $\MM_{n,q}$.
	
	Applying \eqref{eq:glnq_hyperbola_term1} and \eqref{eq:glnq_hyperbola_term2} to \eqref{eq:gl_short_decomp}, we see that
	\begin{equation}
	\PP_{M\in \glnq}(\chpo(M) \in I(f;h)) = \frac{q^{h+1}}{q^n} - \frac{1}{q^n} + K\cdot\left(\frac{n-h}{q^n}\right),
	\end{equation}
	where $K$ is some number in between $0$ and $1$. This verifies Theorem~\ref{thm:glveryshort}.
\end{proof}

\section{Results for $\slnq$}
\subsection{$\slnq$ and the space of characteristic polynomials}\label{subsec:back_SL}
The special linear group over $\FF_q$ is $\slnq = \{M \in \Mat(n,\FF_q):\, \det(M)=1\}$. It is the kernel of the surjective group homomorphism $\det \colon \glnq \rightarrow \FF_q^{\times}$, and so we have the following.
\begin{proposition}\label{prop:SLsize}
	We have $|\slnq| = (q^n-1)(q^n-q)\cdots (q^n-q^{n-1})/(q-1).$
\end{proposition}

We define $\MM_{n,q}^{sl} = \{f \in \MM_{n,q}:\, (f,T)=1, f(0)=(-1)^{n}\}$. It is plain for $M \in \slnq$ that $\chpo(M) \in \MM_{n,q}^{sl}$, and it follows by considering companion matrices that in fact $\MM_{n,q}^{sl} = \{\chpo(M):\, M \in \slnq)\}$. The reader should take a moment to verify
\begin{equation}
|\MM_{n,q}^{sl}| = q^{n-1}, \qquad n \ge 1.
\end{equation}
Let $\MM_{q}^{sl} = \cup_{n\geq 0} \MM_{n,q}^{sl}.$ It is easy to see that $\MM_{q}^{sl}$ is a submonoid of $\MM_q^{gl}$.
\subsection{Expression for $P_{SL}(f)$}\label{subsec:P_SL}
Throughout this subsection we will take random $M \in \slnq$ according to Haar measure (that is, uniform measure). 

\begin{dfn}
	For $f \in \MM_q$, we define the arithmetic function
	\begin{equation}\label{eq:P_SL_def}
	P_{SL}(f) = \PP_{M \in \slnq}(\chpo(M) = f) = \frac{|\{M \in \slnq:\, \chpo(M) = f\}|}{|\slnq|},
	\end{equation}
	where we take $n = \deg(f)$. For $f=1$ we define $P_{SL}(1)=1$.
\end{dfn}
The following observation allows us to use our formula for $P_{GL}(f)$ (given in Theorem~\ref{thm:GL_conv}) in order to study $P_{SL}$. 
\begin{lem}
	For  $f \in \MM_q$ with $\deg(f) >0$, we have
	\begin{equation}
	P_{SL}(f) = \begin{cases} (q-1) P_{GL}(f) &\mbox{if $f \in \MM_{q}^{sl}$,} \\ 0 & \mbox{otherwise.}\end{cases}
	\end{equation}
Furthermore, for all $f \in \MM_q$,
\begin{equation}
P_{SL}(f) = P_{GL}(f) \sum_{\chi_T \in G(R_{0,T,q})} \chi_T((T^n+(-1)^n)) \overline{\chi_T}(f).
\end{equation}
\end{lem}
\begin{proof}
The first part follows from $|\glnq|/|\slnq| = q-1$.  
A special case of \eqref{ortho2}, with $\ell=1$ and $H=T$, is the following orthogonality relation:
\begin{equation}\label{eq:orthosln}
\frac{1}{q-1}\sum_{\chi_T   \in G(R_{0,T,q})} \overline{\chi_T}(T+(-1)^n) \chi_T(f) = \mathbf{1}_{f(0)=(-1)^n}
\end{equation}
for all $f\in \MM_{q}^{gl}$. Since $f \in \MM_{q}^{sl}$ if and only if $f(0) = (-1)^n$, the second part follows.
\end{proof}
\subsection{Character sums over $\slnq$}\label{subsec:charsums_slnq}
We have the following result on character sums, which allows us to use our results for character sums over $\glnq$ (given in Theorem~\ref{thm:expo}) in order to study corresponding sums over $\slnq$.
\begin{lem}
Let $n$ be a positive integer. Let $\chi$ be a non-trivial Hayes character from $G(R_{\ell,H,q})$. We have
	\begin{equation}\label{eq:slniden}
\frac{1}{\left|\slnq\right|}\sum_{M \in \slnq} \chi(\chpo(M)) = \sum_{\chi_T \bmod T} \overline{\chi_T}(T+(-1)^n)\frac{\sum_{M \in \glnq} (\chi \cdot \chi_T)(\chpo(M))}{\left|\glnq\right|},
\end{equation}
where the sum in the right hand side is over all $q-1$ Dirichlet characters $\chi_T$ modulo $T$ (i.e. $G(R_{0,T,q})$).
\end{lem}
\begin{proof}
Plugging $f=\chpo(M)$ in \eqref{eq:orthosln}, multiplying by $\chi(\chpo(M))$ and averaging over $M \in \glnq$ yields the result.
\end{proof}
\subsection{The distribution of the characteristic polynomial: superexponential bounds for $\slnq$}
\begin{thm}\label{thm:apshort sln}
Let $M \in \slnq$ be a random matrix chosen according to Haar measure. Then for $0 \leq h < n-1$ and for $f \in \MM_{n,q}$,
\begin{equation}
\label{eq:slnshort} \left|\PP_{M\in\slnq}(\chpo(M) \in I(f;h) ) - \frac{q^{h+1}}{q^{n}} \right| \le q^{-\frac{n^2}{2(n-h-1)}} (q-1)(1+\frac{1}{q-1})^n \binom{2n-h-2}{n}.
\end{equation}
\end{thm}
\begin{proof}
The proof follows the line of Theorem~\ref{thm:apshort}. Using \eqref{eq:slniden} and the same argument that gave us \eqref{eq:glshortint}, we obtain
	\begin{equation}\label{eq:slshortint}
	\PP_{M\in\slnq}(\chpo(M) \in I(f;h) ) = \frac{1}{q^{n-h-1}} \sum_{\substack{\chi \in G(R_{n-h-1,1,q})\\ \chi_T \bmod T}} \overline{\chi}(f) \overline{\chi_T}((-1)^n) \frac{\sum_{M \in \glnq}(\chi \cdot \chi_T)(\chpo(M))}{\left|\glnq\right|},
	\end{equation}
	and so from \eqref{eq:slshortint} and Theorem~\ref{thm:expo}, letting $\psi=\chi \cdot \chi_T \in G(R_{n-h-1,T,q})$, we have
	\begin{equation}\label{eq:slnshortwithd}
	\begin{split}
	\left|\PP_{M\in\slnq}(\chpo(M) \in I(f;h) ) - q^{-(n-h-1)}\right| &\le q^{-(n-h-1)} \sum_{\chi_0 \neq \psi \in G(R_{n-h-1,T,q})} \left|\frac{\sum_{M \in \glnq}\psi(\chpo(M))}{\left|\glnq\right|}\right| \\
	& \le q^{-\frac{n^2}{2d_2}}(q-1)(1+\frac{1}{q-1})^n \binom{n+d_2-1}{n},
	\end{split}
	\end{equation}
	where $d_2 = \max_{\chi_0 \neq \psi \in G(R_{n-h-1,T,q})} \deg (L_{GL}(u,\psi))$.
	Since $\psi(T)=0$ for $\psi \in G(R_{n-h-1,T,q})$, it follows that $d_2\le \max_{\chi_0 \neq \psi \in G(R_{n-h-1,T,q})} \deg (L(u,\psi)) \le n-h-1$. Thus, \eqref{eq:slnshort} is established from \eqref{eq:slnshortwithd}. 
\end{proof}
\subsection{The distribution of traces: superexponential bounds for $\slnq$}
\begin{thm}\label{thm:arbexpsln}
Let $M \in \slnq$ be a random matrix chosen according to Haar measure. Fix a strictly increasing sequence $b_1,\ldots,b_k$ of positive integers coprime to $p$. Then for any sequence $a_1,...,a_k$ of elements of elements from $\FF_{q}$, we have
	\begin{equation}\label{eq:probsboundedsln}
	\left|\PP_{M \in \slnq}(\forall 1 \le i \le k: \Tr(M^{b_i}) = a_i) - q^{-k}\right| \le q^{-\frac{n^2}{2 b_k}} (q-1) (1+\frac{1}{q-1})^n \binom{n-1+b_k}{n}. 
	\end{equation}
\end{thm}
The proof is omitted, as it follows closely the proof of Theorem~\ref{thm:arbexp}.
\subsection{The characteristic polynomial in very short intervals for $\slnq$}\label{subsec:slveryshort}

We have the following lemma, whose proof follows the proofs of Lemmas~\ref{lem:glnq_divisorcount1} and \ref{lem_glnq_divisorcount2}.
\begin{lem}\label{lem:slnq_divisorcount1}
	Let $\chi_T$ be a non-trivial Dirichlet character modulo $T$. Let $n > h \geq 0$, $\delta \in \MM_q$ and $f \in \MM_{n,q}$. If $\deg(\delta) \leq h$, then
	\begin{equation}\label{eq:slnq_divisorcount1}
	\sumgl_{d:\, d\delta \in I(f;h)} \chi_T(d) =0.
	\end{equation}
	If  $\deg(\delta) \geq h+1$, then
	\begin{equation}\label{eq:slnq_divisorcount2}
	\Big|\sumgl_{d:\, d\delta \in I(f;h)} \chi_T(d) \Big|\leq 1.
	\end{equation}
\end{lem}

\begin{thm}\label{thm:slveryshort}
	Let $M \in \slnq$ be a random matrix chosen according to Haar measure. Then for $0 \leq h < n$ and $f \in \MM_{n,q}$,
	\begin{equation}\label{eq:slnveryshort}
	\left| \PP_{M \in \slnq}( \chpo(M) \in I(f;h)) - \frac{q^{h+1}}{q^n}\right| \leq \frac{n-h}{q^n/(q-1)}.
	\end{equation}
\end{thm}
\begin{proof}
	Note that
\begin{multline}
\PP_{M\in \slnq}(\chpo(M) \in I(f;h)) - \PP_{M\in \glnq}(\chpo(M) \in I(f;h))\\
= \sum_{\chi_0 \neq \chi_T \bmod T} \chi_T(T+(-1)^n)\sum_{g \in I(f;h)} P_{GL}(g)\overline{\chi_T(g)} \\
= \sum_{\chi_0 \neq \chi_T \bmod T} \chi_T(T+(-1)^n) \sumgl_{d,\delta:\, d\delta \in I(f;h)} \frac{P_{GL}(\delta)\overline{\chi_T}(d\delta)}{|d\delta|},
\end{multline}
by Lemma \ref{lem:glnq_P_GL_iterative}. But this is
\begin{multline}\label{eq:sl_short_decomp}
= \sum_{\chi_0 \neq \chi_T \bmod T} \chi_T(T+(-1)^n) \sum_{\deg(\delta) \leq h} \frac{P_{GL}(\delta)\overline{\chi_T}(\delta)}{q^n}\sumgl_{d:\, d\delta \in I(f;h)} \overline{\chi_T}(d) \\+ \sum_{\chi_0 \neq \chi_T \bmod T}\chi_T(T+(-1)^n) \sum_{h+1 \leq \deg(\delta)  \leq n} \frac{P_{GL}(\delta) \overline{\chi_T}(\delta)}{q^n}\sumgl_{d:\, d\delta \in I(f;h)} \overline{\chi_T}(d).
\end{multline}
According to the first part of Lemma~\ref{lem:slnq_divisorcount1}, the first term  in \eqref{eq:sl_short_decomp} vanishes, while by the second part of it, we have
\begin{equation}\label{eq:slnq_hyperbola_term2}
\left| \chi_T(T+(-1)^n)\sum_{h+1 \leq \deg(\delta)  \leq n} \frac{P_{GL}(\delta) \overline{\chi_T}(\delta)}{q^n}\sumgl_{d:\, d \delta \in I(f;h)} \overline{\chi_T}(d) \right| \leq \sum_{h+1 \leq \deg(\delta)  \leq n} \frac{P_{GL}(\delta)}{q^n} \leq \frac{n-h}{q^n}
\end{equation}
for any $\chi_0 \neq \chi_T \bmod T$, using \eqref{eq:slnq_divisorcount2} to bound the inner sum and then the fact that $P_{GL}$ is a probability measure on $\MM_{n,q}$. We see that
\begin{equation}\label{eq:sln to gln}
\left| \PP_{M\in \slnq}(\chpo(M) \in I(f;h)) - \PP_{M\in \glnq}(\chpo(M) \in I(f;h)) \right| \le  \frac{(q-2)(n-h)}{q^n}.
\end{equation}
Applying the triangle inequality to \eqref{eq:sln to gln} and \eqref{eq:glnveryshort}, this verifies the theorem.
\end{proof}
From Theorem~\ref{thm:slveryshort} we may deduce immediately using Lemma~\ref{lem:symm} an analogous result for traces.
\begin{thm}\label{thm:manytraces sln}
	Let $M \in \slnq$ be a random matrix chosen according to Haar measure. Then for $1 \le k \le n$ and for $\{a_i\}_{1 \le i \le k, \, p \nmid i} \subseteq \FF_q$, we have
	\begin{equation}\label{eq:glnmanytraces sln}
	\left| \PP_{M \in \slnq}( \forall 1 \le i \le k,\, p \nmid i: \Tr(M^i)=a_i)  - \frac{1}{q^{k-\lfloor \frac{k}{p} \rfloor}}\right| \leq \frac{(k+1)q^{\lfloor \frac{k}{p} \rfloor}}{q^n/(q-1)}.
	\end{equation}
\end{thm}

\begin{remark}
The above ideas can be used similarly to obtain results for $\mathrm{SU}(n,q)$ from results for $\unq$, and for $\mathrm{SO}^{\pm}(n,q)$ from results for $\mathrm{O}^{\pm}(n,q)$. In fact, one may study other cosets, e.g. $g\slnq \subseteq \glnq$ for $g \notin \slnq$.
\end{remark}
 
\section{Results for $\unq$}
 
\subsection{$\unq$ and the space of characteristic polynomials}\label{subsec:back_U}
Recall that $\unq = \{M \in \gln{q^2}:\, M \overline{M}^t = I\}$, where $\overline{M}$ is the matrix obtained by replacing entries of $M$ by their $q$-th powers, and $t$ is a transpose.\footnote{More generally, one may consider the group $\{ M \in \gln{q^2}:\, M A \overline{M}^t = A\}$ for a fixed Hermitian invertible matrix $A \in \gln{q^2}$, i.e. $A=\overline{A}^t$. When varying $A$, we obtain groups which are conjugate to one another in $\gln{q^2}$, see \cite[Ch.~10]{taylor1992}.}
 
\begin{proposition}\label{prop:U_size}\cite[Ch.~10]{taylor1992}
We have $|\unq| = q^{\binom{n}{2}} \prod_{i=1}^n(q^i - (-1)^i).$
\end{proposition}
 
In order to treat characteristic polynomials of matrices from $\unq \subseteq \gln{q^2}$, we require the following setup. Recall for $a \in \FF_{q^2}$, the map $a \mapsto a^q$ is an involution of $\FF_{q^2}$. For $f \in \MM_{q^2}$, let $\sigma(f)$ be the map which raises each coefficient to the $q$-th power.
 
\begin{dfn}\label{dfn:f_tilde}
For $n\geq 0$ and $f \in \MM_{n,q^2}^{gl}$, define $\tilde{f} \in \MM_{n,q^2}^{gl}$ by
\begin{equation}
\tilde{f}(T) = \frac{T^n \sigma(f)(T^{-1})}{\sigma(f)(0)}.
\end{equation}
\end{dfn}
 That is, if $f(T) = a_0 + \ldots + a_{n-1} T^{n-1} + T^n$, we have
 \begin{equation}\label{eq:usr_explicitdesc}
 \tilde{f}(T) = \left(\frac{1}{a_0}\right)^q + \left(\frac{a_{n-1}}{a_0}\right)^q T + \ldots + T^n.
 \end{equation}
 Note that $f \mapsto \tilde{f}$ is an involution of $\MM_{n,q^2}^{gl}$.
 
 \begin{dfn}\label{dfn:usr_class}
 	We say that a polynomial $f \in \MM_{n,q^2}^{gl}$ is \emph{unitary self-reciprocal} if $f = \tilde{f}$, and write
 	\begin{equation}
 	\MM_{n,q^2}^{usr} = \{f \in \MM_{n,q^2}^{gl}: f = \tilde{f}\}.
 	\end{equation}
 	Furthermore we write $\MM_{q^2}^{usr} = \cup_{n\geq 0} \MM_{n,q^2}^{usr}$.
 \end{dfn}
 It is easy to see that $\MM_{q^2}^{usr}$ is a submonoid of $\MM_{q^2}^{gl}$. By definition, if $M \in \unq$ we have $M=(\overline{M}^t)^{-1}$, and so $\chpo(M) = \chpo((\overline{M}^t)^{-1}) = \widetilde{\chpo(M)}$, which forces the inclusion $\{ \chpo(M):\, M \in \unq\} \subseteq \MM_{n,q^2}^{usr}$. In fact, from Theorem~\ref{thm:unq_charpoly} below, these two sets are equal.
 
 \begin{proposition}\label{prop:usr_count}
 	We have $|\MM_{n,q^2}^{usr}| = q^n(1+q^{-1})$ for $n \ge 1$.
 \end{proposition}
 
 \begin{proof}
 This is clear by analyzing explicit description \eqref{eq:usr_explicitdesc} of $\tilde{f}$, considering $n$ odd and $n$ even separately. One uses the fact that $\FF_{q^2}^{\times}$ is a cyclic group, and in particular that $a_0^{q+1}=1$ has $q+1$ solutions $a_0 \in \FF_{q^2}^\times$.
 \end{proof}
 
 From Definition \ref{dfn:f_tilde}, one sees for $f,g \in \MM_{q^2}^{gl}$ that $\widetilde{fg} = \tilde{f} \cdot \tilde{g}$. As a consequence of this it follows that $\MM_{q^2}^{usr}$ is a submonoid of $\MM_{q^2}^{gl}$ under multiplication. Furthermore a polynomial $P \in \MM_{q^2}^{gl}$ is irreducible if and only if $\tilde{P}$ is irreducible. 
 
 In fact we have
 
 \begin{thm}\label{thm:usr_factorization}
 	Every $f \in \MM_{q^2}^{usr}$ factorizes uniquely into irreducibles as
 	\begin{equation}
 	f = \prod_{i=1}^r P_i^{e_i} \prod_{j=1}^s (Q_j \tilde{Q_j})^{e'_j},
 	\end{equation}
 	where $P_i = \tilde{P_i}$ for all $i$ and $Q_j \neq \tilde{Q_j}$ for all $j$, and $P_i, Q_j \in\MM_{q^2} \setminus \{T\}$ are irreducible for all $i,j$. Conversely, every such product is in $\MM_{q^2}^{usr}$.
 \end{thm}
 \begin{proof}
 	Consider some $f \in \MM_{q^2}^{usr}$. We can write $f$ as a product of irreducibles:
 	\begin{equation}\label{eq:funfactor}
 	f = \prod_{P \in \MM_{q^2} \setminus \{T\}, \tilde{P}=P} P^{e_P} \prod_{ \{ Q, \tilde{Q}\} \subseteq \MM_{q^2} \setminus \{T\}, \tilde{Q} \neq Q} Q^{g_{Q}} \tilde{Q}^{g_{\tilde{Q}}}.
 	\end{equation}
 	As $g \mapsto \tilde{g}$ is an involution which is completely multiplicative, we can compute $\tilde{f}$ to be
 	\begin{equation}\label{eq:tildefunfactor}
 	\tilde{f} = \prod_{P \in \mathcal{M}_{q^2} \setminus \{T\}, \tilde{P}=P} P^{e_P} \prod_{ \{ Q, \tilde{Q}\} \subseteq \mathcal{M}_{q^2} \setminus \{T\}, \tilde{Q} \neq Q} \tilde{Q}^{g_{Q}} Q^{g_{\tilde{Q}}}.
 	\end{equation}
 	Since $f=\tilde{f}$, the right hand side of \eqref{eq:funfactor} and the right hand side of \eqref{eq:tildefunfactor} are equal. By unique factorization in $\FF_{q^2}[T]$, it follows that $g_{\tilde{Q}}=g_Q$, as needed.
 	
 	Conversely, $\MM_{q^2}^{usr}$ contains $Q\tilde{Q}$ for any irreducible $Q \in \mathcal{M}_{q^2} \setminus \{ T\}$ (again since $g\mapsto \tilde{g}$ is a multiplicative involution). By definition, $\MM_{q^2}^{usr}$ contains every irreducible $P \in \mathcal{M}_{q^2}\setminus \{T\}$ such that $P = \tilde{P}$. As $\MM_{q^2}^{usr}$ is a monoid, it follows that it contains every product of the kind described in the theorem.
 \end{proof}
 For an element of $\MM_{q^2}^{usr}$ to be irreducible imposes a restriction on its possible degree:
 
 \begin{thm}[Fulman]\label{thm:usr_primes_odd}
 	If $P \in \MM_{q^2}^{usr}$ is an irreducible polynomial, then $\deg(P)$ is odd.
 \end{thm}
 
 This is a consequence of \cite[Thm.~9]{fulman1999}, but we give a different short and self-contained proof here.
 
 \begin{proof}
 	If $P$ is irreducible, we may write $P(T) = \prod_{i=0}^{\deg(P)-1} (T - z^{q^{2i}})$ for some $z \in \FF_{q^{2\deg(P)}}$ which is not in any $\FF_{q^{2d}}$ for $d < \deg(P)$. In fact, for such $z$ we have that $z \in \FF_{q^{2d}}$ if and only if $d$ is a multiple of $\deg(P)$.
 	
 	Since $\tilde{P} = P$, this means that $t\mapsto t^{-q}$ is a permutation on the roots of $P$. In particular $z^{-q}$ is a root of $P$ also, so $z^{-q} = z^{q^{2i}}$ for some $0 \leq i < \deg(P)$, i.e. $(z^{q^{2i-1}+1})^q=1$. Since we are in a field of characteristic $p$, it follows that $z^{q^{2i-1}+1}=1$. Since the polynomial $T^{q^{2i-1}+1}-1$ divides $T^{q^{2(2i-1)}-1}-1$, it follows that $z \in \FF_{q^{2(2i-1)}}$. Thus $2i-1$ is a multiple of $\deg(P)$. But $-1 \le  2i-1 < 2\deg(P)$, so necessarily $\deg(P) = 2i-1$.
 \end{proof}
 
 \subsection{Expressions for $P_U(f)$}
 
 Throughout this subsection we will take random $M \in \unq$ according to Haar measure (that is, uniform measure). 
 
 \begin{dfn}
 	For $f \in \MM_q$, we define the arithmetic function
 	\begin{equation}\label{eq:P_U_def}
 	P_{U}(f) = \PP_{M \in \unq}(\chpo(M) = f) = \frac{|\{M \in \unq:\, \chpo(M) = f\}|}{|\unq|},
 	\end{equation}
 	where we take $n = \deg(f)$. For $f=1$ we have $P_{U}(1)=1$.
 \end{dfn}
 Fulman \cite[Thm.~18]{fulman1999} has proved a closed formula for the count $|\{M \in \unq:\, \chpo(M) = f\}|$. We have found it natural to phrase his formula in the language of probability as above.
 \begin{thm}[Fulman]\label{thm:unq_charpoly}
 	Suppose $f \in \MM_{q^2}^{usr}$ factorizes as
 	\begin{equation}
 	f = \prod_{i=1}^r P_i^{e_i} \prod_{j=1}^s (Q_j \tilde{Q_j})^{e'_j},
 	\end{equation}
 	with $P_i = \tilde{P}_i$ for all $i$ and $Q_j \neq \tilde{Q_j}$ for all $j$, with $P_i, Q_j \in \MM_{q^2}$ irreducible for all $i,j$. Set 
 	\begin{equation}
 	m_i = \deg(P_i), \quad m'_j = \deg(Q_j).
 	\end{equation} 
 	Then
 	\begin{equation}
 	P_{U}(f) = \prod_{i=1}^r \frac{q^{m_i e_i(e_i-1)}}{|\mathrm{U}(e_i, q^{m_i})|} \prod_{j=1}^s \frac{q^{2m'_j e'_j(e'_j-1)}}{|\mathrm{GL}(e'_j, q^{2m'_j})|}.
 	\end{equation}
 	If $f$ is a monic polynomial but $f \notin \MM_{q^2}^{usr}$, then $P_{U}(f) = 0$.
 \end{thm}
 
 As above, we observe that $P_{U}(f)$ can be written as an infinite Dirichlet convolution of simple arithmetic functions defined on $\MM_{q^2}$.
 
 \begin{thm}\label{thm:U_conv}
 	Let $\lambda \colon \MM_{q^2} \to \{-1,1\}$ be the Liouville function, namely, the unique completely multiplicative function with $\lambda(P)=-1$ for every irreducible $P$ in $\MM_{q^2}$. For $g \in \MM_{q^2}$, define the arithmetic functions
 	\begin{align*}
 	\beta^U(g) &= |g|^{1/2} \mathbf{1}_{\MM_{q^2}^{usr}}(g), \\
 	\gamma^U(g) &= \lambda(g) \mathbf{1}_{\MM_{q^2}^{usr}}(g),
 	\end{align*} 
 	and further define
 	\begin{equation}\label{eq:alpha_U}
 	\alpha_i^{U}(g) =  \frac{1}{|g|^i} (\beta^U\ast\gamma^U)(g)
 	\end{equation}
 	for $i \ge 1$. Then for $f \in \MM_{q^2}$,
 	\begin{equation}\label{eq:U_conv}
 	P_{U}(f) = \lim_{k\rightarrow\infty} (\alpha_1^{U}\ast\alpha_2^{U}\ast \cdots \ast \alpha_k^{U}) (f).
 	\end{equation}
 \end{thm}
 
 The proof of Theorem~\ref{thm:U_conv} is given below. It is a corollary of the following theorem:
 
 \begin{thm}\label{thm:L_U_expansion}
 	For a Hayes character $\chi$, define
 	\begin{equation}\label{eq:L_usr_chi}
 	L_{usr}(u,\chi) = \sum_{f \in \MM_{q^2}^{usr}} \chi(f) u^{\deg(f)},
 	\end{equation}
 	\begin{equation}\label{eq:L_usr_chilambda}
 	L_{usr}(u,\chi\lambda) = \sum_{f \in \MM_{q^2}^{usr}} \chi(f) \lambda(f) u^{\deg(f)}
 	\end{equation}
 	and
 	\begin{equation}\label{eq:L_U_def}
 	L_U(u,\chi) = L_{usr}(qu,\chi) L_{usr}(u,\chi\lambda).
 	\end{equation}
 	The sums \eqref{eq:L_usr_chi} and \eqref{eq:L_usr_chilambda} converge absolutely for $|u| < 1/q$, and for $|u| < 1$ we have 
 	\begin{equation}\label{eq:P_U_generating}
 	\sum_{f \in \MM_{q^2}} P_U(f) \chi(f) u^{\deg(f)} = \prod_{i=1}^\infty L_U\Big(\frac{u}{q^{2i}},\chi\Big),
 	\end{equation}
 	with both the left hand sum and the right hand product converging absolutely.
 \end{thm}
 
 \begin{proof}
 	We treat claims about convergence first. For \eqref{eq:L_usr_chi} and \eqref{eq:L_usr_chilambda}, note that $\sum_{f \in \MM_{n,q^2}^{usr}} |\chi(f)| \le 2 q^n$ from Prop.~\ref{prop:usr_count}, which furthermore implies $|L_{usr}(u,\chi)|,|L_{usr}(u,\lambda\chi)| \le 1/(1-|qu|)^2$. For the convergence of the left hand side of \eqref{eq:P_U_generating} note that $\sum_{f \in \MM_{n,q^2}} |P_U(f)| \leq 1$, and for the convergence of the right hand side note that $L_U(u/q^{2i}) \le 1/(1-|u/q^{2i-2}|)^4$.
 	
 	We have the Euler products
 	\begin{equation}\label{eq:LChi_usr_Euler}
 	L_{usr}(u,\chi) = \prod_{\substack{P \neq T \\ P = \tilde{P}}} \frac{1}{1-\chi(P) u^{\deg(P)}} \prod_{Q \neq \tilde{Q}} \frac{1}{1 - \chi(Q\tilde{Q}) u^{2\deg(Q)}}
 	\end{equation}
 	\begin{equation}\label{eq:LChilambda_usr_Euler}
 	L_{usr}(u,\chi\lambda) = \prod_{\substack{P \neq T \\ P = \tilde{P}}} \frac{1}{1+\chi(P) u^{\deg(P)}} \prod_{Q \neq \tilde{Q}} \frac{1}{1 - \chi(Q\tilde{Q}) u^{2\deg(Q)}}
 	\end{equation}
 	following from Theorem~\ref{thm:usr_factorization} and the fact that $\lambda(P) = -1$ and $\lambda(Q \tilde{Q}) = 1$ for all $P, Q$. Here and in what follows, the products are respectively over all irreducible monics $P$ such that $P = \tilde{P}$ and $P \neq T$, and likewise over pairs of irreducible monics $\{Q,\tilde{Q}\}$, with $Q, \tilde{Q} \neq T$, and $Q \neq \tilde{Q}$, as in \eqref{eq:funfactor} and \eqref{eq:tildefunfactor}. 

By using the facts that $|\{P \in \MM_{n,q^2}^{usr}:\, P \textrm{ irreducible}\}| \leq |\MM_{n,q^2}^{usr}|\leq q^n(1+q^{-1})$ and $|\{Q \in \MM_{n,q^2}:\, Q\neq \tilde{Q}, \, Q\textrm{ irreducible}\}|\leq |\MM_{n,q^2}| \leq q^{2n}$, it can be seen that the products in \eqref{eq:LChi_usr_Euler} and \eqref{eq:LChilambda_usr_Euler} converge absolutely for $|u| < 1/q$. Hence
 	\begin{multline}
 	L_U\Big(\frac{u}{q^{2i}},\chi\Big) = \prod_{\substack{P\neq T \\ P = \tilde{P}}} \frac{1}{1 - \chi(P)u^{\deg(P)}/|P|^{(2i-1)/2}} \cdot \frac{1}{1 + \chi(P) u^{\deg(P)}/|P|^{i}} \\
 	\times \prod_{Q \neq \tilde{Q}} \frac{1}{1-\chi(Q\tilde{Q})(u^{\deg(Q)}/|Q|^{(2i-1)/2})^2} \cdot \frac{1}{1 - \chi(Q\tilde{Q}) (u^{\deg(Q)}/|Q|^{i})^2},
 	\end{multline}
 	and
 	\begin{equation}\label{eq:L_U_Product_Euler}
 	\prod_{i=1}^\infty L_U\Big(\frac{u}{q^{2i}},\chi\Big) = \prod_{\substack{P\neq T \\ P = \tilde{P}}} \prod_{j=1}^\infty \frac{1}{1 + \chi(P) u^{\deg(P)} (-|P|^{1/2})^{-j}} \prod_{Q \neq \tilde{Q}} \prod_{j=1}^\infty \frac{1}{1 - \chi(Q\tilde{Q}) (u^{\deg(Q)} |Q|^{-j/2})^2}.
 	\end{equation}
 	But from Theorem \ref{thm:unq_charpoly}, one sees that $P_U(\prod P_i^{e_i} \prod (Q_j \tilde{Q}_j)^{e'_j}) = \prod P_U(P_i^{e_i}) \prod P_U((Q_j \tilde{Q}_j)^{e'_j})$, with
 	\begin{align}
 	P_U(P^e) &= \frac{|P|^{e(e-1)/4}}{\prod_{i=1}^e (|P|^{i/2} - (-1)^i)},\\
 	P_U((Q\tilde{Q})^e) &= \frac{|Q|^{e(e-1)}}{(|Q|^{e}-1)(|Q|^{e}-|Q|)\cdots (|Q|^{e}-|Q|^{e-1})},
 	\end{align}
 	for $P = \tilde{P}$ and $Q\neq \tilde{Q}$. Hence
 	\begin{multline}
 	\sum_{f \in \MM_{q^2}} P_U(f)\chi(f) u^{\deg(f)} = \prod_{\substack{P\neq T \\ P = \tilde{P}}} \Big(1 + \sum_{e=1}^\infty \frac{|P|^{e(e-1)/4}}{\prod_{i=1}^e (|P|^{i/2} - (-1)^i)} (\chi(P) u^{\deg(P)})^e\Big) \\
 	\times \prod_{Q\neq \tilde{Q}} \Big(1+ \sum_{e=1}^\infty \frac{|Q|^{e(e-1)}}{(|Q|^{e}-1)\cdots (|Q|^{e}-|Q|^{e-1})} (\chi(Q\tilde{Q}) u^{2\deg(Q)})^e\Big).
 	\end{multline}
 	By \eqref{eq:swapped_infinity_numden} with $y=-\chi(P)u^{\deg(P)}$, $V=-|P|^{1/2}$ and \eqref{eq:swapped_infinity} with $y=\chi(Q\tilde{Q})u^{2\deg(Q)}$, $V=|Q|$, this agrees with the right hand side of \eqref{eq:L_U_Product_Euler}.
 \end{proof}
 
 \begin{proof}[Proof of Theorem \ref{thm:U_conv}]
 	We proceed as in the proof of Theorem \ref{thm:GL_conv}. We have from Theorem \ref{thm:L_U_expansion} that
 	\begin{equation}
 	\sum_{f \in \MM_{q^2}} P_U(f) \chi(f) u^{\deg(f)} = \lim_{k\rightarrow\infty} \sum_{f \in \MM_{q^2}} 
 	(\alpha_1^U\ast\cdots \ast \alpha_k^U)(f) \chi(f) u^{\deg(f)}.
 	\end{equation}
 	And exactly as in the proof of Theorem~\ref{thm:GL_conv}, we may swap the order of limit and summation here and obtain the desired result by using Lemma~\ref{lem:uniqueness}.
 \end{proof}
 
 \subsection{The zeros of $L_{usr}(u,\chi)$}
 
 We examine the $L$-function $L_{usr}(u,\chi)$ defined in \eqref{eq:L_usr_chi}. 
 
 \begin{lem}\label{lem:lfunclu}
 	Let $\chi$ be a Hayes character of the form $\chi_1$ or $\chi_1  \cdot \chi_T$, where $\chi_1$ is a non-trivial short interval character of $\ell$ coefficients and $\chi_T$ is a Dirichlet character modulo $T$.
 	The power series $L_{usr}(u,\chi)$ is a polynomial in $u$ of degree at most $2\ell$.
 \end{lem}
 \begin{proof}
 	We denote the coefficients of $f \in \MM_{n,q^2}$ by $a_i$:
 	\begin{equation}
 	f=T^n+a_1T^{n-1}+\ldots + a_{n-1}T+a_n.
 	\end{equation}
 	The condition $f \in \MM_{n,q^2}^{usr}$  can be described in terms of the coefficients of $f$ as follows:
 	\begin{equation}\label{eq:tildecond}
 	a_{n}^{q+1}=1,\, \forall 1 \le i \le \lfloor \frac{n}{2} \rfloor: a_{n-i} = \frac{a_{i}^q}{a_n^q}.
 	\end{equation}
 	
 	We first treat the case where $\chi=\chi_1$ or $\chi_1\cdot \chi_T$, where $\chi_T$ is the trivial character modulo $T$. For $n \ge 2\ell+1$, the number of $f \in \MM_{n,q^2}^{usr}$ with given $a_1,\ldots,a_{\ell} \in \FF_{q^2}$ is $q^{n-2\ell}+q^{n-2\ell-1}$ and does not depend on the values of $a_1,\ldots,a_{\ell}$. In other words, $\MM_{n,q^2}^{usr}$ is a disjoint union of $q^{n-2\ell-1}(q+1)$ representative sets modulo $R_{\ell,1,q^2}$. Thus, by the orthogonality relation \eqref{ortho1}, we have for all $n \ge 2 \ell +1$ 
 	\begin{equation}
 	\sum_{ f \in \MM_{n,q^2}} \chi(f) = q^{n-2\ell-1}(q+1) \cdot 0 =0.
 	\end{equation}
 	This shows that $L_{usr}(u,\chi)$ is a polynomial of degree at most $2\ell$. If $\chi=\chi_1\cdot \chi_T$ where $\chi_T$ is non-trivial, and $n \ge 2\ell+1$, then the number of $f \in \MM_{n,q^2}^{usr}$ with given $a_1,\ldots,a_{\ell},a_n \in \FF_{q^2}$ with $a_n^{q+1}=1$ is $q^{n-2\ell-1}$ and does not depend on the values  $a_1,\ldots,a_{\ell},a_n$. As $\chi_1(f)$ depends only on $a_1,\ldots,a_{\ell}$ while $\chi_T$ depends only on $a_n$, for all $n \ge 2 \ell+1$ we may write
 	\begin{equation}
 	\begin{split}
 	\sum_{ f \in \MM_{n,q^2}^{usr}} &\chi(f) = \sum_{ f \in \MM_{n,q^2}^{usr}} \chi_1(f) \chi_T(f)\\
 	&=q^{n-2\ell-1}\sum_{ a_1,\ldots,a_\ell \in \FF_{q^2}, \, a_{n}^{q+1} = 1} \chi_1(T^n+a_1T^{n-1}+\ldots + a_{\ell}T^{\ell}+\ldots) \chi_T(T^n+\ldots + a_n)\\
 	&= q^{n-2\ell-1} (\sum_{ a_1,\ldots,a_\ell \in \FF_{q^2}} \chi_1(T^n+a_1T^{n-1}+\ldots + a_{\ell}T^{\ell}+\ldots)) (\sum_{a_{n}^{q+1}=1}  \chi_T(T^n+\ldots + a_n)) = 0,
 	\end{split}
 	\end{equation}
 	where in the last passage we have again used the orthogonality relation \eqref{ortho1} applied to $\chi_1$. Thus $\deg (L_{usr}(u,\chi))$ is at most $2 \ell$ again, as needed.
 \end{proof}
 
 \begin{cor}\label{lem:L_usr_zero_bounds}
 	Let $\chi$ be a Hayes character of the form $\chi_1$ or $\chi_1  \cdot \chi_T$, where $\chi_1$ is a non-trivial short interval character of $\ell$ coefficients and $\chi_T$ is a Dirichlet character modulo $T$. Then
 	\begin{equation}\label{lunqfactor}
 	L_{usr}(u,\chi) = \prod_{i=1}^{\deg (L_{usr}(u,
 		\chi))} (1-\gamma_i u),
 	\end{equation}
 	with $|\gamma_i| \le q$ for all $i$ and $\deg(L_{usr}(u,\chi)) \leq 2\ell$. (Note that $\gamma_i$ will have a dependence on $\chi$.)
 \end{cor}
 
 \begin{proof}
 	The degree bound follows directly from Lemma~\ref{lem:lfunclu}. That $|\gamma_i| \le q$ for all $i$ follows from the convergence of the Euler product \eqref{eq:LChi_usr_Euler} for $|u| < 1/q$, which implies that $L(u,\chi) \neq 0$ for $|u| < 1/q$. If for some $i$ we have $|\gamma_i| > q$, then we would have a contradiction that $L(u,\chi) = 0$ for some $|u| < 1/q$.
 \end{proof}
 We expect that the $L$-function $L_{usr}(u,\chi)$, for $\chi$ a Hayes character in $\chi \in G(R_{\ell,T,q^2})$, will satisfy the Riemann Hypothesis, in the sense that its inverse roots will have absolute value $q^{-1/2}$ or $1$. In fact, Li \cite{li1992, li2006} proved this for certain characters $\chi$, by showing that $L_{usr}(u,\chi)$ is essentially an $L$-function of a Hayes character in $G(R_{\ell,M,q})$. See also Curtis and Shinoda \cite{curtis1999}. We do not pursue this here, as it can only lead to a modest saving of $q^{\frac{n}{2}}$ in Theorem~\ref{thm:arbexpun} and in \eqref{eq:unestchar} below. 
 \subsection{$\lambda(f)$ and $L_{usr}(u,\chi\lambda)$}
 
 We now turn to $L_{usr}(u,\chi\lambda)$. It ends up that this function is no more complicated than $L_{usr}(u,\chi)$. This is due to a peculiar feature of $\lambda(f)$ for $f \in \MM_{q^2}^{usr}$.
 
 \begin{thm}\label{thm:lambda_in_usr}
 	For $f \in \MM_{q^2}^{usr}$, we have $\lambda(f) = (-1)^{\deg(f)}$.
 \end{thm}
 
 \begin{proof}
 	Note that both $\lambda(f)$ and $(-1)^{\deg(f)}$ are completely multiplicative functions. Thus by factorization in $\MM_{q^2}^{usr}$ (Theorem~\ref{thm:usr_factorization}), to prove the theorem we need only show that for irreducible $P = \tilde{P}$ we have $\lambda(P) = (-1)^{\deg(P)}$ and for irreducible $Q \neq \tilde{Q}$ we have $\lambda(Q\tilde{Q}) = (-1)^{\deg(Q\tilde{Q})}$. But $\lambda(P) = -1$ by definition, and likewise $(-1)^{\deg(P)} = -1$ by Theorem~\ref{thm:usr_primes_odd}. And obviously $\lambda(Q\tilde{Q}) = 1 = (-1)^{\deg(Q\tilde{Q})}$.
 \end{proof}
 
 \begin{cor}\label{cor:L_usr_chi_to_chilambda}
 	For a Hayes character $\chi$, $L_{usr}(u,\chi\lambda) = L_{usr}(-u,\chi)$.
 \end{cor}
 
 \begin{proof}
 	This is evident from Theorem~\ref{thm:lambda_in_usr} and the definitions \eqref{eq:L_usr_chi}, \eqref{eq:L_usr_chilambda} of $L_{usr}(u,\chi)$ and $L_{usr}(u,\chi\lambda)$.
 \end{proof}
 
 Hence \eqref{eq:P_U_generating} can be rewritten
 \begin{equation}\label{eq:P_U_generating2}
 \sum_{f \in \MM_{q^2}} P_U(f) \chi(f) u^{\deg(f)}  = \prod_{j=1}^\infty L_{usr}(\frac{u}{q^j(-1)^{j+1}},\chi),
 \end{equation}
 for $|u| < 1$.
 
 \subsection{Character sums over $\unq$}
 
 \begin{thm}\label{thm:expoun}
 	Let $n$ be a positive integer. Let $\chi$ be a Hayes character of the form $\chi_1$ or $\chi_1\cdot\chi_T$, where $\chi_1$ is a non-trivial short interval character of $\ell$ coefficients, and $\chi_T$ is a Dirichlet character modulo $T$.

 	We have (for $L_{usr}(u,\chi)$ defined in \eqref{eq:L_usr_chi}),
 	\begin{equation}
 	L_{usr}(u,\chi) = \prod_{i=1}^{d} (1-\gamma_i u)
 	\end{equation}
 	with
 	\begin{equation}
 	d:=\deg (L_{usr}(u,\chi)) \leq 2\ell.
 	\end{equation}
 	\begin{enumerate}
 		\item We have the identity
 		\begin{equation}\label{eq:undesirediden}
 		\frac{1}{{\left|\unq\right|}}\sum_{M \in \unq} \chi(\chpo(M)) = \sum_{\substack{a_1+\ldots+a_d=n\\a_1,...,a_d \geq 0}}\prod_{j=1}^{d} \frac{(-1)^{\binom{a_j}{2}}\gamma_j^{a_j}}{(q^{a_j}-(-1)^{a_j})\cdots (q^2-1)(q+1)}
 		\end{equation}
 		with the convention that $(q^{a}-(-1)^{a})\cdots (q^2-1)(q+1) = 1$ when $a=0$.
 		\item If $d>0$ then the following estimate holds:
 		\begin{equation}\label{eq:unestchar}
 		\left| \frac{1}{{\left|\unq\right|}}\sum_{M \in \unq} \chi(\chpo(M))\right| \le q^{-\frac{n^2}{2d}} q^{\frac{n}{2}}(1+\frac{1}{q-1})^n \binom{n+d-1}{n}.
 		\end{equation}
 	\end{enumerate}
 \end{thm}
 
 \begin{proof}
 	The proof proceeds in the same way as the proof of Theorem~\ref{thm:expo} with two minor differences: when applying \eqref{eq:swapped_infinity2} to deduce \eqref{eq:undesirediden}, let $y = -\gamma_j u$ and $V = -q$, and in the passage from \eqref{eq:undesirediden} to \eqref{eq:unestchar}, we use the bound $|\gamma_i| \le q$ from Corollary~\ref{lem:L_usr_zero_bounds} rather than the Riemann Hypothesis bound we used for $\glnq$.
 \end{proof}
 
\begin{remark}
In the next subsections we only use  Theorem~\ref{thm:expoun} with $\chi$ a short interval character. It is possible to extend our results from $\unq$ to $\mathrm{SU}(n,q)$ by using Theorem~\ref{thm:expoun} with $\chi \in G(R_{\ell,T,q^2})$.
\end{remark}
 
 \subsection{The distribution of the characteristic polynomial: superexponential bounds for $\unq$}
 
 \begin{thm}\label{thm:apshortun}
 	Let $M \in \unq$ be a random matrix chosen according to Haar measure. Then for $0 \leq h < n-1$ and for $f \in \MM_{n,q^2}$,
 	\begin{equation}\label{eq:unshort}
 	\left| \PP_{M \in \unq} (\chpo(M) \in I(f;h) ) - \frac{q^{2(h+1)}}{q^{2n}} \right| \leq q^{-\frac{n^2}{4(n-h-1)}} q^{\frac{n}{2}}(1+\frac{1}{q-1})^n \binom{3n-2h-3}{n}.
 	\end{equation}
 \end{thm}
 
 \begin{proof}[Proof of Theorem~\ref{thm:apshortun}]
 	By \eqref{ortho2} with $H=1$ and $\ell=n-h-1$, we have
 	\begin{equation}
 	\mathbf{1}_{\chpo(M) \in I(f;h)} = \frac{1}{q^{2(n-h-1)}} \sum_{\chi \in G(R_{n-h-1,1,q^2})} \overline{\chi}(f)\chi(\chpo (M)).
 	\end{equation}
 	Thus,
 	\begin{equation}\label{eq:ushortint}
 	\begin{split}
 	\PP_{M\in\unq}(\chpo(M) \in I(f;h) ) &= \frac{1}{q^{2(n-h-1)}} \sum_{\chi \in G(R_{n-h-1,1,q^2})} \overline{\chi}(f) \frac{\sum_{M \in \unq}\chi(\chpo(M))}{\left|\unq\right|} \\
 	&= q^{-2(n-h-1)} \left(1+\sum_{\substack{\chi \in G(R_{n-h-1,1,q^2})\\\chi\neq \chi_0}} \overline{\chi}(f) \frac{\sum_{M \in \unq}\chi(\chpo(M))}{\left|\unq\right|}\right).
 	\end{split}
 	\end{equation}
 	When $n$ is fixed, the right hand side of \eqref{eq:unestchar} is monotone increasing in $d$. From \eqref{eq:ushortint} and \eqref{eq:unestchar}, we have
 	\begin{equation}\label{eq:unshortwithd}
 	\begin{split}
 	\left|\PP_{M\in\unq}(\chpo(M) \in I(f;h) ) - q^{-2(n-h-1)}\right| & \le q^{-2(n-h-1)} \sum_{\substack{\chi \in G(R_{n-h-1,1,q^2})\\\chi\neq\chi_0}} \left| \frac{\sum_{M \in \unq}\chi(\chpo(M))}{\left|\unq\right|} \right| \\
 	&\le q^{-\frac{n^2}{2d_1}}q^{\frac{n}{2}}(1+\frac{1}{q-1})^n \binom{n+d_1-1}{n},
 	\end{split}
 	\end{equation}
 	where 
 	\begin{equation}d_1 = \max_{\substack{\chi \in G(R_{n-h-1,1,q})\\\chi\neq \chi_0}} \deg (L_{usr}(u,\chi)).
 	\end{equation}
 	By Corollary~\ref{lem:L_usr_zero_bounds}, $d_1 \le 2(n-h-1)$, which concludes the proof. 
 \end{proof}

 \subsection{The distribution of traces: superexponential bounds for $\unq$}
 \begin{proof}[Proof of Theorem~\ref{thm:arbexpun}]
 	Consider the additive character $\psi\colon \FF_{q^2}\to\CC$ defined as $\psi(a) = e^{\frac{2\pi i}{p} \Tr_{\FF_{q^2}/\FF_p}(a)}$. For every $\lambda_1,\ldots,\lambda_k \in \FF_{q^2}$, set
 	\begin{equation}
 	S(n;\lambda_1,\ldots,\lambda_k):=\frac{1}{\left| \unq \right|} \sum_{M \in \unq} \psi\left(\sum_{i=1}^{k} \lambda_i \Tr(M^{b_i})\right).
 	\end{equation}
 	By orthogonality of the characters of the additive group $(\FF_{q^2})^k$, we have
 	\begin{equation}
 	\PP_{M \in \unq}(\forall 1 \le i \le k: \Tr(M^{b_i}) = a_i) = q^{-2k} \sum_{(\lambda_1,\ldots,\lambda_k) \in (\FF_{q^2})^k} \psi(-\sum_{i=1}^{k}  \lambda_i a_i) S(n;\lambda_1,\ldots,\lambda_k).
 	\end{equation}
 	The choice $\lambda_1=\ldots=\lambda_k=0$ contributes $q^{-2k}$. For the other choices, Lemma~\ref{lem:sym} says that
 	\begin{equation}
 	S(n;\lambda_1,\ldots,\lambda_k) = \frac{\sum_{M \in \unq} \chi_{\vec{\lambda}}(\chpo(M))}{\left|\unq\right|} ,
 	\end{equation}
 	where $\chi_0 \neq \chi_{\vec{\lambda}}\in G(R_{b_k,1,q^2})$ was defined in the lemma. By Theorem~\ref{thm:expoun},
 	\begin{equation}\label{eq:snboundun}
 	\left| S(n;\lambda_1,\ldots,\lambda_k) \right| \le q^{-\frac{n^2}{2d}}q^{\frac{n}{2}}(1+\frac{1}{q-1})^n  \binom{n+d-1}{n}  
 	\end{equation}
 	where $d=\deg(L_{usr}(u,\chi_{\vec{\lambda}})) \le 2b_k$ by Corollary~\ref{lem:L_usr_zero_bounds}. As in the proof of Theorem \ref{thm:apshort}, we use the observation that the right hand side of \eqref{eq:snboundun} is monotone increasing in $d$, together with the bound $d \le 2b_k$, to establish the theorem.
 \end{proof}
 \subsection{The characteristic polynomial in very short intervals for $\unq$}
 
 Finally we treat the case of short intervals which may contain only a small number of elements of $\MM_{n,q^2}^{usr}$. In $\MM_{n,q^2}^{usr}$ a polynomial is determined by only half its coefficients, and for this reason short intervals may contain no elements.
 
 \begin{proposition}\label{prop:usrcount_in_shortint}
 	For any $f \in \MM_{n,q^2}$ and $h < n$, we have
 	\begin{equation}
 	|\MM_{n,q^2}^{usr} \cap I(f;h)| = 
 	\begin{cases}
 	q^{2(h+1)-n}(1+q^{-1}) & \textrm{if}\; h \geq (n-1)/2 \\
 	1 & \textrm{if}\; h \leq (n-2)/2\; \textrm{and}\; f \in \MM_{n,q^2}^{usr} \\
 	0 & \textrm{if}\; h \leq (n-2)/2\; \textrm{and}\; f \notin \MM_{n,q^2}^{usr}.
 	\end{cases}
 	\end{equation}
 \end{proposition}
 
 \begin{proof}
 Those cases in which $h\leq (n-2)/2$ are evident. For $h \geq (n-1)/2$, this result is a minor generalization of Prop.~\ref{prop:usr_count}, and follows likewise from analyzing the explicit description \eqref{eq:usr_explicitdesc} of $\tilde{f}$.
 \end{proof}
 
 Thus for $M \in \unq$, it is only if $h$ is larger than $(n-1)/2$ that there exists the possibility that the characteristic polynomial $\chpo(M)$ equidistributes in all short intervals $I(f;h)$ of $\MM_{n,q^2}$. But we show that for $h$ only slightly larger than $(n-1)/2$ that equidistribution does indeed occur. 
 
 \begin{thm}\label{thm:unveryshort}
 	Let $M \in \unq$ be a random matrix chosen according to Haar measure. Then for $(n-1)/2 \leq h < n$ and $f \in \MM_{n,q^2}$,
 	\begin{equation}
 	\left| \PP_{M \in \unq}(\chpo(M) \in I(f;h)) - \frac{(q^2)^{h+1}}{(q^2)^n}\right| \leq \frac{2(n-h)}{q^n}.
 	\end{equation}
 \end{thm}
 
 \begin{remark}
 	Note that working within $\MM_{n,q^2}$, $|I(f;h)|/|\MM_{n,q^2}| = (q^2)^{h+1}/(q^2)^n$, so this result describes equidistribution, at least as long as
 	\begin{equation}
 	\frac{h - (n-1)/2}{\log_q(n)} \rightarrow \infty,
 	\end{equation}
 	since in this case $2(n-h)/q^n = o((q^2)^{h+1}/(q^2)^n)$. As in the other examples we have considered, for $h - (n-1)/2$ growing sufficiently slowly we expect that equidistribution will cease, but we have not investigated this.
 \end{remark}
 
 Our method is similar to that of \S\ref{subsec:glveryshort}. We use the notation that $\sumusr$ denotes a sum with all summands restricted to $\MM_{n,q^2}^{usr}$, so for instance
 \begin{equation}
 \sumusr_{d,\delta: \, d\delta = f} = \sum_{ d,\delta \in \MM_{q^2}^{usr}, d\delta  = f}.
 \end{equation}
 
 \begin{lem}\label{lem:unq_P_U_iterative}
 	For $f \in \MM_{q^2}$, define $P_U^\diamond(f) = (-1)^{\deg(f)}P_U(f)$. Then we have
 	\begin{equation}
 	P_U(f) = \frac{1}{|f|^{1/2}} (\mathbf{1}_{\MM_{q^2}^{usr}}\ast P_U^\diamond)(f) = \sumusr_{d,\delta:\, d\delta = f} \frac{P_U^\diamond(\delta)}{|d\delta|^{1/2}}
 	\end{equation}
 \end{lem}
 
 \begin{proof}
 	The middle and last identity here are obviously equal.
 	
 	In order to show that $P_U(f)$ is equal to these, note that from Theorem \ref{thm:U_conv} we have
 	\begin{equation}\label{eq:P_U_formulaA}
 	P_U(f) = \lim_{k\rightarrow\infty} \underbrace{\sumusr_{\substack{d_1,...,d_{2k} \\ d_1d_2\cdots d_{2k} = f}} \frac{1}{|d_1|^{1/2}} \frac{\lambda(d_2)}{|d_2|} \frac{1}{|d_3|^{3/2}} \frac{\lambda(d_4)}{|d_4|^2} \cdots \frac{1}{|d_{2k-1}|^{k-1/2}}\frac{\lambda(d_{2k})}{|d_{2k}|^k}}_{=:A_k(f)}.
 	\end{equation}
 	We also have
 	\begin{equation}\label{eq:P_U_formulaB}
 	P_U(f) = \lim_{k\rightarrow\infty} \underbrace{\sumusr_{\substack{d_1,...,d_{2k-1} \\ d_1d_2\cdots d_{2k-1} = f}} \frac{1}{|d_1|^{1/2}} \frac{\lambda(d_2)}{|d_2|} \frac{1}{|d_3|^{3/2}} \frac{\lambda(d_4)}{|d_4|^2} \cdots \frac{1}{|d_{2k-1}|^{k-1/2}}}_{=:B_k(f)}.
 	\end{equation}
 	\eqref{eq:P_U_formulaB} is a consequence of \eqref{eq:P_U_formulaA} by the following argument. We have
 	\begin{equation}\label{eq:A_to_B}
 	A_k(f) = B_k(f) + \sumusr_{d | f,\, d\neq 1} A_k(f/d) \frac{\lambda(d)}{|d|^{k}}.
 	\end{equation}
 	The set of divisors $\{d:\, d|f, d\neq 1\}$ is finite and $\lambda(d)/|d|^k \rightarrow 0$ for all $d \neq 1$ as $k \rightarrow \infty$, so for fixed $f$ the second term in \eqref{eq:A_to_B} tends to $0$ as $k\rightarrow\infty$, so 
 	\begin{equation}
 	\lim_{k\rightarrow\infty} A_k(f) = \lim_{k\rightarrow\infty} B_k(f),
 	\end{equation}
 	implying \eqref{eq:P_U_formulaB}.
 	
 	Working with \eqref{eq:P_U_formulaA}, note that from Theorem \ref{thm:lambda_in_usr} we have
 	\begin{align}
 	P_U(f) &= \lim_{k\rightarrow\infty} \sumusr_{\substack{d_1,...,d_{2k} \\ d_1d_2\cdots d_{2k} = f}} \frac{1}{|d_1|^{1/2}} \frac{(-1)^{\deg(d_2)}}{|d_2|} \cdots \frac{1}{|d_{2k-1}|^{k-1/2}}\frac{(-1)^{\deg(d_{2k})}}{|d_{2k}|^k} \\
 	&= \lim_{k\rightarrow\infty} \sumusr_{d,\delta:\, d\delta = f} \frac{(-1)^{\deg(\delta)}}{|d|^{1/2}|\delta|^{1/2}} \sumusr_{\substack{d_2,...,d_{2k} \\ d_2 d_3 \cdots d_{2k} = \delta}} \frac{1}{|d_2|^{1/2}} \frac{(-1)^{\deg(d_3)}}{|d_3|}\cdots \frac{(-1)^{\deg(d_{2k-1})}}{|d_{2k-1}|^{k-1}} \frac{1}{|d_{2k}|^{k-1/2}} \\
 	&= \sumusr_{d,\delta:\, d\delta = f} \frac{P_U^\diamond(\delta)}{|d\delta|^{1/2}},
 	\end{align}
 	where in the final step we use \eqref{eq:P_U_formulaB}. We can exchange sum and limit because $f$ is fixed and the sum is finite.
 \end{proof}
 
 \begin{lem}\label{lem:unq_divisorcount1}
 	Take $n > 0$, $f \in \MM_{n,q^2}$, and $n > h \geq (n-1)/2$. For $\delta \in \MM_{q^2}^{usr}$, if $\deg(\delta) \leq 2h-n+1$, we have
 	\begin{equation}\label{eq:unqdivisorcount1}
 	\sumusr_{d:\, d\delta \in I(f;h)} 1 = \frac{(q^2)^{h+1}}{(q^2)^n} \sumusr_{d:\, d\delta \in \MM_{n,q^2}} 1.
 	\end{equation}
 	In particular the above expression is constant as $f$ ranges over $\MM_{n,q^2}$ for fixed $n$ and $q$. Furthermore the right hand side of \eqref{eq:unqdivisorcount1} simplifies. In fact for $\delta \in \MM_{q^2}^{usr}$ and $\deg(\delta) \leq n$, we have
 	\begin{equation}\label{eq:unqdivisorcount1prime}
 	\frac{(q^2)^{h+1}}{(q^2)^n} \sumusr_{d:\, d\delta \in \MM_{n,q^2}} 1 = \begin{cases}
 	(q^{2(h+1)-n}/|\delta|^{1/2})(1 + q^{-1}) & \textrm{if}\; \deg(\delta) < n, \\
 	q^{2(h+1)-n}/|\delta|^{1/2} & \textrm{if}\; \deg(\delta) = n.
 	\end{cases}
 	\end{equation}
 	(Note that $|\delta| = q^{2\deg(\delta)}$ in this setting.)
 \end{lem}
 
 \begin{proof}
 	The details are very similar to the proof of Lemma \ref{lem:glnq_divisorcount1}, but because of some important changes (e.g. in the degree of relevant $L$-functions) we give a full proof.
 	
 	Let $\Delta = \deg(\delta)$. By the orthogonality relation \eqref{ortho2},
 	\begin{equation}
 	\sumusr_{d: d\delta \in I(f;h)} 1 = \frac{1}{(q^2)^{n-h-1}} \sum_{\chi \in G(R_{n-h-1,1,q^2})} \overline{\chi}(f) \sum_{d \in \MM_{n-\Delta,q^2}^{usr}} \chi(d)\chi(\delta).
 	\end{equation}
 	By Lemma \ref{lem:lfunclu}, $L_{usr}(u,\chi)$ is a polynomial of degree no more than $2(n-h-1)$, and for $\Delta \leq 2h-n+1$ we have $n-\Delta > 2(n-h-1)$, so
 	\begin{equation}
 	\sum_{d \in \MM_{n-\Delta,q^2}^{usr}} \chi(d) = 0, \quad \textrm{for}\; \chi \in G(R_{n-h-1,1,q^2}),\; \chi \neq \chi_0.
 	\end{equation}
 	Hence
 	\begin{equation}
 	\sumusr_{d:\, d\delta \in I(f;h)} 1 = \frac{1}{(q^2)^{n-h-1}} \chi_0(f) \sum_{d \in \MM_{n-\Delta,q^2}^{usr}} \chi_0(d) \chi_0(\delta),
 	\end{equation}
 	which does not depend on $f$ as $\chi_0(f)=1$ for all $f \in \MM_{n,q^2}$. To see \eqref{eq:unqdivisorcount1}, note that because this expression is constant for all $f \in \MM_{n,q^2}$,
 	\begin{equation}
 	\sumusr_{d:\, d\delta \in I(f;h)} 1 = \frac{1}{(q^2)^n} \sum_{f \in \MM_{n,q^2}} \sumusr_{d:\, d\delta \in I(f;h)} 1 = \frac{(q^2)^{h+1}}{(q^2)^n} \sumusr_{d:\, d\delta \in \MM_{n,q^2}} 1,
 	\end{equation}
 	as in passing to the last equality, each $f \in \MM_{n,q^2}$ will have been counted $(q^2)^{h+1}$ times.
 	
 	Finally, in passing to \eqref{eq:unqdivisorcount1prime}, note that for $\delta \in \MM_{\Delta,q^2}^{usr}$,
 	\begin{equation}
 	\sumusr_{d:\, d\delta \in \MM_{n,q^2}} 1 = |\MM^{usr}_{n-\Delta,q^2}| = 
 	\begin{cases}
 	q^{n-\Delta}(1+q^{-1}) & \textrm{for}\; \Delta < n, \\
 	1 & \textrm{for}\; \Delta = n.
 	\end{cases}
 	\end{equation}
 	and \eqref{eq:unqdivisorcount1prime} follows from this.
 \end{proof}
 
 \begin{lem}\label{lem:unq_divisorcount2}
 	For $n > 0$, $n > h \geq (n-1)/2$ and $\delta \in \MM_{q^2}^{usr}$, if $\deg(\delta) \geq 2h-n+2$, then
 	\begin{equation}\label{eq:unq_divisorcount2}
 	\Big|\sumusr_{d:\, d\delta \in I(f;h)} 1 \Big|\leq 1
 	\end{equation}
 	for all $f \in \MM_{n,q^2}$.
 \end{lem}
 
 \begin{proof}
 We are counting elements $d$ in $\MM^{usr}_{n-\deg(\delta),q^2}$ such that $d \delta \equiv f \bmod R_{n-h-1,1,q^2}$, or equivalently, $d \equiv \delta^{-1} f \bmod R_{n-h-1,1,q^2}$. This gives a unique choice for the first $n-h-1 \ge (n-\deg(\delta))/2$ next-to-leading coefficients of $d$, which by definition of unitary self-reciprocals gives at most one option for $d$. 
 \end{proof}
 
 We now have all the tools for Theorem \ref{thm:unveryshort}.
 
 \begin{proof}[Proof of Theorem \ref{thm:unveryshort}]
 	We have
 	\begin{equation}
 	\PP_{M\in \unq}(\chpo(M) \in I(f;h)) = \sumusr_{d,\delta:\, d\delta \in I(f;h)} \frac{P_U^\diamond(\delta)}{|d\delta|^{1/2}}
 	\end{equation}
 	by Lemma~\ref{lem:unq_P_U_iterative}. Letting $H = 2h-n+1$ for notational reasons, this is
 	\begin{equation}\label{eq:un_short_decomp}
 	= \sum_{\deg(\delta) \leq H} \frac{P_U^\diamond(\delta)}{q^n} \sumusr_{d:d\delta \in I(f;h)} 1 + \sum_{H+1 \leq \deg(\delta) \leq n} \frac{P_U^\diamond(\delta)}{q^n} \sumusr_{d:\, d\delta \in I(f;h)} 1.
 	\end{equation}
 	Utilizing \eqref{eq:unqdivisorcount1}, we have for the first term of \eqref{eq:un_short_decomp},
 	\begin{align}
 	\sum_{\deg(\delta) \leq H} \frac{P_U^\diamond(\delta)}{q^n} \sumusr_{d:\, d\delta \in I(f;h)} 1 
 	&= \frac{(q^2)^{h+1}}{(q^2)^n} \sum_{\deg(\delta) \leq H} \frac{P_U^\diamond(\delta)}{q^n} \sumusr_{d:\, d\delta \in \MM_{n,q^2}} 1 \\
 	&= \frac{(q^2)^{h+1}}{(q^2)^n} \sum_{\delta \in \MM_q} \frac{P_U^\diamond(\delta)}{q^n} \sumusr_{d:\, d\delta \in \MM_{n,q^2}} 1 - \frac{(q^2)^{h+1}}{(q^2)^n} \sum_{\deg(\delta) > H} \frac{P_U^\diamond(\delta)}{q^n} \sumusr_{d\delta \in \MM_{n,q^2}} 1 \\
 	&= \frac{(q^2)^{h+1}}{(q^2)^n} \sumusr_{f \in \MM_{n,q^2}} P_U(f) - \frac{1}{q^n} \sum_{n \geq \deg(\delta) > H} P_U^\diamond(\delta) \frac{q^{2(h+1)-n}}{|\delta|^{1/2}}(1 + q^{-1} \mathbf{1}_{\deg(\delta) < n}).
 	\end{align}
 	But $P_U$ is a probability measure on $\MM_{n,q}$ and $P_U^\diamond(f) = (-1)^{\deg(f)} P_U(f)$ and so using this to simplify both terms above, that expression becomes
 	\begin{align}\label{eq:un_hyperbola_term1}
 	&= \frac{(q^2)^{h+1}}{(q^2)^n} - (-1)^{H+1} q^{2(h+1)-2n}\left(\Big(\frac{1}{q^{H+1}} - \frac{1}{q^{H+2}} + \cdots \pm \frac{1}{q^{(n-1)}}\Big)(1+q^{-1}) \mp \frac{1}{q^n}\right) \\
 	&= \frac{(q^2)^{h+1}}{(q^2)^n} -\frac{(-1)^n}{q^n},
 	\end{align}
 	summing the geometric sum and recalling $H = 2h-n+1$.
 	
 	On the other hand, turning to the second term in \eqref{eq:un_short_decomp}, we have from Lemma \ref{lem:unq_divisorcount2},
 	\begin{align}\label{eq:un_hyperbola_term2}
 	\left| \notag \sum_{H+1 \leq \deg(\delta) \leq n} \frac{P_U^\diamond(\delta)}{q^n} \sumusr_{d:\, d\delta \in I(f;h)} 1  \right|
 	& \leq \sum_{H+1 \leq \deg(\delta) \leq n} \frac{|P_U^\diamond(\delta)|}{q^n} \\
 	\notag & \leq \frac{n-H}{q^n} \\
 	&= \frac{2n-2h-1}{q^n}.
 	\end{align}
 	Applying \eqref{eq:un_hyperbola_term1} and \eqref{eq:un_hyperbola_term2} to \eqref{eq:un_short_decomp}, we see that
 	\begin{equation}
 	\PP_{M \in \unq}(\chpo(M) \in I(f;h)) =  \frac{(q^2)^{h+1}}{(q^2)^n} - \frac{(-1)^n}{q^n} + K\cdot\Big(\frac{2n-2h-1}{q^n}\Big),
 	\end{equation}
 	where $K$ is a number in between $-1$ and $1$. This implies Theorem \ref{thm:unveryshort}.
 \end{proof}
 
 From Theorem~\ref{thm:unveryshort} we may immediately deduce an analogous result for traces using Lemma~\ref{lem:symm}.
 \begin{thm}\label{thm:unmanytraces}
 	Let $M \in \unq$ be a random matrix chosen according to Haar measure. Then for $1 \le k \le (n-1)/2$ and for $\{a_i\}_{1 \le i \le k, \, p \nmid i} \subseteq \FF_{q^2}$, we have
 	\begin{equation}\label{eq:unmanytraces}
 	\left| \PP_{M \in \unq}( \forall 1 \le i \le k,\, p \nmid i: \Tr(M^i)=a_i)  - \frac{1}{(q^2)^{k-\lfloor \frac{k}{p} \rfloor}}\right| \leq \frac{2(k+1)q^{2\lfloor \frac{k}{p} \rfloor}}{q^n}.
 	\end{equation}
 \end{thm}

\section{Results for $\spnq$}
In this section, $q$ is an odd prime power. 
\subsection{$\spnq$ and the space of characteristic polynomials}\label{subsec:back_SP}
The symplectic group $\spnq \subseteq \mathrm{GL}(2n,q)$ is defined to be
\begin{equation}
\spnq = \{ g \in \mathrm{GL}(2n,q) : g^t J_{2n} g = J_{2n} \},
\end{equation}
where $J_{2n}$ is any fixed invertible anti-symmetric matrix -- the different choices giving rise to groups that are conjugate in $\mathrm{GL}(2n,q)$.\footnote{For instance, it is typical to take $J_{2n} = \begin{pmatrix} 0 & I_n \\ -I_n & 0\end{pmatrix}$.}
\begin{proposition}\label{prop:SP_size}\cite[Ch.~8]{taylor1992}
	We have $|\spnq| = q^{n^2} \prod_{i=1}^n(q^{2i} - 1).$
\end{proposition}

In order to treat characteristic polynomials of matrices from $\spnq \subseteq \mathrm{GL}(2n,q)$, we require an involution defined on $\MM_{q}^{gl}$. 
\begin{dfn}\label{dfn:f_tilde_sp}
	For $n\geq 0$ and $f \in \MM_{n,q}^{gl}$, define $\overline{f} \in \MM_{n,q}^{gl}$ by
	\begin{equation}
	\overline{f}(T) = \frac{T^n f(T^{-1})}{f(0)}.
	\end{equation}
\end{dfn}
That is, if $f(T) = a_0 + \ldots + a_{n-1} T^{n-1} + T^n$, we have
\begin{equation}\label{eq:sr_explicitdesc}
\overline{f}(T) = \frac{1}{a_0} + \frac{a_{n-1}}{a_0} T + \ldots + T^n.
\end{equation}

\begin{dfn}\label{dfn:sr_class}
	We say that a polynomial $f \in \MM_{n,q}^{gl}$ is \emph{self-reciprocal} if $f = \overline{f}$, and write
	\begin{equation}
	\MM_{n,q}^{sr} = \{f \in \MM_{n,q}^{gl}: f = \overline{f}\},
	\end{equation}
	and $\MM_{q}^{sr} = \cup_{n\geq 0} \MM_{n,q}^{sr}$. Additionally, let
	\begin{align}
	 \MM_{2n, q}^{sr,0} &=  \{ f \in \MM_{2n,q}^{sr} : (f,(T-1)(T+1))=1 \},\\
	 \MM_{2n,q}^{sr,1} &=  \{ f \in \MM_{2n,q}^{sr} : f(0)=1 \},
	\end{align}
	and write $\MM_{q}^{sr,i} = \cup_{n \ge 0} \MM_{2n,q}^{sr,i}$ for $i=0,1$.

\end{dfn}

By definition, if $M \in \spnq$ we have $M=J_{2n}^{-1}(M^t)^{-1}J_{2n}$, and so $$\chpo(M) = \chpo((M^t)^{-1}) = \overline{\chpo(M)},$$ so $\chpo(M)$ is self-reciprocal. Moreover, $\spnq \subseteq \mathrm{SL}(2n,q)$ \cite[Cor.~6]{taylor1992}. Thus, we have the inclusion $\{ \chpo(M):\, M \in \spnq\}\subseteq  \MM_{2n,q}^{sr,1}$. In fact, from Theorem~\ref{thm:spnq_charpoly} below, these two sets are equal.

\begin{proposition}\label{prop:sr_count}
	We have $|\MM_{n,q}^{sr}| = q^{\lfloor \frac{n}{2} \rfloor} + q^{\lfloor \frac{n-1}{2} \rfloor}$, $|\MM_{2n,q}^{sr,1}| = q^n$ and $|\MM_{2n,q}^{sr,0}| = q^n-2q^{n-1} + q^{n-2} \mathbf{1}_{n > 1}$ for $n \ge 1$. 
	\end{proposition}
\begin{proof}
	The first part is clear by analyzing explicit description \eqref{eq:sr_explicitdesc} of $\overline{f}$, considering $n$ odd and $n$ even separately. The second part is similar. As $T\pm 1$ are self-reciprocal, the third part is solved by inclusion-exclusion, namely $|\MM_{2n,q}^{sr,0}| = |\MM_{2n,q}^{sr}|-2|\MM_{2n-1,q}^{sr}|+|\MM_{2n-2,q}^{sr}|$.
\end{proof}

From Definition \ref{dfn:f_tilde_sp}, one sees for $f,g \in \MM_{q}^{gl}$ that $\overline{fg} = \overline{f} \cdot \overline{g}$ and that $\deg(f)=\deg(\overline{f})$. As a consequence of this it follows that $\MM_{q}^{sr}$ and $\MM_{q}^{sr,i}$ are submonoids of $\MM_{q}^{gl}$ under multiplication. Furthermore a polynomial $P \in \MM_{q}^{gl}$ is irreducible if and only if $\overline{P}$ is irreducible. 

In fact we have

\begin{thm}\label{thm:sr_factorization}
	Every $f \in \MM_{q}^{sr}$ factorizes uniquely into irreducibles as
	\begin{equation}\label{eq:f sr fact}
	f = (T-1)^a (T+1)^b \prod_{i=1}^r P_i^{e_i} \prod_{j=1}^s (Q_j \overline{Q_j})^{e'_j},
	\end{equation}
	where $P_i = \overline{P_i}$ for all $i$ and $Q_j \neq \overline{Q_j}$ for all $j$, and $P_i, Q_j \in\MM_{q} \setminus \{T,T-1,T+1\}$ are irreducible for all $i,j$. Conversely, every such product is in $\MM_{q}^{sr}$.
\end{thm}
\begin{proof}
	The proof is completely analogous to that of Theorem~\ref{thm:usr_factorization}.
\end{proof}

For an element of $\MM_{q}^{sr}$ to be irreducible imposes a restriction on its possible degree, as shown e.g. in \cite{carlitz1967} and in \cite[Thm.~11]{fulman1999}. We give a different proof of this, with additional information, in the following theorem.
\begin{thm}\label{thm:sr_primes_odd}
	If $P \in \MM_{q}^{sr}$ is an irreducible polynomial, then either $P \in \{T-1,T+1\}$, or $\deg(P)$ is even. In the later case, $P(0)=1$.
	
	If $Q \in \MM_{q}^{gl} \setminus \MM_{q}^{sr}$ is an irreducible polynomial, then $Q \overline{Q} \in \MM_{q}^{sr}$ and $(Q \overline{Q})(0)=1$.
\end{thm}
\begin{proof}
	We may write $P(T) = \prod_{i=0}^{\deg(P)-1} (T - z_i)$ for some distinct $z_i$ in the algebraic closure. Since $\overline{P} = P$, this means that $t\mapsto t^{-1}$ is a permutation on the roots of $P$. If $P(T) \neq T\pm 1$ then there are no fixed points for the permutation induced by $t\mapsto t^{-1}$, and so its cycle structure consists of $\deg(P)/2$ transpositions, forcing $\deg(P)$ to be even. In this case, assuming $z_i^{-1}=z_{\deg(P)-i-1}$ w.l.o.g, $P(0) = \prod_{i=0}^{\deg(P)-1} z_i = \prod_{i=0}^{\frac{\deg(P)}{2}-1} (z_i z_i^{-1}) = 1$.
	
	Since $f\mapsto \overline{f}$ is an involution, clearly $Q \overline{Q} \in \in \MM_{q}^{sr}$. Again, since $t\mapsto t^{-1}$ permutes the roots of $Q \overline{Q}$ and has no fixed points, we find that $(Q \overline{Q})(0)=1$.
\end{proof}
By comparing degrees and constant coefficients in \eqref{eq:f sr fact}, we obtain the following corollary.
\begin{cor}\label{cor:sr01}
In the notation of Theorem~\ref{thm:sr_factorization}, $f \in \MM_{q}^{sr}$ is in $\MM_{q}^{sr,1}$ if and only if $a$ and $b$ are both even, and $f \in \MM_{q}^{sr}$ is in $\MM_{q}^{sr,0}$ if and only if $a$ and $b$ are both $0$. In particular, $\MM_{q}^{sr,0} \subseteq \MM_{q}^{sr,1}$.
\end{cor}

\subsection{Expressions for $P_{Sp}(f)$}

Throughout this subsection we will take random $M \in \spnq$ according to Haar measure (that is, uniform measure). 

\begin{dfn}
	For $f \in \MM_q$, we define the arithmetic function
	\begin{equation}\label{eq:P_SP_def}
	P_{Sp}(f) = \PP_{M \in \spnq}(\chpo(M) = f) = \frac{|\{M \in \spnq:\, \chpo(M) = f\}|}{|\spnq|}
	\end{equation}
	where we take $2n = \deg(f)$ if $\deg(f)$ is even, and where we set $P_{Sp}(f)=0$ if $\deg(f)$ is odd. For $f=1$ we define $P_{Sp}(1)=1$.
\end{dfn}
Fulman \cite[Thm.~19]{fulman1999} has proved a closed formula for the count $|\{M \in \spnq:\, \chpo(M) = f\}|$. We have found it natural to phrase his formula in the language of probability as above.
\begin{thm}[Fulman]\label{thm:spnq_charpoly}
	Suppose $f \in \MM_{q}^{sr,1}$ factorizes as
	\begin{equation}
	f = (T-1)^{2a} (T+1)^{2b} \prod_{i=1}^r P_i^{e_i} \prod_{j=1}^s (Q_j \overline{Q_j})^{e'_j},
	\end{equation}
	with $P_i = \overline{P}_i$ for all $i$ and $Q_j \neq \overline{Q_j}$ for all $j$, with $P_i, Q_j \in \MM_{q} \setminus \{T,T-1,T+1\}$ irreducible for all $i,j$. Set 
	\begin{equation}
	m_i = \deg(P_i), \quad m'_j = \deg(Q_j).
	\end{equation} 
	Then
	\begin{equation}
	P_{Sp}(f) = \frac{q^{2a^2}}{|\mathrm{Sp}(2a,q)|} \frac{q^{2b^2}}{|\mathrm{Sp}(2b,q)|}\prod_{i=1}^r \frac{q^{\frac{m_i e_i(e_i-1)}{2}}}{|\mathrm{U}(e_i, q^{\frac{m_i}{2}})|} \prod_{j=1}^s \frac{q^{m'_j e'_j(e'_j-1)}}{|\mathrm{GL}(e'_j, q^{m'_j})|}.
	\end{equation}
	If $f$ is a monic polynomial but $f \notin \MM_{q}^{sr,1}$, then $P_{Sp}(f) = 0$.
\end{thm}

We observe that $P_{Sp}(f)$ can be written as an infinite Dirichlet convolution of simple arithmetic functions defined on $\MM_{q}$.

\begin{thm}\label{thm:Sp_conv}
Let $\lambda \colon \MM_{q} \to \{-1,1\}$ be the Liouville function, namely, the unique completely multiplicative function with $\lambda(P)=-1$ for every irreducible $P$ in $\MM_{q}$. For $g \in \MM_{q}$, define the arithmetic functions
	\begin{align*}
	\beta^{Sp}(g) &= |g|^{1/2} \mathbf{1}_{\MM_{q}^{sr,1}}(g), \\
	\gamma^{Sp}(g) &= \lambda(g) \mathbf{1}_{\MM_{q}^{sr,0}}(g),
	\end{align*} 
	and further define
	\begin{equation}\label{eq:alpha_SP}
	\alpha_i^{Sp}(g) =  \frac{1}{|g|^i} (\beta^{Sp}\ast\gamma^{Sp})(g)
	\end{equation}
	for $i \ge 1$. Then for $f \in \MM_{q}$,
	\begin{equation}\label{eq:SP_conv}
	P_{Sp}(f) = \lim_{k\rightarrow\infty} (\alpha_1^{Sp}\ast\alpha_2^{Sp}\ast \cdots \ast \alpha_k^{Sp}) (f).
	\end{equation}
\end{thm}

The proof of Theorem~\ref{thm:Sp_conv} is given below. It is a corollary of the following theorem:

\begin{thm}\label{thm:L_Sp_expansion}
	For a Hayes character $\chi$, define
	\begin{equation}\label{eq:L_sr_chi}
	L_{sr,i}(u,\chi) = \sum_{f \in \MM_{q}^{sr,i}} \chi(f) u^{\deg(f)},
	\end{equation}
	\begin{equation}\label{eq:L_sr_chilambda}
	L_{sr,i}(u,\chi\lambda) = \sum_{f \in \MM_{q}^{sr,i}} \chi(f) \lambda(f) u^{\deg(f)}
	\end{equation}
	for $i=0,1$, and
	\begin{equation}\label{eq:L_Sp_def}
	L_{Sp}(u,\chi) = L_{sr,1}(\sqrt{q}u,\chi) L_{sr,0}(u,\chi\lambda).
	\end{equation}
	The sums \eqref{eq:L_sr_chi} and \eqref{eq:L_sr_chilambda} converge absolutely for $|u| < 1/\sqrt{q}$, and for $|u| < 1$ we have 
	\begin{equation}\label{eq:P_Sp_generating}
	\sum_{f \in \MM_{q}} P_{Sp}(f) \chi(f) u^{\deg(f)} = \prod_{i=1}^\infty L_{Sp}\Big(\frac{u}{q^{i}},\chi\Big),
	\end{equation}
	with both the left hand sum and the right hand product converging absolutely.
\end{thm}

\begin{proof}
	We treat claims about convergence first. For \eqref{eq:L_sr_chi} and \eqref{eq:L_sr_chilambda}, note that $\sum_{f \in \MM_{n,q}^{sr}} |\chi(f)| \le 2q^{\frac{n}{2}}$ from Prop.~\ref{prop:sr_count}, which furthermore implies $|L_{sr,1}(u,\chi)|,|L_{sr,0}(u,\lambda\chi)| \le 1/(1-|\sqrt{q}u|)^2$. For the convergence of the left hand side of \eqref{eq:P_Sp_generating} note that $\sum_{f \in \MM_{n,q}} |P_{Sp}(f)| \leq 1$, and for the convergence of the right hand side note that $|L_{Sp}(u/q^{i})| \le 1/(1-|u/q^{i-1}|)^4$.
	
	We have the Euler products
	\begin{align}
	\label{eq:LChi_sr_Euler}
	L_{sr,1}(u,\chi) &= \frac{1}{1-\chi^2(T-1)u^2} \frac{1}{1-\chi^2(T+1)u^2}\\
	&\qquad \times \prod_{\substack{P \neq T,T\pm 1 \\ P = \overline{P}}} \frac{1}{1-\chi(P) u^{\deg(P)}} \prod_{Q \neq \overline{Q}} \frac{1}{1 - \chi(Q\overline{Q}) u^{2\deg(Q)}},\\
	\label{eq:LChilambda_sr_Euler}
	L_{sr,0}(u,\chi\lambda) &= \prod_{\substack{P \neq T,T\pm 1 \\ P = \overline{P}}} \frac{1}{1+\chi(P) u^{\deg(P)}} \prod_{Q \neq \overline{Q}} \frac{1}{1 - \chi(Q\overline{Q}) u^{2\deg(Q)}}
	\end{align}
	following from Theorem~\ref{thm:sr_factorization}, Corollary~\ref{cor:sr01} and the fact that $\lambda(P) = -1$ and $\lambda(Q \overline{Q}) = 1$ for all $P, Q$. Here the products are respectively over all irreducible monics $P$ such that $P = \overline{P}$ and $P \neq T,T\pm 1$, and likewise over irreducible monics $Q$ such that $Q \neq \overline{Q}$. By using the facts that $|\{P \in \MM_{n,q}^{sr}:\, P \textrm{ irreducible}\}| \leq |\MM_{n,q}^{sr}|\leq 2q^{\frac{n}{2}}$ and $|\{Q \in \MM_{n,q}:\, Q\neq \overline{Q}, \, Q\textrm{ irreducible}\}|\leq |\MM_{n,q}| \leq q^{n}$, it can be seen that the products in \eqref{eq:LChi_sr_Euler} and \eqref{eq:LChilambda_sr_Euler} converge absolutely for $|u| < 1/\sqrt{q}$. Hence
	\begin{multline}
	L_{Sp}\Big(\frac{u}{q^{i}},\chi\Big) = \frac{1}{1-\chi^2(T-1)u^2/q^{2i-1}} \frac{1}{1-\chi^2(T+1)u^2/q^{2i-1}}\\
	\times \prod_{\substack{P\neq T,T\pm 1 \\ P = \overline{P}}} \left(\frac{1}{1 - \chi(P)u^{\deg(P)}/|P|^{(2i-1)/2}} \cdot \frac{1}{1 + \chi(P) u^{\deg(P)}/|P|^{i}}\right) \\
	\times \prod_{Q \neq \overline{Q}} \left( \frac{1}{1-\chi(Q\overline{Q})(u^{\deg(Q)}/|Q|^{(2i-1)/2})^2} \cdot \frac{1}{1 - \chi(Q\overline{Q}) (u^{\deg(Q)}/|Q|^{i})^2}\right),
	\end{multline}
	and
	\begin{multline}\label{eq:L_Sp_Product_Euler}
	\prod_{i=1}^\infty L_{Sp}\Big(\frac{u}{q^{i}},\chi\Big) = \prod_{j=1}^{\infty} \left(\frac{1}{1-\chi^2(T-1)u^2/q^{2j-1}} \frac{1}{1-\chi^2(T+1)u^2/q^{2j-1}}\right)\\
	\prod_{\substack{P\neq T,T\pm 1 \\ P = \overline{P}}} \prod_{j=1}^\infty \frac{1}{1 + \chi(P) u^{\deg(P)} (-|P|^{1/2})^{-j}} \prod_{Q \neq \overline{Q}} \prod_{j=1}^\infty \frac{1}{1 - \chi(Q\overline{Q}) (u^{\deg(Q)} |Q|^{-j/2})^2}.
	\end{multline}
	But from Theorem \ref{thm:spnq_charpoly}, one sees that $P_{Sp}((T-1)^{2a}(T+1)^{2b}\prod P_i^{e_i} \prod (Q_j \overline{Q}_j)^{e'_j}) = P_{Sp}((T-1)^{2a}) P_{Sp}((T+1)^{2b})\prod P_{Sp}(P_i^{e_i}) \prod P_{Sp}((Q_j \overline{Q}_j)^{e'_j})$, with
	\begin{align}
	P_{Sp}((T\pm 1)^{2e}) &= \frac{q^{e^2}}{(q^{2e}-1)(q^{2(e-1)}-1)\cdots (q^2-1)},\\
	P_{Sp}(P^e) &= \frac{|P|^{e(e-1)/4}}{\prod_{i=1}^e (|P|^{i/2} - (-1)^i)},\\
	P_{Sp}((Q\overline{Q})^e) &= \frac{|Q|^{e(e-1)}}{(|Q|^{e}-1)(|Q|^{e}-|Q|)\cdots (|Q|^{e}-|Q|^{e-1})},
	\end{align}
	for $P = \overline{P} \neq T\pm 1$ and $Q\neq \overline{Q}$. Hence
	\begin{multline}
	\sum_{f \in \MM_{q}} P_{Sp}(f)\chi(f) u^{\deg(f)} = \\
	\Big(1 + \sum_{e=1}^\infty \frac{q^{e^2}}{\prod_{i=1}^e (q^{2i}-1)} (\chi(T-1) u)^{2e}\Big) \Big(1 + \sum_{e=1}^\infty \frac{q^{e^2}}{\prod_{i=1}^e (q^{2i}-1)} (\chi(T+1) u)^{2e}\Big)\\
	\times \prod_{\substack{P\neq T, T \pm 1 \\ P = \overline{P}}} \Big(1 + \sum_{e=1}^\infty \frac{|P|^{e(e-1)/4}}{\prod_{i=1}^e (|P|^{i/2} - (-1)^i)} (\chi(P) u^{\deg(P)})^e\Big) \\
	\times \prod_{Q\neq \overline{Q}} \Big(1+ \sum_{e=1}^\infty \frac{|Q|^{e(e-1)}}{(|Q|^{e}-1)\cdots (|Q|^{e}-|Q|^{e-1})} (\chi(Q\overline{Q}) u^{2\deg(Q)})^e\Big).
	\end{multline}
	By \eqref{eq:swapped_infinity_numden} with $(y,V)\in \{(q (\chi(T\pm 1)u)^2,q^2), (-\chi(P)u^{\deg(P)}, -|P|^{1/2})\}$ and \eqref{eq:swapped_infinity} with $y=\chi(Q\overline{Q})u^{2\deg(Q)}$, $V=|Q|$, this agrees with the right hand side of \eqref{eq:L_Sp_Product_Euler}.
\end{proof}

\begin{proof}[Proof of Theorem \ref{thm:Sp_conv}]
	We proceed as in the proof of Theorem \ref{thm:GL_conv}. We have from Theorem \ref{thm:L_Sp_expansion} that
	\begin{equation}
	\sum_{f \in \MM_{q}} P_{Sp}(f) \chi(f) u^{\deg(f)} = \lim_{k\rightarrow\infty} \sum_{f \in \MM_{q}} 
	(\alpha_1^{Sp}\ast\cdots \ast \alpha_k^{Sp})(f) \chi(f) u^{\deg(f)}.
	\end{equation}
	And exactly as in the proof of Theorem~\ref{thm:GL_conv}, we may swap the order of limit and summation here and obtain the desired result by using Lemma~\ref{lem:uniqueness}.
\end{proof}

\subsection{The zeros of $L_{sr,1}(u,\chi)$}

We examine the $L$-function $L_{sr,1}(u,\chi)$ defined in \eqref{eq:L_sr_chi}.
\begin{lem}\label{lem:map F}
	Consider the homomorphism $\gtosym \colon \MM_q \to \MM_{q}^{sr}$ defined by $\gtosym(g(T)) = g(T+\frac{1}{T}) T^{\deg(g)}$. 
	\begin{enumerate}
		\item $\deg (\gtosym(g)) = 2 \deg (g)$.
		\item $\gtosym$ is injective.
		\item The $k$-th next-to-leading coefficient of $\gtosym(g)$ (for $1 \le k \le \deg (g)$) is a linear function of the first $k$ next-to-leading coefficients of $g$.
		\item The image of $\gtosym$ is $\MM_q^{sr,1}$. 
	\end{enumerate}
\end{lem}
\begin{proof}
The first two parts are simple, and the third part follows by expanding $g(T+\frac{1}{T})$. For the fourth part, note that if $g(T)=\prod_{i=1}^{\deg(g)} (T-\alpha_i)$ then $\gtosym(g)= \prod_{i=1}^{\deg(g)} (T^2-\alpha_i T + 1)$, which is self-reciprocal of even degree, and moreover, the multiplicity of $T\pm 1$ in it is even, as these factors may only come from $\alpha_i = \pm 2$. Hence $\gtosym(g) \in \MM_{2\deg(g),q}^{sr,1}$. A counting argument using Prop.~\ref{prop:sr_count} shows that $\gtosym(\MM_{n,q}) = \MM_{2n,q}^{sr,1}$.
\end{proof} 
Let us denote the inverse of $\gtosym$ by $\symtog$.
\begin{proposition}\label{prop:rh_sp}
	Given a short interval character of $\ell$ coefficients $\chi$, there is a unique short interval character of $\ell$ coefficients $\chi'$ such that  $\chi(f)=\chi'(\symtog(f))$ for all $f \in \MM^{sr,1}_q$. Moreover,
	\begin{equation}
	L_{sr,1}(u,\chi) = L(u^2,\chi').
	\end{equation}
	Moreover, $\chi$ is non-trivial if and only if $\chi'$ is non-trivial. Here $L(u^2,\chi')$ refers to the usual $L$-function, defined in \S\ref{sec:lfunc}. 
\end{proposition}
\begin{proof}
	The condition $\chi(f)=\chi'(\symtog(f))$ for all $f \in \MM^{sr,1}_q$ is the same as
	\begin{equation}\label{eq:chiprime}
	\chi'(g(T))=\chi(g(T+\frac{1}{T})T^{\deg (g)})
	\end{equation}
	for all $g \in \MM_q$. This definition determines $\chi'$ uniquely, and we get for free the multiplicativity of $\chi'$ and the fact that $\chi'(1)=1$. To see that $\chi'$ is well-defined, one uses Lemma~\ref{lem:map F} which tells us that the first $\ell$ next-to-leading coefficients of $g(T+\frac{1}{T})T^{\deg (g)}$ depend only on the first $\ell$ next-to-leading coefficients of $g$. This shows that $\chi'(g)$ depends only on the first $\ell$ next-to-leading coefficients of $g$.
	
	If $\chi'$ is trivial then $\chi(f)=1$ for all $f \in \MM^{sr,1}_q$. Since for any given $\ell$ next-to-leading coefficients $a_1,a_2,\ldots,a_{\ell}$ one may find a polynomial $f$ in $\MM^{sr,1}_q$ with such next-to-leading coefficients (e.g. $f=T^{2\ell} + a_1 T^{2\ell-1} + \ldots + a_{\ell} T^{\ell} + a_{\ell-1} T^{\ell-1} + \ldots + 1$), it follows that $\chi$ is trivial as well. The other direction is immediate.
	
	By the last part of Lemma~\ref{lem:map F},
	\begin{equation}
	L_{sr,1}(u,\chi) = \sum_{f \in \MM^{sr,1}_q} \chi(f) u^{\deg(f)} = \sum_{f \in \MM^{sr,1}_q} \chi'(\symtog(f)) u^{2\deg(\symtog(f))} = \sum_{g \in \MM_q} \chi'(g) u^{2\deg (g)} = L(u^2,\chi'),
	\end{equation}
	as needed.	
\end{proof}
In particular, the Riemann Hypothesis over Function Fields gives us the following corollary:
\begin{cor}\label{cor:rh lsp 1}
Let $\chi$ 
be a non-trivial short interval character of $\ell$ coefficients. Then
\begin{equation}\label{lspnqfactor}
L_{sr,1}(u,\chi) = \prod_{i=1}^{\deg (L_{sr,1}(u,\chi))/2} (1-\gamma_i' u^2),
\end{equation}
with $|\gamma_i'| \in \{ \sqrt{q},1\}$ for all $i$ and $\deg(L_{sr,1}(u,\chi)) \leq 2(\ell-1)$. (Note that $\gamma_i'$ will have a dependence on $\chi$.)
\end{cor}
\subsection{$\lambda(f)$ and $L_{sr,0}(u,\chi\lambda)$}

We now turn to $L_{sr,0}(u,\chi\lambda)$.
\begin{dfn}
Define the quadratic character $\chi_{sr}\colon  \FF_q[T] \to \{-1,0,1\}$ by 
\begin{equation}\label{def:chisr}
\chi_{sr}(f) = \left( \frac{f}{T^2-4}\right),
\end{equation}
where $\left( \frac{\cdot}{\cdot} \right)$ is the Jacobi symbol.
\end{dfn}
Set
\begin{equation}
S = \{f \in \MM_{q} : (f,(T-2)(T+2))=1\}.
\end{equation}
The following lemma is due to Carlitz \cite{carlitz1967}. We give an independent proof.
\begin{lem}[Carlitz]\label{lem:cool iden}
For all  $g \in S$, we have
	\begin{equation}\label{eq:cool iden}
	\lambda(\gtosym(g)) =\chi_{sr}(g(T)).
	\end{equation} 

\end{lem}
\begin{proof}
Any element of $S$ can be expressed as a product of squarefree elements of $S$. Since both sides of \eqref{eq:cool iden} are completely multiplicative, it suffices to establish \eqref{eq:cool iden} for squarefree $g \in S$. 

One may verify that $\gtosym|_{S}$ takes squarefree polynomials to squarefree polynomials. A formula due to Pellet \cite[Lem.~4.1]{conrad2005} gives
\begin{equation}\label{eq:pellet}
\lambda(f(T))= (-1)^{\deg (f)} \chi_2(\disc(f)),
\end{equation}
for any squarefree $f$, where $\disc(f)$ is the discriminant of $f$, and $\chi_2$ is the unique non-trivial quadratic character of $\FF_q^{\times}$, extended to $\FF_q$ by $0$. As $\deg (\gtosym(g)) = 2\deg (g)$, we have
\begin{equation}
\lambda(\gtosym(g))= (-1)^{\deg (\gtosym(g))} \chi_2(\disc(\gtosym(g))) = \chi_2(\disc(\gtosym(g))),
\end{equation}
and so it suffices to show that
\begin{equation}
\chi_2(\disc(\gtosym(g))) = \chi_{sr}(g)
\end{equation}
for squarefree $g$. Writing $g$ as $g(T) = \prod_{i=1}^{n} (T-\beta_i)$, we have $\gtosym(g) = \prod_{i=1}^{n} (T-\alpha_i)(T-\alpha_i^{-1})$ where $\alpha_i$ are defined via $\alpha_i+\alpha_i^{-1} = \beta_i$, and so
\begin{equation}
\disc(\gtosym(g)) = \prod_{i<j}(\alpha_i-\alpha_j)^2 \prod_{i<j}(\alpha^{-1}_i-\alpha^{-1}_j)^2 \prod_{i,j}(\alpha_i-\alpha^{-1}_j)^2 = A^2 \cdot B,
\end{equation}
where
\begin{equation}
A = \prod_{i<j}(\alpha_i-\alpha_j)^2  \prod_{i<j}(\alpha_i \alpha_j -1)^2 \prod_{i<j} (\alpha_i \alpha_j)^{-2}, \, B = \prod_{i}(\alpha_i -\alpha_i^{-1})^2.
\end{equation}
Since $\beta_{i}-\beta_j = (\alpha_i-\alpha_j)(\alpha_i \alpha_j-1) (\alpha_i \alpha_j)^{-1}$ and $(2-\beta_i)(-2-\beta_i) = (\alpha_i-\alpha_i^{-1})^2$, we have
\begin{equation}
A = \disc(g), \, B=  g(2)g(-2),
\end{equation}
and so $\chi_2(\disc(\gtosym(g))) = \chi_2( \disc(g)^2) \chi_2(g(2)g(-2)) = \chi^2_2( \disc(g)) \chi_2(g(2)g(-2)) = \chi_{sr}(g) $, as needed. 
\end{proof}
\begin{lem}\label{lem:rest s}
The restriction $\gtosym|_{S}$ is a bijection from $S$ to $\MM_q^{sr,0}$. 
\end{lem}
\begin{proof}
As in the proof of the last part of Lemma~\ref{lem:map F}, one can see that $\gtosym(S \cap \MM_{n,q}) \subseteq \MM^{sr,0}_{2n,q}$, and a counting argument using Prop.~\ref{prop:sr_count} implies equality.
\end{proof}
\begin{proposition}\label{prop:rh_sp2}
In the notation of Prop.~\ref{prop:rh_sp} and \eqref{def:chisr},
	\begin{equation}
	L_{sr,0}(u,\chi \lambda) = L(u^2,\chi' \chi_{sr}).
	\end{equation}
\end{proposition}
\begin{proof}
By Lemmas~\ref{lem:cool iden} and \ref{lem:rest s},
\begin{multline}
L_{sr,0}(u,\chi \lambda) = \sum_{f \in \MM^{sr,0}_q} \chi(f) \lambda(f) u^{\deg(f)} = \sum_{f \in \MM^{sr,0}_q} \chi'(\symtog(f)) \lambda(f) u^{2\deg(\symtog(f))} \\
= \sum_{g \in S} \chi'(g) \lambda(\gtosym(g)) u^{2\deg (g)} = \sum_{g \in \MM_q} \chi'(g) \chi_{sr}(g) u^{2\deg (g)} = L(u^2,\chi' \chi_{sr}),
\end{multline}
as needed.
\end{proof}
In particular, the Riemann Hypothesis over Function Fields gives us the following corollary:
\begin{cor}\label{cor:rh lsp 2}
Let $\chi$ 
be a short interval character of $\ell$ coefficients. Then
\begin{equation}\label{lspnqfactor2}
L_{sr,0}(u,\chi \lambda) = \prod_{i=1}^{\deg (L_{sr,0}(u,\chi \lambda))/2} (1-\gamma_i^\circ u^2),
\end{equation}
with $|\gamma_i^\circ| \in \{ \sqrt{q},1\}$ for all $i$ and $\deg (L_{sr,0}(u,\chi \lambda)) \leq 2(\ell+1)$. (Note that $\gamma_i^\circ$ will have a dependence on $\chi$.)
\end{cor}

\subsection{Character sums over $\spnq$}
\begin{thm}\label{thm:exposp}
	Let $n$ be a positive integer. Let $\chi$ be a non-trivial short interval character of $\ell$ coefficients. We have (for $L_{Sp}(u,\chi)$ defined in \eqref{eq:L_Sp_def})
	\begin{equation}
	L_{Sp}(u,\chi) = \prod_{i=1}^{d} (1-\gamma_i u^2)
	\end{equation}
with
	\begin{equation}\label{eq:d symp bound}
	d:=\frac{\deg (L_{Sp}(u,\chi))}{2} \le 2\ell.
	\end{equation}
	\begin{enumerate}
		\item We have the identity
		\begin{equation}\label{eq:sp2ndesirediden}
		\frac{1}{\left|\spnq\right|}\sum_{M \in \spnq} \chi(\chpo(M)) = (-1)^n \sum_{a_1+\ldots+a_d=n}\prod_{j=1}^{d} \frac{\gamma_j^{a_j}}{(q^{2a_j}-1)\cdots (q^2-1)}.
		\end{equation}
		\item If $d>0$ then the following estimate holds:
		\begin{equation}\label{eq:sp2nestchar}
		\left| \frac{1}{\left|\spnq\right|}\sum_{M \in \spnq} \chi(\chpo(M))\right| \le q^{-\frac{n^2}{d}} q^{\frac{n}{2}}(1+\frac{1}{q^2-1})^n \binom{n+d-1}{n}.
		\end{equation}
	\end{enumerate}
\end{thm}
\begin{proof}
		By Corollaries~\ref{cor:rh lsp 1} and \ref{cor:rh lsp 2} we have \eqref{eq:d symp bound}. The rest of the proof proceeds in the same way as the proof of Theorem~\ref{thm:expo} with two minor differences: when applying \eqref{eq:swapped_infinity2} to deduce \eqref{eq:sp2ndesirediden}, let $y = \gamma_j u$ and $V = q^2$, and in the passage from \eqref{eq:sp2ndesirediden} to \eqref{eq:sp2nestchar}, we use the bound $|\gamma_i| \le q\sqrt{q}$ which follows from Corollaries~\ref{cor:rh lsp 1} and \ref{cor:rh lsp 2}.
\end{proof}

\subsection{The distribution of the characteristic polynomial: superexponential bounds for $\spnq$}

\begin{thm}\label{thm:apshortspn}
	Let $M \in \spnq$ be a random matrix chosen according to Haar measure. Then for $0 \leq h < 2n-1$ and for $f \in \MM_{n,q}$,
	\begin{equation}\label{eq:spnshort}
	\left| \PP_{M \in \spnq} (\chpo(M) \in I(f;h) ) - \frac{q^{h+1}}{q^{2n}} \right| \leq q^{-\frac{n^2}{4n-2h-2}} q^{\frac{n}{2}}(1+\frac{1}{q^2-1})^n \binom{5n-2h-3}{n}.
	\end{equation}
\end{thm}

\begin{proof}[Proof of Theorem~\ref{thm:apshortspn}]
	As in the proof of Theorem~\ref{thm:apshort}, from \eqref{eq:ortho charpoly} we have
	\begin{equation}\label{eq:spshortint}
	\PP_{M\in\spnq}(\chpo(M) \in I(f;h) ) = q^{-(2n-h-1)} \left(1+\sum_{\substack{\chi \in G(R_{2n-h-1,1,q})\\\chi\neq \chi_0}} \overline{\chi}(f) \frac{\sum_{M \in \spnq}\chi(\chpo(M))}{\left|\spnq\right|}\right).
	\end{equation}
	When $n$ is fixed, the right hand side of \eqref{eq:sp2nestchar} is monotone increasing in $d$. From \eqref{eq:spshortint} and \eqref{eq:sp2nestchar}, we have
	\begin{equation}\label{eq:spnshortwithd}
	\begin{split}
	\left|\PP_{M\in\spnq}(\chpo(M) \in I(f;h) ) - q^{-(2n-h-1)}\right| & \le q^{-(2n-h-1)} \sum_{\substack{\chi \in G(R_{2n-h-1,1,q})\\\chi\neq\chi_0}} \left| \frac{\sum_{M \in \spnq}\chi(\chpo(M))}{\left|\spnq\right|} \right| \\
	&\le q^{-\frac{n^2}{d_1}}q^{\frac{n}{2}}(1+\frac{1}{q^2-1})^n \binom{n+d_1-1}{n},
	\end{split}
	\end{equation}
	where 
	\begin{equation}d_1 = \max_{\substack{\chi \in G(R_{2n-h-1,1,q})\\\chi\neq \chi_0}} \deg (L_{Sp}(u,\chi))/2.
	\end{equation}
	By \eqref{eq:d symp bound} we have $d_1 \le 2(2n-h-1)$ which concludes the proof. 
\end{proof}

\subsection{The distribution of traces: superexponential bounds for $\spnq$}
\begin{thm}\label{thm:arbexpspn}
	Let $M \in \spnq$ be a random matrix chosen according to Haar measure. Fix a strictly increasing sequence $b_1,\ldots,b_k$ of positive integers coprime to $p$. Then for any sequence $a_1,...,a_k$ of elements of elements from $\FF_{q}$, we have
	\begin{equation}\label{eq:probsboundedspn}
	\left|\PP_{M \in \spnq}(\forall 1 \le i \le k: \Tr(M^{b_i}) = a_i) - q^{-k}\right| \le q^{-\frac{n^2}{2 b_k}} q^{\frac{n}{2}} (1+\frac{1}{q^2-1})^n \binom{n-1+2b_k}{n}. 
	\end{equation}
\end{thm}
The proof is omitted, as it follows closely the proof of Theorem~\ref{thm:arbexp}.

\subsection{The characteristic polynomial in very short intervals for $\spnq$}

Finally we treat the case of short intervals which are sufficiently small that they contain few elements of $\MM_{n,q}^{sr,1}$. In this space a polynomial is determined by only half its coefficients, and for this reason short intervals may contain no elements.

Indeed, from an analysis of the explicit formula \eqref{eq:sr_explicitdesc} for $\overline{f}$ one may prove

\begin{proposition}\label{prop:srcount_in_shortint}
For any $f \in \MM_{2n,q}$ and $h < 2n$, we have
\begin{equation}
|\MM_q^{sr,1} \cap I(f;h)| = 
\begin{cases}
q^{h+1-n} & \textrm{if}\; h\geq n, \\
1 & \textrm{if}\; h \leq n-1 \;\textrm{and}\; f \in \MM_q^{sr,1}, \\
0 & \textrm{if}\; h \leq n-1 \;\textrm{and}\; f \notin \MM_q^{sr,1}.
\end{cases}
\end{equation}
\end{proposition}

Thus for $h$ smaller than $n$ one cannot expect characteristic polynomials to equidistribute in short intervals $I(f;h)$. For $h$ only slightly larger than $n$ however one may show that equidistribution occurs. 

\begin{thm}
\label{thm:spqveryshort}
Let $M \in \spnq$ be a random matrix chosen according to Haar measure. Then for $n \leq h < 2n$ and $f \in \MM_{2n,q}$,
\begin{equation}
\left| \PP_{M \in \spnq}(\chpo(M) \in I(f;h)) - \frac{q^{h+1}}{q^{2n}}\right| \leq B\frac{2n-h}{q^n},
\end{equation}
where $B$ is an absolute constant (independent of $n, h$ and $q$).
\end{thm}

This result implies equidistribution as long as $(h-n)/\log_q(n)\rightarrow\infty$. 

\begin{remark}\label{remark:Sp_constant_B}
The constant $B$ in Theorem \ref{thm:spqveryshort} can be seen to be the same constant as in Lemma \ref{lem:spq_P_Sp_diamond_bound} below. We give a crude bound for this constant after that lemma, but with additional tools (see Remark \ref{remark:quick_diamond_bound}), it can be seen that one may take $B=2$.
\end{remark}

The proof of this result requires some changes from the proofs in previous sections but is nonetheless similar. For this reason the proofs that follow will be somewhat abbreviated.

For $g \in \MM_{q}$, define the arithmetic functions
\begin{align*}
\beta^{Sp,\diamond}(g) &= \lambda(g)|g|^{1/2} \mathbf{1}_{\MM_{q}^{sr,0}}(g), \\
\gamma^{Sp,\diamond}(g) &=  \mathbf{1}_{\MM_{q}^{sr,1}}(g),
\end{align*} 
and further define
\begin{equation}\label{eq:alpha_SP_2}
\alpha_i^{Sp,\diamond}(g) =  \frac{1}{|g|^i} (\beta^{Sp,\diamond}\ast\gamma^{Sp,\diamond})(g)
\end{equation}
for $i \ge 1$. For $f \in \MM_{q}$, set
\begin{equation}\label{eq:SP_conv_2}
P_{Sp}^{\diamond}(f) = \lim_{k\rightarrow\infty} (\alpha_1^{Sp,\diamond}\ast\alpha_2^{Sp,\diamond}\ast \cdots \ast \alpha_k^{Sp,\diamond}) (f).
\end{equation}

For notational reasons, we let $\sumsp$ denote a sum with all summands restricted to $\MM_q^{sr,1}$. We need the following lemmas. 

\begin{lem}
\label{lem:spq_P_Sp_iterative}
For $f \in \MM_q$,
\begin{equation}
P_{Sp}(f) = 
\frac{1}{|f|^{1/2}} (\mathbf{1}_{\MM_{q}^{sr,1}}\ast  P_{Sp}^\diamond)(f) =\sumsp_{d,\delta:\, d\delta = f} \frac{P_{Sp}^\diamond(\delta)}{|d\delta|^{1/2}}.
\end{equation}
\end{lem}

\begin{proof}
This follows from Theorem~\ref{thm:Sp_conv}, following the same strategy as the proof of Lemma~\ref{lem:unq_P_U_iterative} for $\unq$.
\end{proof}

\begin{lem}
\label{lem:spq_P_Sp_positivity}
For all $f \in \MM_q$, 
\begin{equation}
\lambda(f) P_{Sp}^\diamond(f) \geq 0.
\end{equation}
\end{lem}

\begin{proof}
Using Lemma~\ref{lem:uniqueness} together with $q$-series identities one can verify, in the notation of Theorem~\ref{thm:spnq_charpoly}, that if $f \in \MM_q^{sr,1}$,
	\begin{equation}\label{eq:p sp diamond formula}
	P_{Sp}^{\diamond}(f) = \frac{q^{2a(a-1)}}{|\mathrm{GL}(a,q^2)|} \frac{q^{2b(b-1)}}{|\mathrm{GL}(b,q^2)|}\prod_{i=1}^r \frac{q^{\frac{m_i e_i(e_i-1)}{2}}(-1)^{e_i}}{|\mathrm{U}(e_i, q^{\frac{m_i}{2}})|} \prod_{j=1}^s \frac{q^{m'_j e'_j(e'_j-1)}}{|\mathrm{GL}(e'_j, q^{m'_j})|},
	\end{equation}
and if $f \notin \MM_q^{sr,1}$ then $P_{Sp}^{\diamond}(f) = 0$. The lemma then follows from an examination of this formula.
\end{proof}

Note that we have $P_{Sp}^\diamond(f) = 0$ if $\deg(f)$ is odd, because $\deg(f)$ is even for all $f \in \MM_q^{sr,1}$ (see \S\ref{subsec:back_SP}).

\begin{lem}
\label{lem:spq_P_Sp_diamond_bound}
There exists an absolute constant $B$ such that for all $n\geq 0$ and all prime powers $q$,
\begin{equation}
\sum_{f \in \MM_{2n,q}} |P^\diamond_{Sp}(f)| \leq B.
\end{equation}
\end{lem}

\begin{proof}
We outline a proof using generating series. Define for $i = 0, 1$,
\begin{equation}
Z_{sr,i}(u) = \sum_{f \in \MM_q^{sr,i}} u^{\deg(f)}, \quad\quad Z_{sr,i}^\lambda(u) = \sum_{f \in \MM_q^{sr,i}} \lambda(f) u^{\deg(f)}.
\end{equation}
(These are just the already-defined $L$-functions at trivial characters, and the sums converge absolutely for $|u| < 1/\sqrt{q}$.) Corollary \ref{cor:sr01} implies that all elements of $\MM_q^{sr,1}$ can be written uniquely as $(T-1)^a (T+1)^b g$ for $a, b$ even and $g \in \MM_q^{sr,0}$. This implies that
\begin{equation}
\label{eq:spq_Zsr1_to_Zsr0}
Z_{sr,1}(u) = \frac{1}{(1-u^2)^2}Z_{sr,0}(u),\quad\quad Z_{sr,1}^\lambda(u) = \frac{1}{(1-u^2)^2} Z_{sr,0}^\lambda(u).
\end{equation}

Theorem \ref{thm:Sp_conv} (or Theorem \ref{thm:L_Sp_expansion}) implies
\begin{equation}
\sum_{f \in \MM_q} P_{Sp}(f) u^{\deg(f)} = \prod_{i=1}^\infty Z_{sr,1}\Big(\frac{u}{q^{i-1/2}}\Big) Z_{sr,0}^\lambda\Big(\frac{u}{q^i}\Big),
\end{equation}
while the same reasoning implies
\begin{equation}
\sum_{f \in \MM_q} \lambda(f)P_{Sp}^\diamond(f) u^{\deg(f)} = \prod_{i=1}^\infty Z_{sr,0}\Big(\frac{u}{q^{i-1/2}}\Big) Z_{sr,1}^\lambda\Big(\frac{u}{q^i}\Big).
\end{equation}
Applying \eqref{eq:spq_Zsr1_to_Zsr0} to these identities, we have
\begin{align}
\sum_{f \in \MM_q} \lambda(f)P_{Sp}^\diamond(f) u^{\deg(f)} &= \Big( \prod_{i=1}^\infty \frac{(1-u^2/q^{2i-1})^2}{(1-u^2/q^{2i})^2}\Big) \Big( \sum_{f \in \MM_q} P_{Sp}(f) u^{\deg(f)} \Big) \\
&= \Big( \prod_{i=1}^\infty \frac{(1-u^2/q^{2i-1})^2}{(1-u^2/q^{2i})^2}\Big) \frac{1}{1-u^2},
\end{align}
with the second line following from the fact that $P_{Sp}$ is a probability measure on $\MM_{2n,q}$. For notational reasons we write the product on the last line as $G(u^2,q) := \prod (1-u^2/q^{2i-1})^2/(1-u^2/q^{2i})^2$.

We therefore have
\begin{align}
\sum_{f \in \MM_{2n,q}} |P_{Sp}^\diamond(f)| &= \sum_{f \in \MM_{2n,q}} \lambda(f) P_{Sp}^\diamond(f) \\
&= [u^{2n}]\, G(u^2,q) \frac{1}{1-u^2} \\
&= \frac{1}{2\pi i} \int_{|z|=\rho} G(z,q) \frac{1}{1-z} z^{-(n+1)}\, dz,
\end{align}
for any radius $\rho < 1$. Due to the rapid convergence of the product defining $G$, a little work shows
\begin{equation}
G(z,q) = G(1,q) + O(|z-1|), \quad \textrm{for}\; q \geq 2,\; |z| \leq 1,
\end{equation}
where the implies constant is absolute. Thus for all $\rho < 1$,
\begin{align}
\sum_{f \in \MM_{2n,q}} |P_{Sp}^\diamond(f)| &= \frac{1}{2\pi i}\int_{|z| = \rho} \Big( \frac{G(1,q)}{1-z} + O(1)\Big) z^{-(n+1)}\, dz \\
&= G(1,q) + \frac{1}{2\pi i} \int_{|z|=\rho} O(|z|^{-(n+1)})\, dz.
\end{align}
As $G(1,q) = O(1)$ and $\rho < 1$ is arbitrary, we can let $\rho \rightarrow 1^-$ and see that the above is bounded by an absolute constant as claimed.
\end{proof}

The following two lemmas are proved analogously to Lemmas~\ref{lem:unq_divisorcount1} and \ref{lem:unq_divisorcount2}.
\begin{lem}\label{lem:spnq_divisorcount1}
	Take $n > 0$, $f \in \MM_{2n,q}$, and $2n > h \geq n$. For $\delta \in \MM_{q}^{sr,1}$, if $\deg(\delta) \leq 2h+2-2n$, we have
	\begin{equation}\label{eq:spnqdivisorcount1}
	\sumsp_{d:\, d\delta \in I(f;h)} 1 = \frac{q^{h+1}}{q^{2n}} \sumsp_{d:\, d\delta \in \MM_{2n,q}} 1.
	\end{equation}
	In particular the above expression is constant as $f$ ranges over $\MM_{2n,q}$ for fixed $n$ and $q$. Furthermore the right hand side of \eqref{eq:spnqdivisorcount1} simplifies. In fact for $\delta \in \MM_{q}^{sr,1}$ and $\deg(\delta) \leq 2n$, we have
	\begin{equation}\label{eq:spnqdivisorcount1prime}
	\frac{q^{h+1}}{q^{2n}} \sumsp_{d:\, d\delta \in \MM_{2n,q}} 1 = \frac{q^{h+1-n}}{|\delta|^{1/2}} .
	\end{equation}
\end{lem}

\begin{lem}\label{lem:spnq_divisorcount2}
	For $n > 0$, $2n > h \geq n$ and $\delta \in \MM_{q}^{sr,1}$, if $\deg(\delta) > 2h+2-2n$, then
	\begin{equation}\label{eq:spnq_divisorcount2}
	\Big|\sumsp_{d:\, d\delta \in I(f;h)} 1 \Big|\leq 1
	\end{equation}
	for all $f \in \MM_{2n,q}$.
\end{lem}

\begin{proof}[Proof of Theorem~\ref{thm:spqveryshort}]
With the lemmas above, this proof follows in the same way as Theorems~\ref{thm:glveryshort} and \ref{thm:unveryshort}.
\end{proof}

 From Theorem~\ref{thm:spqveryshort} we may immediately deduce an analogous result for traces using Lemma~\ref{lem:symm}.
 \begin{thm}\label{thm:spmanytraces}
 	Let $M \in \spnq$ be a random matrix chosen according to Haar measure. Then for $1 \le k \le n-1$ and for $\{a_i\}_{1 \le i \le k, \, p \nmid i} \subseteq \FF_{q}$, we have
 	\begin{equation}\label{eq:spmanytraces}
 	\left| \PP_{M \in \spnq}( \forall 1 \le i \le k,\, p \nmid i: \Tr(M^i)=a_i)  - \frac{1}{q^{k-\lfloor \frac{k}{p} \rfloor}}\right| \leq B\frac{(k+1)q^{\lfloor \frac{k}{p} \rfloor}}{q^n},
 	\end{equation}
	where $B$ is an absolute constant (independent of $n, k,$ and $q$).
 \end{thm}

\section{Results for $\onqp$ and $\onqm$}\label{sec:ortho}

In this section, $q$ is an odd prime power. 
\subsection{$\onqp$, $\onqm$ and the space of characteristic polynomials}\label{subsec:back_O}
For any symmetric  matrix $K \in \glnq$, we define the orthogonal group preserving the quadratic form with matrix $K$ to be
\begin{equation}
\mathrm{O}(n,q,K) = \{ g \in \glnq : g^t K g = K \}.
\end{equation}

When $K_1$ and $K_2$ and congruent, $\mathrm{O}(n,q,K_1)$ and $\mathrm{O}(n,q,K_2)$ are conjugate in $\glnq$. Let $\chi_2$ be the non-trivial quadratic character of $\FF_q^{\times}$. For each $n$, there are two congruence classes for invertible symmetric matrices of size $n$: one for symmetric matrices with $\chi_2(\det(K) (-1)^{\lfloor \frac{n}{2} \rfloor})=1$ and another for symmetric matrices with $\chi_2(\det(K) (-1)^{\lfloor \frac{n}{2} \rfloor})=-1$. For $\epsilon \in \{-,+\}$, we set
\begin{equation}
\onqeps = \mathrm{O}(n,q,K_n^{\epsilon})
\end{equation}
where $K_n^{\epsilon}$ is a symmetric matrix with $\chi_2(\det(K^{\epsilon}_n) (-1)^{\lfloor \frac{n}{2} \rfloor})=\epsilon 1$. For odd $n$, given $K_n^+$, one may choose $K_n^-$ to be $K_n^-=c K_n^+$ for a non-square $c \in \FF_q^{\times}$. This shows that $\onqp$ and $\onqm$ are always conjugate in $\glnq$ (for odd $n$). Still, the distinction between the two groups will be useful. For even $n$, $\onqp$ and $\onqm$ are not isomorphic, and in fact have different orders.\footnote{Concrete choices for $K_n^+$ and $K_n^-$ can be found in \cite[\S4.3]{fulman1999} for instance.}
\begin{proposition}\label{prop:O_size}\cite[Ch.~11]{taylor1992}
	We have $|\mathrm{O}(2n+1,q)| = 2q^{n^2} \prod_{i=1}^n(q^{2i} - 1)$, $|\mathrm{O}^{+}(2n,q)| = 2q^{n^2-n}(q^n-1) \prod_{i=1}^{n-1}(q^{2i} - 1)$ and $|\mathrm{O}^{-}(2n,q)| = 2q^{n^2-n}(q^n+1)\prod_{i=1}^{n-1}(q^{2i} - 1)$.
\end{proposition}
By definition, if $M \in \mathrm{O}(n,q,K)$ then $M=K^{-1}(M^t)^{-1} K$, and so $\chpo(M) = \chpo((M^t)^{-1}) = \overline{\chpo(M)}$ where $f\mapsto \overline{f}$ is defined in Definition~\ref{dfn:f_tilde_sp}. Hence $\chpo(M)$ is self-reciprocal, so that $\{ \chpo(M):\, M \in \mathrm{O}(n,q,K)\}\subseteq  \MM_{n,q}^{sr}$. In fact, from Theorem~\ref{thm:onq_charpoly} below, 
\begin{equation}
\cup_{\epsilon\in \{+,-\}}\{ \chpo(M):\, M \in \onqeps\}=  \MM_{n,q}^{sr}.
\end{equation}
\subsection{Expressions for $P^{+}_{O}(f)$ and $P^{-}_{O}(f)$}
\begin{dfn}
	Fix $\epsilon \in \{+,-\}$. For $f \in \MM_q$, we define the arithmetic function
	\begin{equation}\label{eq:P_O_def}
	P^{\epsilon}_{O}(f) = \PP_{M \in \onqeps}(\chpo(M) = f) = \frac{|\{M \in \onqeps:\, \chpo(M) = f\}|}{|\onqeps|}
	\end{equation}
	with $n = \deg(f)$. For $f=1$ we define $P^{+}_{O}(1)=1$, $P^{-}_{O}(1)=0$.
\end{dfn}
Fulman \cite[Thm.~20]{fulman1999} has proved a closed formula for the sum of counts $|\{M \in \onqp:\, \chpo(M) = f\}|+|\{M \in \onqm:\, \chpo(M) = f\}|$. We have found it natural to phrase this formula in the language of probability as above.
\begin{thm}[Fulman]\label{thm:onq_charpoly}
	Suppose $f \in \MM_{q}^{sr}$ factorizes as
	\begin{equation}
	f = (T-1)^{a} (T+1)^{b} \prod_{i=1}^r P_i^{e_i} \prod_{j=1}^s (Q_j \overline{Q_j})^{e'_j},
	\end{equation}
	with $P_i = \overline{P}_i$ for all $i$ and $Q_j \neq \overline{Q_j}$ for all $j$, with $P_i, Q_j \in \MM_{q} \setminus \{T,T-1,T+1\}$ irreducible for all $i,j$. Set 
	\begin{equation}
	m_i = \deg(P_i), \quad m'_j = \deg(Q_j).
	\end{equation} 
	Then
	\begin{equation}
	P^{+}_{O}(f) + P^{-}_{O}(f) = F(a) F(b)\prod_{i=1}^r \frac{q^{\frac{m_i e_i(e_i-1)}{2}}}{|\mathrm{U}(e_i, q^{\frac{m_i}{2}})|} \prod_{j=1}^s \frac{q^{m'_j e'_j(e'_j-1)}}{|\mathrm{GL}(e'_j, q^{m'_j})|},
	\end{equation}
where
\begin{equation}
F(n) := \begin{cases} |\mathrm{Sp}(n,q)|^{-1} q^{\frac{1}{2}n^2} & \mbox{if $n\equiv 0 \bmod 2$,}\\ |\mathrm{Sp}(n-1,q)|^{-1} q^{\frac{1}{2}(n-1)^2} &\mbox{if $n \equiv 1 \bmod 2$.}\end{cases}
\end{equation}
	If $f$ is a monic polynomial but $f \notin \MM_{q}^{sr}$, then $P^{+}_{O}(f) = P^{-}_{O}(f) = 0$.
\end{thm}
We have the following formula for $P^{+}_{O}(f) - P^{-}_{O}(f)$, which does not seem to have appeared in the literature, but which follows from Fulman's methods. We fill in the details of the proof in an appendix.
\begin{thm}\label{thm:onq_charpoly2} 
	In the notation of Theorem~\ref{thm:onq_charpoly} and \eqref{eq:SP_conv_2},
	\begin{equation}
	P^{+}_{O} - P^{-}_{O} =P_{Sp}^{\diamond}.
	\end{equation}
\end{thm}

\begin{remark}\label{remark:quick_diamond_bound}
Theorem \ref{thm:onq_charpoly2} can be used to quickly give an alternative proof of Lemma \ref{lem:spq_P_Sp_diamond_bound} with $B=2$.
\end{remark}

By definition $P^{+}_{O} - P^{-}_{O}=P_{Sp}^{\diamond}$ is an infinite Dirichlet convolution of simple arithmetic functions defined on $\MM_q$. The same is true for $P^{+}_{O} + P^{-}_{O}$ as well.
\begin{thm}\label{thm:O_conv}
	For $g \in \MM_{q}$, define the arithmetic function
	\begin{equation}
	\Delta(g) = \mathbf{1}_{g \mid T^2-1}.
	\end{equation}
	 Then for $f \in \MM_{q}$,
	\begin{align}\label{eq:O_conv}
	P^+_{O} + P^-_{O}&= \Delta \ast P_{Sp}.
\end{align}
\end{thm}
One may verify \eqref{eq:O_conv} directly using Theorems~\ref{thm:spnq_charpoly} and \ref{thm:onq_charpoly}.

We also need the following theorem.
\begin{thm}\label{thm:L_O_expansion}
	For any Hayes character $\chi$, let
	\begin{equation}\label{eq:L_O_def}
	L_{O}(u,\chi) = L_{sr,0}(\sqrt{q}u,\chi \lambda) L_{sr,1}(u,\chi).
	\end{equation}
	For $|u| < 1$ we have 
	\begin{equation}\label{eq:P_O_generating}
\sum_{f \in \MM_{q}} (P^+_{O}(f) - P^-_{O}(f)) \chi(f) u^{\deg(f)} = \prod_{i=1}^\infty L_{O}\Big(\frac{u}{q^{i}},\chi\Big) = L_{sr,0}(\frac{u}{\sqrt{q}},\chi \lambda) \prod_{i=1}^\infty L_{Sp}\Big(\frac{u}{q^{1/2+i}},\chi\Big),
\end{equation}
	with both the left hand sum and the right hand product converging absolutely.
\end{thm}
The proof is a direct consequence of \eqref{eq:SP_conv_2}. In fact, we may show the following.
\begin{proposition}\label{prop:Sp_diamond_iterative}
For all $f \in \MM_q$,
\begin{equation}
P_{Sp}^{\diamond}(f) = \sumsp_{\substack{d,\delta:\, d\delta = f \\ d \in \MM_{q}^{sr,0}}} \frac{\lambda(d) P_{Sp}(\delta)}{|d \delta|^{1/2}} .
\end{equation}
\end{proposition}

\subsection{Character sums over $\onqp$, $\onqm$}
\begin{thm}\label{thm:expoo}
	Let $n$ be a positive integer. Let $\chi$ be a Hayes character of the form $\chi_1$ or $\chi_1\cdot \chi_T$, where $\chi_1$ is a non-trivial short interval character of $\ell$ coefficients and $\chi_T$ is a Dirichlet character modulo $T$.
	For $L_{Sp}(u,\chi)$ defined in \eqref{eq:L_Sp_def}, we have the factorization
	\begin{equation}\label{eq:L_Sp_factor2}
	L_{Sp}(u,\chi) = \prod_{i=1}^{d} (1-\gamma_i u^2),
	\end{equation}
	and for $L_{O}(u,\chi)$ defined in \eqref{eq:L_O_def}, we have the factorization
	\begin{equation}\label{eq:L_O_factor}
	L_{O}(u,\chi) = \prod_{i=1}^{d^\dag} (1-\gamma^\dag_i u^2) 
	\end{equation}
	with
	\begin{equation}\label{eq:L_O_degree_bound}
	d = d^\dag\le 2\ell.
	\end{equation}
	(Note that $\gamma_i$, $\gamma'_j$ will have a dependence on $\chi$.)
	\begin{enumerate}
		\item We have the following identities. For odd $n$,
		\begin{multline}\label{eq:ondesirediden_odd}
		\frac{1}{\left|\onqp\right|}\sum_{M \in \onqp} \chi(\chpo(M)) + \frac{1}{\left|\onqm\right|}\sum_{M \in \onqm} \chi(\chpo(M)) \\
		= (-1)^{\frac{n-1}{2}} \left(\chi(T-1)+\chi(T+1)\right) \sum_{a_1+\ldots+a_d=(n-1)/2}\prod_{j=1}^{d} \frac{\gamma_j^{a_j}}{(q^{2a_j}-1)\cdots (q^2-1)}
		\end{multline}
		and (still for odd $n$)
		\begin{equation}\label{eq:ondesirediden2_odd}
		\frac{1}{\left|\onqp\right|}\sum_{M \in \onqp} \chi(\chpo(M)) - \frac{1}{\left|\onqm\right|}\sum_{M \in \onqm} \chi(\chpo(M)) = 0.
		\end{equation}

		And for even $n$,
		\begin{multline}\label{eq:ondesirediden_even}
		\frac{1}{\left|\onqp\right|}\sum_{M \in \onqp} \chi(\chpo(M)) + \frac{1}{\left|\onqm\right|}\sum_{M \in \onqm} \chi(\chpo(M)) \\
		= (-1)^{\frac{n}{2}}  \left( \sum_{a_1+\ldots+a_d=n/2}\prod_{j=1}^{d} \frac{\gamma_j^{a_j}}{(q^{2a_j}-1)\cdots (q^2-1)} - \chi(T^2-1) \sum_{a_1+\ldots+a_d=(n-2)/2}\prod_{j=1}^{d} \frac{\gamma_j^{a_j}}{(q^{2a_j}-1)\cdots (q^2-1)} \right)
		\end{multline}
		and (still for even $n$)
		\begin{multline}\label{eq:ondesirediden2_even}
		\frac{1}{\left|\onqp\right|}\sum_{M \in \onqp} \chi(\chpo(M)) - \frac{1}{\left|\onqm\right|}\sum_{M \in \onqm} \chi(\chpo(M)) \\
		= (-1)^{\frac{n}{2}}\sum_{a_1+\ldots+a_{d}=n/2} \prod_{j=1}^{d} \frac{{\gamma_j^\dag}^{a_j}}{(q^{2a_j}-1)\cdots (q^2-1)} .
		\end{multline}
		\item The following estimate holds for $\epsilon \in \{\pm\}$:
		\begin{equation}\label{eq:onestchar}
		\left| \frac{1}{\left|\onqeps\right|}\sum_{M \in \onqeps} \chi(\chpo(M))\right| \le 3 q^{-\frac{(n-2)^2}{4d}} q^{\frac{n}{4}}(1+\frac{1}{q^2-1})^{\frac{n}{2}} \binom{\frac{n}{2}+d-1}{\frac{n}{2}}.
		\end{equation}
	\end{enumerate}
\end{thm}
\begin{proof}
Note that all $f \in \MM_q^{sr,1}$ satisfy $f(0)=1$, and so any Dirichlet character modulo $T$ is identically $1$ on $\MM_{q}^{sr,1}$. It follows that $L_{sr,1}(u,\chi) = L_{sr,1}(u,\chi  \chi_T)$ and $L_{sr,0}(u,\chi \lambda) = L_{sr,0}(u,\chi  \chi_T \lambda)$, where $\chi$ is any short interval character. Hence \eqref{eq:L_Sp_factor2} was already proved in Theorem~\ref{thm:exposp}.
\eqref{eq:L_O_factor} and the degree bound \eqref{eq:L_O_degree_bound} follow from Corollaries~\ref{cor:rh lsp 1} and \ref{cor:rh lsp 2}. The equality $d=d^{\dag}$ follows from the fact that $\deg L_{sr,1}(u,\chi) = \deg L_{sr,1}(\sqrt{q}u,\chi)$ and similarly $\deg L_{sr,0}(u,\chi \lambda) = \deg L_{sr,0}(\sqrt{q}u,\chi \lambda)$.
	
	For the exact formulas \eqref{eq:ondesirediden_odd}, \eqref{eq:ondesirediden_even},  note that
	\begin{multline}
	\frac{1}{\left|\onqp\right|} \sum_{M \in \onqp} \chi(\chpo(M)) + \frac{1}{\left|\onqm\right|} \sum_{M \in \onqm} \chi(\chpo(M)) \\
	= \sum_{f \in \MM_{n,q}} (P^+_{O}(f)+P^-_{O}(f)) \chi(f) = [u^n] (1+\chi(T-1)u)(1+\chi(T+1)u)\prod_{i=1}^\infty L_{Sp}\Big(\frac{u}{q^i},\chi\Big)
	\end{multline}
	with the second equality following from \eqref{eq:O_conv} and \eqref{eq:P_Sp_generating}. For the exact formulas \eqref{eq:ondesirediden2_odd}, \eqref{eq:ondesirediden2_even} note that 
	\begin{multline}
\frac{1}{\left|\onqp\right|} \sum_{M \in \onqp} \chi(\chpo(M)) - \frac{1}{\left|\onqm\right|} \sum_{M \in \onqm} \chi(\chpo(M)) \\
= \sum_{f \in \MM_{n,q}} (P^+_{O}(f)-P^-_{O}(f)) \chi(f) = [u^n]\prod_{i=1}^\infty L_{O}\Big(\frac{u}{q^i},\chi\Big)
\end{multline}
	with the second equality following from \eqref{eq:P_O_generating}. Using \eqref{eq:L_Sp_factor2} and \eqref{eq:L_O_factor} we obtain
	\begin{multline}\label{formav3o}
		\frac{1}{\left|\onqp\right|} \sum_{M \in \onqp} \chi(\chpo(M)) + \frac{1}{\left|\onqm\right|} \sum_{M \in \onqm} \chi(\chpo(M)) \\
 = [u^n] (1+\chi(T-1)u)(1+\chi(T+1)u)\prod_{j=1}^{d} \prod_{i = 1}^{\infty} (1-\gamma_j \frac{u^2}{q^{2i}})
	\end{multline}
	and 
	\begin{multline}\label{formav3o2}
\frac{1}{\left|\onqp\right|} \sum_{M \in \onqp} \chi(\chpo(M)) - \frac{1}{\left|\onqm\right|} \sum_{M \in \onqm} \chi(\chpo(M)) \\
= [u^n] \prod_{j=1}^{d} \prod_{i = 1}^{\infty} (1-\gamma^\dag_j \frac{u^2}{q^{2i}}).
\end{multline}	
	Using \eqref{eq:swapped_infinity2} with $y = \gamma_j u^2$ and $V = q^2$, we have
	\begin{equation}\label{eq:eulerappo}
	\prod_{i = 1}^{\infty} (1-\gamma_j \frac{u^2}{q^{2i}}) = \sum_{a \ge 0} \frac{(-1)^a \gamma_j^au^{2a}}{(q^{2a}-1)(q^{2(a-1)}-1)\cdots (q^2-1)},
	\end{equation}
	which together with \eqref{formav3o} and \eqref{formav3o2} gives \eqref{eq:ondesirediden_odd}, \eqref{eq:ondesirediden_even} and \eqref{eq:ondesirediden2_odd}, \eqref{eq:ondesirediden_even}. (Note that \eqref{eq:ondesirediden2_odd} follows more simply: for odd $n$, the groups $\onqp$ and $\onqm$ are conjugate.) 

	We may bound \eqref{eq:ondesirediden_odd}, \eqref{eq:ondesirediden_even} much in the same way as in the proof of Theorem~\ref{thm:expo}, and obtain
	\begin{multline}
	|\frac{1}{\left|\onqp\right|}\sum_{M \in \onqp} \chi(\chpo(M)) + \frac{1}{\left|\onqm\right|}\sum_{M \in \onqm} \chi(\chpo(M)) | \\
	\le 2q^{\frac{n}{4}}(1+\frac{1}{q^2-1})^{\frac{n}{2}} \binom{\frac{n}{2} + d-1}{\frac{n}{2}} q^{-\frac{(n-2)^2}{4d}}.
	\end{multline}
	Bounding \eqref{eq:ondesirediden2_even} is similar, and gives
	\begin{multline}
|\frac{1}{\left|\onqp\right|}\sum_{M \in \onqp} \chi(\chpo(M)) - \frac{1}{\left|\onqm\right|}\sum_{M \in \onqm} \chi(\chpo(M)) | \\
\le q^{\frac{n}{4}}(1+\frac{1}{q^2-1})^{\frac{n}{2}} \binom{\frac{n}{2} + d-1}{\frac{n}{2}} q^{-\frac{n^2}{4d}}.
\end{multline}
Together, we obtain \eqref{eq:onestchar}.
\end{proof}

 \begin{remark}
	In the next subsections we only use  Theorem~\ref{thm:expoo} with $\chi$ a short interval character. We have recorded the theorem for more general characters because it is possible to extend our later results from $\onqeps$ to $\mathrm{SO}^{\epsilon}(n,q)$ using the more general characters.
\end{remark}

\subsection{The distribution of the characteristic polynomial: superexponential bounds for $\onqp$, $\onqm$}

\begin{thm}\label{thm:apshorton}
	Fix $\epsilon \in \{-,+\}$. Let $M \in \onqeps$ be a random matrix chosen according to Haar measure. Then for $0 \leq h < n-1$ and for $f \in \MM_{n,q}$,
	\begin{equation}\label{eq:onshort}
	\left| \PP_{M \in \onqeps} (\chpo(M) \in I(f;h) ) - \frac{q^{h+1}}{q^{n}} \right| \leq 3q^{-\frac{(n-2)^2}{8(n-h-1)}} q^{\frac{n}{4}}(1+\frac{1}{q^2-1})^{\frac{n}{2}} \binom{\frac{5n}{2}-2h-3}{\frac{n}{2}}.
	\end{equation}
\end{thm}
The proof proceeds along the lines of the proof of Theorem~\ref{thm:apshort}.

\subsection{The distribution of traces: superexponential bounds for $\onqp$, $\onqm$}
\begin{thm}\label{thm:arbexp_on}
	Fix $\epsilon \in \{\-,+\}$. Let $M \in \onqeps$ be a random matrix chosen according to Haar measure. Fix a strictly increasing sequence $b_1,\ldots,b_k$ of positive integers coprime to $p$. Then for any sequence $a_1,...,a_k$ of elements of elements from $\FF_{q}$, we have
	\begin{equation}\label{eq:probsboundedon}
	\left|\PP_{M \in \onqeps}(\forall 1 \le i \le k: \Tr(M^{b_i}) = a_i) - q^{-k}\right| \le 3q^{-\frac{(n-2)^2}{8b_k}} q^{\frac{n}{4}} (1+\frac{1}{q^2-1})^{\frac{n}{2}} \binom{\frac{n}{2} + 2b_k-1}{\frac{n}{2}}. 
	\end{equation}
\end{thm}
The proof is omitted, as it follows closely the proof of Theorem~\ref{thm:arbexp}.

\subsection{The characteristic polynomial in very short intervals for $\onqp$, $\onqm$}

As for other groups, an element of $\MM_q^{sr}$ is described by roughly half its coefficients and so there may exist short intervals with no elements of $\MM_q^{sr}$.

From an analysis of the involution $\overline{f}$ of Definition \ref{dfn:f_tilde_sp}, considering odd and even degrees separately, one may prove

\begin{proposition}\label{prop:sronly_count_in_shortint}
For any $f \in \MM_{n,q}$ and $h < n$, we have
\begin{equation}
|\MM_q^{sr} \cap I(f;h)| = 
\begin{cases}
q^{h+1- \lceil n/2 \rceil} + q^{h+1 - \lceil (n+1)/2 \rceil} & \textrm{if}\; h \geq \lceil n/2 \rceil, \\
O(1) & \textrm{if}\; h < \lceil n/2 \rceil.
\end{cases}
\end{equation}
\end{proposition} 
Hence as before we cannot expect equidistribution at all scales. But the same phenomenon we have seen for all other groups persists: as long as $h$ is slightly larger than $n/2$, equidistribution of the characteristic polynomial in intervals $I(f;h)$ occurs for both the groups $\onqp$ and $\onqm$. 

\begin{thm}
\label{thm:overyshort}
Fix $\epsilon \in \{-,+\}$. Let $M \in \onqeps$ be a random matrix chosen according to Haar measure. Then for $n\geq 3$ and $(n+1)/2 \leq h < n$, if $f \in \MM_{n,q}$,
\begin{equation}
\label{eq:Prob_O_veryshort}
\left| \PP_{M \in \onqeps}(\chpo(M) \in I(f;h)) - \frac{q^{h+1}}{q^{n}}\right| \leq C\frac{n-h}{q^{n/2-1}},
\end{equation}
where $C$ is an absolute constant (independent of $n, h$ and $q$).
\end{thm}

This result implies equidistribution as long as $(h-n/2)/\log_q(n) \rightarrow \infty$.

\begin{remark}
More careful book-keeping shows that the argument below allows one to replace the upper bound $C(n-h)/q^{n/2-1}$ by $2(n-h+1)/q^{n/2-1}.$
\end{remark}

Theorem \ref{thm:overyshort} is a consequence of estimates below for sums over short intervals of $P_O^+(f)+P_O^-(f)$ and $P_O^+(f)-P_O^-(f)$. For notational reasons, we let $\sumo$ denote a sum with all summands restricted to $\MM_q^{sr,0}$.

\begin{lem}[An estimate for $P_O^+ + P_O^-$]
\label{lem:O_summed_est}
For $n\geq 3$ and $(n+1)/2 \leq h < n$, if $f \in \MM_{n,q}$,
\begin{equation}
\sum_{g \in I(f;h)} (P_O^+(g) + P_O^-(g)) = 2 \frac{q^{h+1}}{q^n} + O\Big(\frac{n-h}{q^{n/2-1}}\Big),
\end{equation}
where the implied constant is absolute.
\end{lem}

\begin{proof}
Note that from Theorem \ref{thm:O_conv},
\begin{align}
\sum_{g \in I(f;h)} (P_O^+(g) + P_O^-(g)) =& \sum_{g \in I(f;h)} (\Delta\ast P_{Sp})(g) \\
=& \sum_{g \in I(f;h)} P_{Sp}(g) + \sum_{(T+1)g \in I(f;h)} P_{Sp}(g)\\
& + \sum_{(T-1)g \in I(f;h)} P_{Sp}(g) + \sum_{(T^2-1)g \in I(f;h)} P_{Sp}(g).
\end{align}
Because $P_{Sp}$ is supported on even degree polynomials, this simplifies to
\begin{equation}
=\begin{cases}
\sum_{(T+1)g \in I(f;h)} P_{Sp}(g) + \sum_{(T-1)g \in I(f;h)} P_{Sp}(g) & \textrm{for}\;n = \deg(f)\; \textrm{odd},\\
\sum_{g \in I(f;h)} P_{Sp}(g) + \sum_{(T^2-1)g \in I(f;h)} P_{Sp}(g) & \textrm{for}\;n = \deg(f)\; \textrm{even}.
\end{cases}
\end{equation}
But from Theorem \ref{thm:spqveryshort}, considering the case of odd and even $n$ separately, one can verify that if $n$ is odd and $(n+1)/2 \leq h < n$, or if $n$ even and $n/2+1 \leq h < n$, the above is
\begin{equation}
= 2 \frac{q^{h+1}}{q^n} + O\Big(\frac{n-h}{q^{n/2-1}}\Big).
\end{equation}
Inspecting these ranges of $h$ for odd and even $n$ verifies the claim of the lemma.
\end{proof}

\begin{lem}[An estimate for $P_O^+ - P_O^-$]
\label{lem:O_subtract_est}
For $n\geq 3$ and $(n+1)/2 \leq h < n$, if $f \in \MM_{n,q}$,
\begin{equation}
\sum_{g \in I(f;h)} (P_O^+(g) - P_O^-(g)) = O\Big(\frac{n-h}{q^{n/2-1}}\Big),
\end{equation}
where the implied constant is absolute.
\end{lem}

If $n$ is odd this is evident -- in fact the left hand side is $0$ from Theorem \ref{thm:onq_charpoly2}. In the case that $n$ is even we will need the following two preliminary lemmas.

\begin{lem}
\label{lem:sr0_divisorcount1}
	If $n$ is even, for $(n+1)/2 \leq h < n$ and $\delta \in \MM_q^{sr,1}$, if $\deg(\delta) \leq 2h-n-2$, then
	\begin{equation}\label{eq:sr0_divisorcount1}
	\sumo_{d:\, d\delta \in I(f;h)} \lambda(d) = 0
	\end{equation}
	for any $f \in \MM_{n,q}$. In particular the above expression is constant as $f$ ranges over $\MM_{n,q}$, for fixed $n$ and $q$.
\end{lem}

\begin{proof}
Let $\Delta = \deg(\delta)$. For $\Delta \leq 2h-n-2$, we have $n-\Delta \geq 2(n-h-1)+4$. Yet
\begin{equation}
\sumo_{d:\, d\delta \in I(f;h)} \lambda(d) = \frac{1}{q^{n-h-1}} \sum_{\chi \in G(R_{n-h-1,1,q})} \overline{\chi}(f) \sum_{d \in \MM^{sr,0}_{n-\Delta,q}} \lambda(d) \chi(d) \chi(\delta).
\end{equation}
For $\chi \in G(R_{n-h-1,1,q})$, the inner sum above evaluates to $0$ by Corollary \ref{cor:rh lsp 2}.
\end{proof}

\begin{lem}\label{lem:sr0_divisorcount2}
	If $n$ is even, for $n > h \geq n/2$ and $\delta \in \MM_q^{sr,1}$, if $\deg(\delta) \geq 2h-n$, then
	\begin{equation}\label{eq:sr0_divisorcount2}
	\Big|\sumo_{d:\, d\delta \in I(f;h)} 1 \Big|\leq q
	\end{equation}
	for all $f \in \MM_{n,q}$.
\end{lem}

\begin{proof}
As $\MM_q^{sr,0} \subseteq \MM_q^{sr,1}$, we have,
\begin{equation}
\Big|\sumo_{d:\, d\delta \in I(f;h)} 1 \Big|\leq\Big|\sumsp_{d:\, d\delta \in I(f;h)} 1 \Big|\leq q,
\end{equation}
using Lemmas \ref{lem:spnq_divisorcount1} and \ref{lem:spnq_divisorcount2} for the final bound.
\end{proof}

\begin{remark}
If $\deg(\delta) \geq 2h-n+2$, then we can replace $q$ with $1$, but we do not use this.
\end{remark}

\begin{proof}[Proof of Lemma \ref{lem:O_subtract_est}]
We can assume $n$ is even, since otherwise as already noted the left hand side is $0$. Theorem \ref{thm:onq_charpoly2} reduces the lemma to evaluating short interval sums of $P_{Sp}^\diamond$. 

From Proposition \ref{prop:Sp_diamond_iterative}, one has
\begin{equation}
\sum_{g \in I(f;h)} P_{Sp}^\diamond(g) = \sumo_{d,\delta: d\delta \in I(f;h)} \frac{\lambda(d) P_{Sp}(\delta)}{|d\delta|^{1/2}}.
\end{equation}
Using Lemma \ref{lem:sr0_divisorcount1}, it follows that this sum may be restricted to $\delta$ with $2h-n \leq \deg(\delta) \leq n$; the contribution from $\delta$ not in this range sums to $0$. On the other hand, by Lemma \ref{lem:sr0_divisorcount2} and the fact that $P_{Sp}(\delta)$ is a probability measure for $\delta \in \MM_{k,q}$ for even $k$, the magnitude of the above expression is then no more than
\begin{equation}
\sumo_{\substack{d,\delta: d\delta \in I(f;h)\\2h-n \leq \deg(\delta) \leq n}} \frac{P_{Sp}(\delta)}{|d\delta|^{1/2}} \leq \frac{1}{q^{n/2}}q \cdot([n-(2h-n)]/2+1) = \frac{n-h+1}{q^{n/2-1}},
\end{equation}
which yields the claimed bound.
\end{proof}

\begin{proof}[Proof of Theorem \ref{thm:overyshort}]
This is just a matter of differencing the estimates in Lemmas \ref{lem:O_summed_est} and \ref{lem:O_subtract_est}.
\end{proof}

 From Theorem~\ref{thm:overyshort} we may immediately deduce an analogous result for traces using Lemma~\ref{lem:symm}.
 \begin{thm}\label{thm:omanytraces}
 	Fix $\epsilon \in \{-,+\}$. Let $M \in \onqeps$ be a random matrix chosen according to Haar measure. Then for $n\geq 3$, and $1 \le k \le (n-1)/2$ and for $\{a_i\}_{1 \le i \le k, \, p \nmid i} \subseteq \FF_{q}$, we have
 	\begin{equation}\label{eq:omanytraces}
 	\left| \PP_{M \in \onqeps}( \forall 1 \le i \le k,\, p \nmid i: \Tr(M^i)=a_i)  - \frac{1}{q^{k-\lfloor \frac{k}{p} \rfloor}}\right| \leq C\frac{(k+1)q^{\lfloor \frac{k}{p} \rfloor}}{q^{n/2-1}},
 	\end{equation}
	where $C$ is an absolute constant (independent of $n, k,$ and $q$).
 \end{thm}

\appendix
\section{Cycle indices for $\onqp$ and $\onqm$}
We fill in the proof of Theorem~\ref{thm:onq_charpoly2}. It is largely based on Fulman's paper \cite[\S4.3,\, \S6.1]{fulman1999}, and we refer the reader to this source for further details.

\subsection{Parametrization of conjugacy classes}
We describe a parametrization of the conjugacy classes in $\onqp$ and $\onqm$, due to Wall \cite[pp.~38-40]{wall1963}. Here we think of $\onqp$, $\onqm$ as disjoint, abstract groups. Denote by $\mathcal{P}_q \subseteq \MM_q$ the subset of monic irreducible polynomials, not including $T$, and by $\YY$ the set of partitions, which includes the empty partition. 

An ``orthogonal signed partition'' $\lambdao$ is a partition of some natural number $|\lambdao|$ such that the even parts have even multiplicity, together with a choice of sign for the set of parts of size $i$ for each odd $i$. We denote by $\YYO$ the set of all orthogonal signed partitions. There is a certain way to associate with each $g \in \onqp \cup \onqm$ a pair of functions $\lambda_{P}(g)\colon  \mathcal{P}_{q} \setminus \{T \pm 1\} \to \YY$, $\lambdao_{P}(g)\colon  \{ T\pm 1\} \to \YYO$ which enjoys several properties:
\begin{itemize}
	\item $g_1,g_2 \in \onqeps$ are conjugate in $\onqeps$ if and only if $\lambda_P(g_1)=\lambda_P(g_2)$ for all $P \in \mathcal{P}_q \setminus \{T\pm 1\}$ and $\lambda_{P}^{\pm }(g_1) = \lambda_{P}^{\pm }(g_2)$ for all $P \in \{T \pm 1\}$ ($\epsilon \in \{+, - \}$). 
	\item The pairs of functions $(\lambda_P(g), \lambdao_P(g))$, $(\lambda_P(h), \lambdao_P(h))$ are distinct if $g \in \onqp$ while $h \in \onqm$.
	\item $|\lambda_{P}| = |\lambda_{\overline{P}}|$ and $\sum_{P \in \{T \pm 1\}} |\lambda_P^{\pm}| \deg(P)+ \sum_{P \in \mathcal{P}_q \setminus \{ T\pm 1\} } |\lambda_P| \deg(P)=n$.
	\item The characteristic polynomial of $g \in \onqp \cup \onqm$ is
\begin{equation}\label{eq:charpolo}
\chpo(g) =  \prod_{P \in \{ T \pm 1\}} P^{|\lambdao_{P}(g)|} \prod_{P \in \mathcal{P}_{q} \setminus \{T \pm 1\}} P^{|\lambda_{P}(g)|}.
\end{equation}
\end{itemize}
\subsection{Cycles indices for $\onqp$, $\onqm$}
\begin{dfn}
The cycle indices for $\onqp,\onqm$ are the following polynomials in variables $\mathbf{x} = \{ x_{P,\lambda} \}_{P \in \mathcal{P}_{q}\setminus \{ T\pm 1\}, \lambda \in \YY} \cup \{ x_{P,\lambdao} \}_{P \in \{ T\pm 1\}, \lambda \in \YYO}$:
\begin{equation}\label{eq:singlecycleo}
\begin{split}
Z(\onqp, \mathbf{x}) &=\frac{1}{\left|\onqp\right|}\sum_{g \in \onqp} \prod_{P \in \{T\pm 1\}} x_{P, \lambda_{P}^{\pm}(g)} \prod_{P \in \mathcal{P}_{q} \setminus \{ T\pm 1\}} x_{P,\lambda_{P}(g)},\\
Z(\onqm, \mathbf{x}) &=\frac{1}{\left|\onqm \right|}\sum_{g \in \onqm} \prod_{P \in \{T\pm 1\}} x_{P, \lambda_{P}^{\pm}(g)} \prod_{P \in \mathcal{P}_{q} \setminus \{ T\pm 1\}} x_{P,\lambda_{P}(g)}.
\end{split}
\end{equation}
\end{dfn}
Fulman has proved the following identity of formal power series \cite[Thm.~14]{fulman1999}
\begin{multline}\label{eq:cycleo}
1+\sum_{n \ge 1} \left( Z(\onqp, \mathbf{x}) + Z(\onqm, \mathbf{x}) \right) u^n\\
 =  \Big( \sum_{\lambdao \in \YYO} x_{T-1,\lambdao} \frac{u^{|\lambdao|}}{c_{\mathrm{O},q}(\lambdao)} \Big) \Big( \sum_{\lambdao \in \YYO} x_{T+1,\lambdao} \frac{u^{|\lambdao|}}{c_{\mathrm{O},q}(\lambdao)} \Big)\\
 \prod_{P = \overline{P}} \Big( \sum_{\lambda \in \YY} x_{P,\lambda} \frac{(-u^{\deg (P)})^{|\lambda|}}{c_{GL,-q^{\deg(P)/2}}(\lambda)} \Big)  \prod_{Q \neq \overline{Q}} \Big( \sum_{\lambda \in \YY} x_{Q,\lambda} x_{\overline{Q},\lambda} \frac{u^{|\lambda|\deg(Q \overline{Q})}}{c_{GL,q^{\deg(Q)}}(\lambda)} \Big),
\end{multline}
where  $c_{GL,z}(\lambda)$ is a rational function satisfying the following identity due to Stong \cite[Lem.~6,~Prop.~19]{stong1988}
\begin{equation}\label{eq:stongiden}
\sum_{\lambda \in \YY} \frac{u^{|\lambda|}}{c_{GL,z}(\lambda)} = \sum_{j \ge 0}\frac{z^{j(j-1)}}{\prod_{i=0}^{j-1}(z^{j}-z^{i})} u^j
\end{equation}
for $|u| \le 1<|z|$, and $c_{\mathrm{O},z}(\lambdao)$ is also a rational function in $z$.
 
\subsection{The Witt group of $\FF_q$, $q$ odd}
Let $q$ be an odd prime power, and let $W(\FF_q)$ be the Witt group of $\FF_q$, considered as an additive group (we ignore the multiplicative structure). For a non-degenerate quadratic form $Q$ over $\FF_q$, we denote its equivalence class in $W(\FF_q)$ by $[Q]$. It is useful to extend the notation $[\cdot]$ as follows. For a symmetric matrix $K \in \glnq$, we write $[\mathrm{O}(n,q,K)]$ for the equivalence class in $W(\FF_q)$ of the quadratic form associated with $K$.

Fixing a non-square element $c \in \FF_q^{\times}$, the elements of the Witt group $W(\FF_q)$ may be described as follows:
\begin{equation}
W(\FF_q) = \{ [0],  [x^2],  [cx^2] , [x^2-cy^2]\}.
\end{equation}
If $q \equiv 1 \bmod 4$, we identify $W(\FF_q)$ with the group $(\ZZ/2\ZZ)^2$ as follows:
\begin{equation}
[0] = (0,0), \, [x^2] = (1,0), \, [cx^2] = (1,1), \, [x^2-cy^2] = (0,1).
\end{equation}
One can verify from \cite[Prop.~1]{fulman1999} that
\begin{equation}
[\mathrm{O}^{\epsilon}(n,q)] = (n \bmod 2, \epsilon')
\end{equation}
for all $n$ and $\epsilon \in \{\pm\}$, where $\epsilon' = 0$ if $\epsilon = +$ and $\epsilon' = 1$ if $\epsilon=-$.

If $q \equiv 3 \bmod 4$, we identify $W(\FF_q)$ with the group $\ZZ/4\ZZ$ as follows:
\begin{equation}
[0] = 0, \, [x^2] = 1, \, [cx^2] = 3, \, [x^2-cy^2] = 2.
\end{equation}
One can verify from \cite[Prop.~1]{fulman1999} that
\begin{equation}
[\mathrm{O}^{\epsilon}(n,q)] = (n \bmod 2) + 2\epsilon'
\end{equation}
for all $n$ and $\epsilon \in \{\pm\}$, where $(n \bmod 2)$ is the unique integer in $\{0,1\}$ congruent to $n$ modulo $2$.

If $g \in \onqeps$ for some $\epsilon$ and $n$, we set $[g]:=[\onqeps]$.

\subsection{Cycles indices twisted by characters of $W(\FF_q)$}
We give a generalization of \eqref{eq:cycleo} which requires some notation. If $i$ is an odd integer appearing in $\lambdao$, let us denote by $\epsilon_i(\lambdao)$ its sign and by $m_i(\lambdao)$ the multiplicity in which $i$ appears. As shown by Wall \cite[pp.~38-40]{wall1963} (cf. \cite[Thm.~13]{fulman1999}), given $g \in \mathrm{O}^{\epsilon}(n,q)$ for some $\epsilon \in \{\pm \}$, the precise orthogonal group to which $g$ belongs may be determined by the data $\lambda_{T\pm 1}^{\pm}(g)$, $\lambda_{P}(g)$ as follows: 
 \begin{equation}
[g] = \sum_{\substack{P \in \{T-1,T+1\}\\i \text{ odd}}} [\mathrm{O}^{\epsilon_i(\lambdao_{P}(g))}(m_i(\lambdao_{P}(g)))] + [x^2-cy^2] \sum_{P \in \mathcal{P}_q \setminus \{T-1,T+1\}}  |\lambda_P(g)|.
 \end{equation}
 Let $\psi\colon W(\FF_q) \to \CC^{\times}$ be a multiplicative character, and set $\psi(\lambdao) = \psi\left( \sum_{i \text{ odd}} [\mathrm{O}^{\epsilon_i(\lambdao(g))}(m_i(\lambdao(g)))]  \right)$. The method of proof of \cite[Thm.~14]{fulman1999} in fact shows
\begin{multline}\label{eq:cycleo2}
1+\sum_{n \ge 1} \left( \psi\left([\onqp]\right)Z(\onqp, \mathbf{x}) + \psi\left([\onqm]\right)Z(\onqm, \mathbf{x}) \right) u^n \\
= \Big( \sum_{\lambdao \in \YYO} x_{T -1,\lambdao} \psi(\lambdao)  \frac{u^{|\lambdao|}}{c_{\mathrm{O},q}(\lambdao)} \Big)\Big( \sum_{\lambdao \in \YYO} x_{T + 1,\lambdao} \psi(\lambdao)  \frac{u^{|\lambdao|}}{c_{\mathrm{O},q}(\lambdao)} \Big)\\
 \prod_{P = \overline{P} } \Big( \sum_{\lambda \in \YY} x_{P,\lambda} \psi\left( [x^2-cy^2] \right)^{|\lambda|} \frac{u^{\deg (P)|\lambda|}}{c_{GL,-q^{\deg(P)/2}}(\lambda)} \Big)    \prod_{Q \neq \overline{Q}} \Big( \sum_{\lambda \in \YY} x_{Q,\lambda} x_{\overline{Q},\lambda} \frac{u^{|\lambda|\deg(Q \overline{Q})}}{c_{GL,q^{\deg(Q)}}(\lambda)} \Big).
\end{multline}
For $\psi \equiv 1$ we recover \eqref{eq:cycleo}. We leave the details to the reader.
\subsection{Proof of Theorem~\ref{thm:onq_charpoly2}}
Let $\chi$ be a Hayes character, and let us specialize \eqref{eq:cycleo2} to $x_{T\pm 1,\lambdao} = \chi(T\pm 1)^{|\lambdao|}$ and $x_{P,\lambda} = \chi(P^{|\lambda|})$ for all $P \in \mathcal{P}_q \setminus \{T-1,T+1\}$ and $\lambda \in \YY$. We work with formal power series, and so do not check convergence. Additionally we take $\psi(a,b) = (-1)^b$ if $q \equiv 1 \bmod 4$ and $\psi(a) = i^a$ if $q \equiv 3 \bmod 4$. 

We obtain, by appealing to \eqref{eq:stongiden} with $(u,z)=(\chi(P)u^{\deg(P)},-q^{\deg(P)/2})$ and $(u,z)=(\chi(Q\overline{Q})u^{\deg(Q\overline{Q})},q^{\deg(Q)})$, and using the formulas for the orders of $\glnq$ and $\unq$,
\begin{multline}\label{eq:pmp specialization}
1 + \sum_{n \ge 1,\, 2\mid n} \left(\frac{\sum_{M \in \onqp}\chi(\chpo(M))}{|\onqp|} -  \frac{\sum_{M \in \onqm}\chi(\chpo(M))}{|\onqm|}\right) u^n \\
= \Big( \sum_{\lambdao \in \YYO}  \psi(\lambdao)  \frac{(\chi(T-1)u)^{|\lambdao|}}{c_{\mathrm{O},q}(\lambdao)} \Big)\Big( \sum_{\lambdao \in \YYO}  \psi(\lambdao)  \frac{(\chi(T-1)u)^{|\lambdao|}}{c_{\mathrm{O},q}(\lambdao)} \Big)\\
 \prod_{P = \overline{P}} \left( \sum_{j \ge 0} \frac{|P|^{\frac{j(j-1)}{2}}(-1)^j}{|\mathrm{U}(j,q^{\frac{\deg(P)}{2}})|} (\chi(P)u^{\deg(P)})^j \right) 
 \prod_{Q \neq \overline{Q}}\left( \sum_{j \ge 0}\frac{q^{\deg(Q)j(j-1)}}{|\mathrm{GL}(j,q^{\deg(Q)})|} (\chi(Q\overline{Q})u^{\deg(Q\overline{Q})})^j \right).
\end{multline}
Here we have used the fact that for odd $n$, $\onqp$ and $\onqp$ are conjugate and $\psi\left([\onqp]\right)=-\psi\left([\onqm]\right)$, so we may omit the odd $n$ from the left hand side of \eqref{eq:pmp specialization}. By comparing coefficients and using Lemma~\ref{lem:uniqueness}, we find that $\alpha:=P^{+}_{O} - P^{-}_{O}$ is supported on $\MM_{q}^{sr}$, and if $f \in \MM_{q}^{sr}$ factorizes as
	\begin{equation}
	f = (T-1)^{a} (T+1)^{b} \prod_{i=1}^r P_i^{e_i} \prod_{j=1}^s (Q_j \overline{Q_j})^{e'_j}
	\end{equation}
	with the same notation as in Theorem~\ref{thm:sr_factorization}, then
	\begin{equation}\label{eq:pmp mult}
	\alpha(f) = \alpha((T-1)^a) \alpha((T+1)^b)\prod_{i=1}^r \alpha(P_i^{e_i}) \prod_{j=1}^s \alpha((Q_j \overline{Q_j})^{e'_j})
	\end{equation}
	with 
\begin{equation}\label{eq:pmp pe}
\alpha(P_i^{e_i}) = \frac{q^{\frac{m_i e_i(e_i-1)}{2}}(-1)^{e_i}}{|\mathrm{U}(e_i, q^{\frac{m_i}{2}})|}, \quad \alpha((Q_j\overline{Q_j})^{e'_j}) = \frac{q^{m'_j e'_j(e'_j-1)}}{|\mathrm{GL}(e'_j, q^{m'_j})|}, \quad \alpha((T-1)^a)=\alpha((T+1)^a).
\end{equation}
By \eqref{eq:p sp diamond formula}, $P_{Sp}^{\diamond}$ is also supported on $\MM_{q}^{sr}$ and satisfies \eqref{eq:pmp mult} and \eqref{eq:pmp pe} with $P_{Sp}^{\diamond}$ in place of $\alpha$. Additionally, the expectation of $\alpha$ over $\MM_{n, q}$ ($n \ge 1$) is $0$, as well as the expectation of $P_{Sp}^{\diamond}$ (as follows e.g. from Lemma \ref{lem:spq_P_Sp_iterative}). By induction on $n$ one may show that $\alpha((T-1)^n) = P_{Sp}^{\diamond}((T-1)^n)$, which concludes the proof. \qed

\bibliographystyle{alpha}
\bibliography{references}

\Addresses

\end{document}